\documentclass[11pt]{amsart}

\pdfoutput=1
\usepackage[utf8]{inputenc}
\usepackage{array}
\usepackage{amsmath, amsthm, amssymb,xcolor}
\usepackage{enumerate, comment}
\usepackage{dsfont}%for making characteristic function

\usepackage[colorlinks]{hyperref}
\hypersetup{
    colorlinks=true,%change to true to get colors
    linkcolor=blue,
    filecolor=magenta,
    urlcolor=gray,
    citecolor=gray,
}
\allowdisplaybreaks

\usepackage{todonotes}
\usepackage{titlesec}
\usepackage{multicol}
\usepackage{tikz}
\usepackage[capitalise]{cleveref}
\usepackage{standalone}
\usepackage{mathtools}
\usepackage{mathdots}  % for figure 5.
\usepackage{adjustbox}
\usepackage[normalem]{ulem}
\usepackage{enumitem}
\usepackage{subcaption}
\usepackage[backend=biber, style=alphabetic, sorting=nyt, maxnames=100,backref=true]{biblatex}
\renewbibmacro{in:}{}

\renewbibmacro*{volume+number+eid}{%
	\printfield{volume}%
	\setunit*{\addnbspace}% NEW (optional); there's also \addnbthinspace
	\printfield{number}%
	\setunit{\addcomma\space}%
	\printfield{eid}}
\DeclareFieldFormat[article, misc, book, unpublished, thesis]{year}{\mkbiparens{#1}}
\DeclareFieldFormat[article]{volume}{\text{#1}\space}
\DeclareFieldFormat[article]{number}{\mkbibparens{#1}}

\DeclareFieldFormat{journaltitle}{#1,}
\DeclareFieldFormat[thesis]{title}{\mkbibemph{#1}\addperiod}
\DeclareFieldFormat[article, unpublished, misc]{title}{\mkbibemph{#1},}
\DeclareFieldFormat[book]{title}{\mkbibemph{#1}\addperiod}
\DeclareFieldFormat[unpublished]{howpublished}{#1, }

\DeclareFieldFormat{pages}{#1}

\DeclareFieldFormat[article]{series}{#1\addcomma}

\usepackage[left=1in,right=1in,top=1in,bottom=1in,bindingoffset=0cm]{geometry}
\usepackage{mdframed}
\newtheoremstyle{Claim}% name of the style to be used
  {10pt}% measure of space to leave above the theorem. E.g.: 3pt
  {}% measure of space to leave below the theorem. E.g.: 3pt
  {}% name of font to use in the body of the theorem
  {}% measure of space to indent
  {\itshape}% name of head font
  {.}% punctuation between head and body
  { }% space after theorem head; " " = normal interword space
  {\thmname{#1}\thmnumber{ #2}\textnormal{\thmnote{ (#3)}}}

\newtheorem{theorem}{Theorem}[section]
\newtheorem{lemma}[theorem]{Lemma}
\newtheorem{corollary}[theorem]{Corollary}

\newtheorem{observation}[theorem]{Observation}

\newtheorem*{theorem*}{Theorem}

\theoremstyle{Claim}
\newtheorem{claim}{Claim}[theorem]

\theoremstyle{definition}
\newtheorem{definition}[theorem]{Definition}

\newtheoremstyle{appendixA}% name of the style to be used
  {\topsep}% measure of space to leave above the theorem. E.g.: 3pt
  {\topsep}% measure of space to leave below the theorem. E.g.: 3pt
  {\itshape}% name of font to use in the body of the theorem
  {0pt}% measure of space to indent
  {\bfseries}% name of head font
  {. ---}% punctuation between head and body
  { }% space after theorem head; " " = normal interword space
  {\thmname{#1}\thmnumber{ #2}\textnormal{\thmnote{ A.(#3)}}}
\theoremstyle{appendixA}

\newtheoremstyle{appendixA}% name of the style to be used
  {\topsep}% measure of space to leave above the theorem. E.g.: 3pt
  {\topsep}% measure of space to leave below the theorem. E.g.: 3pt
  {\itshape}% name of font to use in the body of the theorem
  {0pt}% measure of space to indent
  {\bfseries}% name of head font
  { }% punctuation between head and body
  {}% space after theorem head; " " = normal interword space
  {\thmname{#1} B.\thmnumber{ #2}\,\textnormal{\thmnote{(#3)}}}
\theoremstyle{appendixA}

\newcommand*{\myproofname}{Proof}
\newenvironment{claimproof}[1][\myproofname]{\begin{proof}[#1]}{\end{proof}}

\newcommand{\defeq}{\coloneqq}
\newcommand{\PCF}{\mathsf{pcf}}

\newcommand{\N}{\mathbb{N}}
\newcommand{\defn}[1]{{\emph{#1}}}
\newcommand{\odd}{\chi_{\mathsf{o}}}
\newcommand{\pcf}{\chi_{\mathsf{pcf}}}

\newcommand{\vp}{\varphi}
\newcommand{\dom}{\mathsf{Dom}}

\newcommand{\oc}{\vp_\mathsf{o}}
\newcommand{\worst}{worst}
\newcommand{\Worst}{\mathsf{worst}}
\newcommand{\Good}{\mathsf{good}}
\newcommand{\Even}{\mathsf{even}}
\newcommand{\Odd}{\mathsf{odd}}
\newcommand{\Bad}{\mathsf{bad}}
\newcommand{\Semibad}{\mathsf{s\text{-}bad}}

\newcommand{\rr}{\eta}

\renewcommand{\phi}{\varphi}

\usepackage{color}

\DeclareMathOperator{\Forb}{Forb}
\DeclareMathOperator{\Flex}{Flex}
\DeclareMathOperator{\flex}{flex}

\makeatletter
\renewcommand*{\@fnsymbol}[1]{\ensuremath{\ifcase#1\or 1\or 2 \or 3 \or 4 \or 5 \or 6 \or 7 \or 8 \or 9 \or 10 
   \else\@ctrerr\fi}}
\makeatother

\DeclareMathOperator{\mad}{mad}

\DeclareMathOperator{\forb}{forb}

\usepackage[foot]{amsaddr}

\title[Odd and PCF colorings with girth conditions]{The forb-flex method for odd coloring and proper conflict-free coloring of planar graphs}

\author{James Anderson}
\address[James Anderson]{Georgia Institute of Technology.~\textnormal{james.anderson@math.gatech.edu}}

\author{Herman Chau}
\address[Herman Chau]{University of Washington.~\textnormal{hchau@uw.edu}}

\author{Eun-Kyung Cho}
\address[Eun-Kyung Cho]{Hankuk University of Foreign Studies.~\textnormal{ ekcho2020@gmail.com}}

\author{Nicholas Crawford}
\address[Nicholas Crawford]{University of Colorado Denver.~\textnormal{nicholas.2.crawford@ucdenver.edu}}

\author{Stephen G. Hartke}
\address[Stephen G. Hartke]{University of Colorado Denver.~\textnormal{stephen.hartke@ucdenver.edu}}

\author{Emily Heath}
\address[Emily Heath]{Iowa State University.~\textnormal{ eheath@iastate.edu}}

\author{Owen Henderschedt}
\address[Owen Henderschedt]{Auburn University.~\textnormal{ olh0011@auburn.edu}}

\author{Hyemin Kwon}
\address[Hyemin Kwon]{Ajou University.~\textnormal{ khmin1121@ajou.ac.kr}}

\author{Zhiyuan Zhang}
\address[Zhiyuan Zhang]{Toronto Metropolitan University.~\textnormal{ zhiyuan.zhang@torontomu.ca}}

\titleformat{\section}[block]{\large\scshape\center}{\thesection.}{1ex}{}
\titleformat{\subsection}[block]{\bfseries}{\thesubsection.}{1ex}{}
\titleformat{\subsubsection}[runin]{\itshape}{}{0em}{}[]

\titlespacing*{\section}{0pt}{*3}{*1}%{indent}{space above}{space below}
\titlespacing*{\subsection}{0pt}{*3}{*1}
\titlespacing*{\subsubsection}{0pt}{*1.5}{*0}

\begin{document}
\maketitle

\begin{abstract}
We introduce a new tool useful for greedy coloring, which we call the forb-flex method, and apply it to odd coloring and proper conflict-free coloring of planar graphs.
The odd chromatic number, denoted $\chi_{\mathsf{o}}(G)$, is the smallest number of colors needed to properly color $G$ such that every non-isolated vertex of $G$ has a color appearing an odd number of times in its neighborhood. The proper conflict-free chromatic number, denoted $\chi_{\mathsf{PCF}}(G)$, is the smallest number of colors needed to properly color $G$ such that every non-isolated vertex of $G$ has a color appearing uniquely in its neighborhood.
 Our new tool works by carefully counting the structures in the neighborhood of a vertex and determining if a neighbor of a vertex can be recolored at the end of a greedy coloring process to avoid conflicts. Combining this with the discharging method allows us to prove $\chi_{\mathsf{PCF}}(G) \leq 4$ for planar graphs of girth at least 11, and $\chi_{\mathsf{o}}(G) \leq 4$ for planar graphs of girth at least 10. These results improve upon the recent works of Cho, Choi, Kwon, and Park.
\end{abstract}

\section{Introduction}
Introduced recently by Petru\v{s}evski and \v{S}krekovski~\cite{petrusevski2021colorings}, an \emph{odd coloring} of a graph $G$ is a proper coloring of $G$ such that every non-isolated vertex has a color which appears an odd number of times in its neighborhood. 
The \emph{odd chromatic number} of $G$, denoted by $\chi_{\mathsf{o}}(G)$, is the minimum number of colors needed for an odd coloring of $G$. In their paper introducing odd coloring, Petru\v{s}evski and \v{S}krekovski~\cite{petrusevski2021colorings} proved $\odd(G)\leq 9$ for any planar graph $G$, and conjectured $\odd(G)\leq 5$ (such a bound would be tight, as witnessed by $\odd(C_5)=5$). Since then, odd coloring has received much attention (see for instance \cite{ahn2022proper, caro2022remarks,cranston2022odd,cranston2023note,liu20231,petr2022odd,qi2022odd,wang2022odd}). 
For example, Cranston~\cite{cranston2022odd} considered odd colorings of graphs which are sufficiently sparse, as measured by their \emph{maximum average degree}:
\[\mad(G) \defeq \max \left\{\frac{2|E(H)|}{|V(H)|} : H \subseteq G, \, H \neq \emptyset\right\}.\]
He conjectured that every graph $G$ with $\mad(G)\leq \frac{4c-4}{c+1}$ has an odd $c$-coloring for $c\geq 4$ and confirmed this for $c\in\{5,6\}$. Cho, Choi, Kwon, and Park~\cite{cho2023odd} resolved this conjecture, showing that it holds for $c\geq 7$ but is false for infinitely many graphs for $c=4$. In addition, they proved that any graph $G$ with $\mad(G)<\frac{22}{9}$ and no induced $5$-cycle has $\odd(G) \le 4$, thereby showing if $G$ is a planar graph with girth at least $11$, then $\odd(G) \le 4$.

We improve upon this result by proving the following: 

\begin{theorem}\label{thm:odd10}
    If $G$ is a planar graph with girth at least $10$, then $\odd(G) \le 4$.
\end{theorem}

In addition, we prove a stronger result for PCF (proper conflict-free) coloring, from which odd coloring arose as a generalization. A \emph{PCF coloring} of a graph is a proper coloring such that the neighborhood of each non-isolated vertex contains a color that appears exactly once. 
The PCF chromatic number of $G$, denoted $\chi_{\PCF}(G)$, is the smallest number of colors used in a PCF coloring of $G$. Recently introduced by Fabrici, Lu\v{z}ar, Rindo\v{s}ov\'a, and Sot\'ak \cite{fabrici2022proper}, PCF coloring has already garnered much attention. For instance, see \cite{ahn2022proper, caro2023remarks, cho2023brooks, cho2022proper, 
cho2022relaxation, liu2024proper, dai2023new, kamyczura2024conflict, liu2024asymptotically}. In particular, as PCF coloring lies between a proper coloring of $G$ and its square, $\chi_{\PCF}(G)$ provides lower bounds on $\chi(G^2)$, a quantity of great interest itself (see, for instance, the discussion in \cite{liu2024proper}). 

In the same paper where they introduced PCF coloring, Fabrici, Lu\v{z}ar, Rindo\v{s}ov\'a, and Sot\'ak \cite{fabrici2022proper} showed $\chi_{\PCF}(G) \leq 8$ and conjectured $\chi_{\PCF}(G) \leq 6$ for any planar graph $G$. Caro, Petru\v{s}evski and \v{S}krekovski~\cite{caro2023remarks} made progress on this conjecture, showing $\chi_{\PCF}(G) \leq 6$ for planar graphs with girth at least 7. Specifically, they provided sufficient bounds on $\mad(G)$ which guarantees $\chi_{\PCF}(G) \leq c$ for $c \in \{4, 5, 6\}$. These bounds were recently improved by Cho, Choi, Kwon, and Park~\cite{cho2022proper}, who showed, among other results, that if $G$ has $\mad(G)<\frac{12}{5}$ and no induced $5$-cycle, then $\pcf(G) \le 4$. Since a planar graph $G$ with girth at least $g$ satisfies $\mad(G) <\frac{2g}{g-2}$, their result implies that any planar graph $G$  with girth at least 12 has $\pcf(G)\leq 4$.

In this paper, we improve upon the above result by showing the following:

\begin{theorem}\label{thm:pcf11}
    If $G$ is a planar graph with girth at least $11$, then $\pcf(G) \le 4$.
\end{theorem}

We note that our proofs do not follow from results about sparse graphs; only the girth conditions on $G$ are used.
Our proofs use the method of discharging. Generally speaking, to prove a theorem using discharging, we start with a counterexample $G$ and assign a real number to the vertices and faces  such that the sum of these charges is negative; then, after applying a series of \textit{discharging rules}---that is, rules which reassign charge yet preserve the total charge --- we show the total charge is positive, thereby reaching a contradiction. To assist in analyzing the discharging, we show in \S \ref{sec: reducible configs} that the minimum counterexamples do not contain various structures, which we call \textit{reducible configurations}.

In \S \ref{sec: forb/flex}, we introduce the key technique used in this paper, a tool that aids in greedy-coloring, which we call the \textit{forb-flex} method. 

 Roughly speaking, this method works as follows: letting $u \in V(G)$, where $G$ is a minimum counterexample to, say, Theorem \ref{thm:odd10}, we assign a \textit{score} to each neighbor $v$ of $u$, and let $\forb(u)$ be the sum of these scores. The score of each neighbor $v$ of $u$ is based on the number of colors that we would forbid from coloring $u$ if we wanted to greedily color $v$ followed by $u$. We simultaneously investigate colors for $u$ that allow a neighbor of $u$ to receive two different colors --- we call these flexible colors. The idea is that if $u$ has at least one flexible color that's not forbidden for $u$, then we can greedily color vertices around $u$ in a way such that one of $u$'s neighbors can be recolored if needed. This extra flexibility allows us to greedily color the neighbors of $u$ without us worrying about $u$ having an odd color --- we can simply utilize the flexibility and change the color on the corresponding neighbor of $u$ to give $u$ an odd color if needed.

Analyzing the $\forb$ of vertices allows us to consider many structural cases at once, thereby reducing casework significantly. In particular, we show configurations that contain a vertex with low $\forb$ are reducible, while configurations with vertices with high $\forb$ retain sufficient charge for the discharging argument to go through. 
We note that while this tool is used closely in conjunction with our discharging rules, its use is not necessarily limited to discharging arguments; indeed, \S \ref{sec: forb/flex} is free of discharging, and one need not have any familiarity with discharging to understand the results in \S \ref{sec: forb/flex}.

We now outline the rest of the paper.
In \S \ref{sec: prelims}, we introduce definitions and notation and prove a few basic lemmas. We then focus on \cref{thm:odd10}. In \S \ref{sec: reducible configs}, we show reducible configurations that cannot occur in a minimum counterexample to \cref{thm:odd10}. In \S \ref{sec: forb/flex}, we introduce the forb-flex method and prove a variety of useful lemmas that show various vertices have high $\forb$. We take advantage of this in section \S \ref{sec:odd10}, where we prove Theorem \ref{thm:odd10} by using discharging. And finally, in \S \ref{sec:pcf11}, we prove Theorem \ref{thm:pcf11} by using similar techniques.

\section{Preliminaries}
\label{sec: prelims}
We follow standard graph theory terminology. For any undefined terms, see \cite{west2001introduction}.
For a (partial) proper coloring $\vp$ of a graph $G$ and a vertex $v$, let $\phi_\mathsf{o}(v)$ denote a color that appears an odd number of times in $N(v)$, if it exists. Also, let $\vp_*(v)$ denote a color that uniquely appears in $N(v)$, if it exist. We say $\phi_\mathsf{o}(v)$ is an \textit{odd color} of $v$, and $\phi_*(v)$ is a \textit{unique color} of $v$. We use $\dom(\vp)$ to denote the domain of $\vp$. Note, then, $\phi$ is an odd coloring if and only if $\phi_\mathsf{o}(v)$ is defined for all $v \in \dom(\phi)$, and similarly $\phi$ is a PCF coloring if and only if $\phi_*(v)$ is defined for all $v \in \dom(\phi)$.

Given $n \in \N$, we let $[n] \defeq \{1, 2, \ldots, n\}$. For $k \in \N$, we say $v$ is a \emph{$k$-vertex}, a \emph{$k^+$-vertex}, or a \emph{$k^-$-vertex} if $\deg(v) = k$, $\deg(v) \ge k$, or $\deg(v) \le k$, respectively. We call $v$ an {\em odd-vertex} if $v$ has an odd degree, and an {\em even-vertex} if $v$ has an even degree. If $v\in N(u)$ we say $v$ is a $k$-neighbor, a $k^+$-neighbor, a $k^-$-neighbor, an even-neighbor, or an odd-neighbor of $u$, respectively.

By a sequence $j_1^+ - j_2^+ - \cdots - j_k^+$, where $k \in \N$ and $j_1, \ldots, j_k \in \N$, we mean an arbitrary sequence of integers $i_1 - i_2 - \cdots - i_k$ such that $j_n \leq i_n$ for each $1 \leq n \leq k$. For instance, $3^+ - 2 - 2 - 4^+$ represents any sequence of the form $i - 2 - 2 - j$ with $i \geq 3$, $j \geq 4$.

Choose a plane graph $G$, and let $F(G)$ be the face set of $G$.
For a face $f \in F(G)$, let $\ell(f)$ be the number of edges incident with $f$.
For $k \ge 3$, we say $f$ is a \emph{$k$-face}, a \emph{$k^+$-face}, or a \emph{$k^-$-face} if $\ell(f)=k$, $\ell(f) \ge k$, or $\ell(f) \le k$, respectively. 

For $k \geq 1$, a \emph{$k$-path} is a path $v_1\cdots v_k$ with $k$ vertices (i.e., it has length $k-1$). 
Of particular importance in our arguments are paths whose vertices are all of degree 2. To this end, we define the following:

\begin{definition} Let $k \in \N$ and $P$ be a $k$-path $v_1\cdots v_k$ such that $\deg(v_i) = 2$ for $1 \leq i \leq k$. Let $u$ and $w$ be the neighbors of $v_1$ and $v_k$ not in $P$. If $u$, $w$ are $3^+$-vertices, then we say $P$ is a \emph{$k$-thread}, and we call $u,w$ \emph{anchors} of $P$. We call $v_1$ and $v_k$ \emph{endpoints} of $P$. Given an $n$-thread $P$, we will say $P$ is a \emph{$k^+$-thread} (respectively $k^-$-thread) if $n \geq k$ (resp. $n \leq k$).
\end{definition}

In our analysis, 2-vertices and 3-vertices have a particularly important role, as they are the only vertices that don't receive a positive initial charge. As such, we introduce additional terminology to refer to certain 2-vertices and 3-vertices based on the degrees of their respective neighbors. 

\begin{definition}\label{def: good, bad, worst}
Let $G$ be a graph and $v \in V(G)$ be a $2$-vertex on a $1$-thread (Recall $v$ has no $2^-$-neighbor). Then $v$ is \begin{itemize}
    \item \defn{good} if $v$ has no $3$-neighbor,
    \item \defn{bad} if $v$ has exactly one $3$-neighbor,
    \item \defn{\worst} if $v$ has two $3$-neighbors.
\end{itemize}
\end{definition}

\begin{definition}
Let $G$ be a graph and $v \in V(G)$ be a 3-vertex with no $1$-neighbor. Then $v$ is \begin{itemize}
    \item \defn{good} if $v$ has no $2$-neighbor,
    \item \defn{semi-bad} (\defn{s-bad}, for short) if $v$ has exactly one $2$-neighbor and no $3$-neighbor,
    \item \defn{bad} if $v$ has exactly one $2$-neighbor and exactly one $3$-neighbor,
    \item \defn{\worst} if at least two $2$-neighbors.
\end{itemize}
\end{definition}

For a face $f$ of a connected plane graph $G$, its \textit{boundary} is the closed walk containing all vertices bordering $f$ (note that the boundary of a face may not be a cycle, hence the use of the closed walk). We also note that as the walk is closed, the degree sequence of the boundary has an arbitrary starting point.

It will be convenient to describe the boundary degree sequence of a face of a minimum counterexample to \cref{thm:odd10} and \cref{thm:pcf11} as the concatenation of various subsequences. To this end, we tactically define the following subsequences, which we will call \textit{arrays}. 
\begin{definition}\label{def:array}
    
    We define six subsequences, which we call \textit{arrays}, as follows:
    %$a_4$, $a_3$, $a_2^{\Worst}$,$a_2^{\Bad}$, $a_2^{\Good}$, and $a_1$
    \begin{enumerate}[label = {(\arabic*)}, itemsep = 5pt]
        \item $a_4$ is the subsequence $4^+-2-2-2$ of sequence $4^+ - 2- 2- 2-4^+$,
        \item $a_3$ is the subsequence $4^+ - 2-2$ of sequence $4^+ - 2- 2-4^+$,
        \item $a_2^{\Worst}$ is the subsequence $3-2$ of sequence $3-2-3$,
        \item $a_2^{\Bad}$ is the subsequence $3-2$ of sequence $3-2-4^+$, or the subsequence $4^+-2$ of the sequence $4^+-2-3$,
        \item $a_2^{\Good}$ is the subsequence $4^+-2$ of the sequence $4^+-2-4^+$,
        \item $a_1$ is the subsequence $3^+$ of the sequence $3^+-3^+$.
    \end{enumerate}
   
By $a_2$, we mean one of $a_2^{\Worst}$, $a_2^{\Bad}$, or $a_2^{\Good}$. A closed walk $w$ has an \textit{array representation} if the degree sequence of $w$ can be written using a concatenation of arrays. Given a face $f$ of a connected plane graph, by an \textit{array representation of $f$} we mean an array representation of its boundary.
\end{definition}

 For example, if 
$4^+ - 2 - 2 -2 - 4 - 2 - 3 - 4^+ - 2- 2$
is the degree sequence of a closed walk (with the degree of the start/end point written only once), then it has an array representation
$a_4a_2^{\Worst}a_1a_3$.
Note that, as a closed walk can have an arbitrary starting point and orientation, it may have more than one array representation.

We will verify in \cref{cor: all faces have array representations} of \S\ref{sec: reducible configs} that each face of a minimum counterexample to Theorem \ref{thm:odd10} has an array representation. Analyzing the array representations will allow us to quickly gain insight into the charge of a face. For instance, after an initial application of discharging rules, particular faces will have negative and nonnegative charges --- we call these \textit{poor} and \textit{rich} faces and define them below in terms of their array representations. 

\begin{definition}\label{def:poor}
    A face of a connected plane graph is \defn{poor} if its boundary can be represented by a permutation of one of the following:
    \begin{enumerate}
        \item $a_4, a_4, a_1, a_1$
        \item $a_4, a_3, a_2^{\Good}, a_1$
    \end{enumerate}
    It is \textit{rich} otherwise.
\end{definition}

The following lemma will be used in \S\ref{sec:odd10} for the proof of \cref{thm:odd10}. 

\begin{lemma}\label{lemma:three-and-two-choose-odd}
Let $n \ge 2$ and $X$ be a multiset of elements in $[n]$. 
If we change one element $x$ of $X$ into an element in $[n] \setminus \{x\}$, and obtain $X'$, then either $X$ or $X'$ has an element of odd multiplicity.    
\end{lemma}

\begin{proof}
Suppose $X$ does not have an element of odd multiplicity.
This means that every element of $X$ appears an even number of times in $X$.
If we change one element $x \in X$ into an element in $[n] \setminus \{x\}$ and obtain $X'$, then $x$ appears an odd number of times in $X'$, satisfying the statement of \cref{lemma:three-and-two-choose-odd}. 
\end{proof}

The following lemma will be used in \S\ref{sec:pcf11} for the proof of \cref{thm:pcf11} and will allow us to utilize the tools developed in \S\ref{sec: forb/flex} for PCF colorings.

\begin{lemma}\label{lem:odd-and-pcf}
    Let $G$ be a graph and $\phi$ be a proper coloring of $G$. Let $v \in V(G)$ be a $2$-vertex or $4$-vertex. Then $v$ has an odd color if and only if $v$ has a unique color.
\end{lemma}
\begin{proof}
 Suppose $v$ has an odd color under $\phi$, say $\alpha$. Then as $v$ is a $2$-vertex or a $4$-vertex, it follows $|\{ u \in N(v) : \phi(u) = \alpha\}|$ is either 1 or 3. In the former case, $\alpha$ is a unique color of $v$. In the later case, $v$ is a $4$-vertex, and thus there exists $w \in N(v)$ such that $\phi(w) \neq \alpha$. Then $\phi(w)$ is a unique color of $v$. The other direction is trivial.
\end{proof}

\section{Reducible configurations for Theorem \ref{thm:odd10}}\label{sec: reducible configs}
In this section we prove that a minimum counterexample to \cref{thm:odd10} cannot contain various structures, which we call
 \textit{reducible configurations}. Throughout this section, we use terminology from \S\ref{sec: prelims}, and we suppose that \cref{thm:odd10} is false and let $G$ be a counterexample to \cref{thm:odd10} such that $|V(G)|$ is minimal. Clearly, $G$ is connected, as otherwise, there is some connected component of $G$ that is a smaller counterexample.
 
 Many of our proofs will follow the framework of deleting certain vertices from $G$, taking an odd coloring of the resulting graph, and then extending the coloring to all $G$, thereby reaching a contradiction. For statements concerning faces, it will often be useful to write the structures in terms of arrays.

The observation below follows from claims 5.1, 5.2, and 5.3 in \cite{cho2023odd}. While these claims are for a graph $G$ that is a minimal counterexample to the statement that all planar graphs with girth at least 11 have $\chi_{\mathsf{o}}(G) \leq 4$, the proof of these claims do not use the girth condition. Thus the statements are true for a minimal counterexample $G$ to Theorem \ref{thm:odd10} as well.

\begin{observation}\label{obs:odd}
The following hold for $G$: 
\begin{enumerate}[label = {(\arabic*)}]%, leftmargin = \leftmargin + 2\parindent]
\item\label[observation]{obs:odd - no degree 1 vertex} $G$ has no $1$-vertex.
\item\label[observation]{obs:odd - no 4+ thread} $G$ has no $4^+$-thread.
\item\label[observation]{obs:odd - no odd vertex adjacent to a 2+ thread} $G$ has no odd-vertex which anchors a $2^+$-thread.
\item\label[observation]{obs:odd - } Let $v$ be a $4^+$-vertex of $G$. If $v$ has an odd degree, then the number of 2-vertices on a thread anchored by $v$ is at most $\deg_G(v)$. If $v$ has an even degree, then the number of 2-vertices on a thread anchored by $v$ is at most $3\deg_G(v)-5$.
\end{enumerate}
\end{observation}

As $G$ is connected, every face has a boundary. The following corollary shows the boundary can always be written in terms of arrays.
\begin{corollary}\label{cor: all faces have array representations}
The boundary of every face of $G$ has an array representation.
\end{corollary}
\begin{proof}
By \ref{obs:odd - no degree 1 vertex} of \cref{obs:odd}, the degree sequence of any face boundary never contains 1. 
By \ref{obs:odd - no 4+ thread}, the degree sequence of any face boundary does not contain the subsequence $2 -2 -2 -2$. 
By \ref{obs:odd - no odd vertex adjacent to a 2+ thread}, the degree sequence of any face boundary does not contain the subsequence $3 - 2 -2$. 
Therefore, the only subsequences the degree sequence of a face boundary can contain are exactly the arrays as defined in \cref{def:array}. Thus the degree sequence of every face boundary can be written as a concatenation of arrays. 
\end{proof}

\begin{observation}\label{cutvertex}
No $2$-vertex of $G$ is a cut vertex.
\end{observation}

\begin{proof}
    Suppose for contradiction that $v$ is a cut vertex of $G$ with $\deg(v) = 2$. Let $N(v) = \{x_1,x_2\}$. Let $\vp$ be an odd 4-coloring of $G - v$ and let $C_1$ and $C_2$ be the components of $G-v$ containing $x_1$ and $x_2$ respectively. Since $C_1$ and $C_2$ are disconnected, we can permute the colors on each component independently such that $\oc(x_1) = 1$, $\vp(x_1) =2$, $\oc(x_2) = 1$, and $\vp(x_2) = 3$. We can then color $\vp(v) = 4$, which extends $\vp$ to $G$ and it is a contradiction.
\end{proof}

\begin{observation}\label{Obs:TheOnlyBadColoringOfa2Thread}
    Let $x_1x_2$ be a $2$-thread with anchors $y_1$ and $y_2$, respectively. Let $\vp$ be an odd 4-coloring of $G - \{x_1, x_2\}$. If $\vp$ cannot be extended to $G$, then $\vp(y_1) \neq \vp(y_2)$ and $\oc(y_1) = \oc(y_2)$.
\end{observation}

\begin{proof}
    Let $y_1x_1x_2y_2$ be as defined and let $\vp$ be an odd $4$-coloring of $G - \{x_1, x_2\}$. Suppose for contradiction that either $\vp(y_1) = \vp(y_2)$ or $\oc(y_1) \neq \oc(y_2)$. Extend the coloring $\vp$ by first coloring $\vp(x_1) \in [4] \setminus \{\oc(y_1), \vp(y_1), \vp(y_2)\}$. Now consider the set $S = \{\vp(x_1), \oc(y_2), \vp(y_1), \vp(y_2)\}$. If $\vp(y_1) = \vp(y_2)$, then $|S| \le 3$ and we can finish extending $\vp$ to $G$ by coloring $\vp(x_2) \in [4] \setminus S$, which is a contradiction. Similarly, if $\oc(y_2) = \vp(y_1)$, then $|S| \le 3$ and again we may extend $\vp$ to $G$.

    Otherwise, we have $\vp(y_1) \neq \vp(y_2)$ and $\oc(y_2) \neq \vp(y_1)$. By the hypothesis, either $\vp(y_1) = \vp(y_2)$ or $\oc(y_1) \neq \oc(y_2)$. Since $\vp(y_1) \neq \vp(y_2)$, it follows that $\oc(y_1) \neq \oc(y_2)$. Clearly, we also have $\oc(y_2) \neq \vp(y_2)$. Thus, it follows that $\oc(y_2) \in [4] \setminus \{\oc(y_1), \vp(y_1), \vp(y_2)\}$. Therefore, we can color $\vp(x_1) = \oc(y_2)$ so that $|S| \le 3$, allowing us to extend $\vp$ to $G$.
    
\end{proof}

\begin{observation}\label{Obs:TheOnlyBadColoringOfa3Thread}
    Let $x_1 x_2 x_3$ be a $3$-thread with anchors $y_1$ and $y_2$. Let $\vp$ be an odd $4$-coloring of $G - \{x_1,x_2,x_3\}$. If $\vp$ cannot be extended to $G$, then $\oc(y_1)=\vp(y_2)$ and $\oc(y_2) = \vp(y_1)$.
\end{observation}

\begin{proof}
Let $y_1x_1x_2x_3y_2$ be as defined and let $\vp$ be an odd $4$-coloring of $G - \{x_1,x_2,x_3\}$. Suppose for contradiction that $\oc(y_1) = \vp(y_2)$ or $\oc(y_2) = \vp(y_1)$. If $\oc(y_1) \neq \vp(y_2)$, then we can extend $\vp$ by coloring $\vp(x_2) = \oc(y_1)$, coloring $\vp(x_3) \in [4] \setminus \{\vp(x_2), \vp(y_2), \oc(y_2)\}$, and coloring $\vp(x_1) \in [4] \setminus \{\vp(x_2), \vp(x_3), \vp(y_1), \oc(y_1)\}$, which is a contradiction. The coloring for $x_1$ is possible because $\vp(x_2) = \oc(y_1)$ so $|\{\vp(x_2), \vp(x_3), \vp(y_1), \oc(y_1)\}| \le 3$. A symmetric argument applies if $\oc(y_2) \neq \vp(y_1)$.

\end{proof}

\begin{observation}\label{obs:3vx4nb}

Every $3$-vertex in $G$ has a $4^+$-neighbor of even degree. 

\end{observation}

\begin{proof}
Suppose for contradiction that $u \in V(G)$ is a 3-vertex with only $2$-neighbors and odd-neighbors. Let $X$ be the set of $2$-neighbors of $u$, $G' = G - (X \cup \{u\})$, and $\vp$ be an odd $4$-coloring of $G'$. Consider the set of vertices $S = \bigcup_{v \in X \cup \{u\}} N_{G'}(v)$ and the set of colors $C = \{\vp(v) : v \in S\}$. Since $\deg_G(u) = 3$, $|S| \le 3$ and hence $|C| \le 3$. Extend $\vp$ to $G$ by first coloring $\vp(u) \in [4] \setminus C$. Then for each $v \in X$, let $N_G(v) = \{u, w\}$ and extend $\vp$ by coloring $\vp(v) \in [4] \setminus \{\vp(u), \vp(w), \oc(w)\}$. This yields an odd 4-coloring of $G$, which is a contradiction.
\end{proof}

\begin{observation}\label{Obs:1ThreadNextToOdd}
    $G$ has no path $v_1v_2v_3v_4v_5$ where $\deg(v_1)$ is odd and $\deg(v_2) = \deg(v_4) = \deg(v_5) = 2$.
\end{observation}

\begin{proof}
    Suppose for contradiction that $G$ contains a path $v_1v_2v_3v_4v_5$ where $\deg(v_1)$ is odd and $\deg(v_2)=\deg(v_4)=\deg(v_5) = 2$. By \cref{obs:odd}\ref{obs:odd - no odd vertex adjacent to a 2+ thread} $v_3$ is an even-vertex. Since $v_4$ and $v_5$ are $2$-vertices, they are part of either a $2$-thread or a $3$-thread. Let the thread be $P$ with anchors $v_3$ and $w$ and let $\vp$ be an odd $4$-coloring of $G' = G - V(P)$.

    Observe that $\deg_{G'}(v_1)$ and $\deg_{G'}(v_3)$ are odd, so recoloring $v_2$ to a color in $[4] \setminus \{\vp(v_1), \vp(v_3)\}$ still yields an odd 4-coloring of $G'$. At least one of the recolorings avoids the situation in \cref{Obs:TheOnlyBadColoringOfa2Thread} and \cref{Obs:TheOnlyBadColoringOfa3Thread} and therefore $\vp$ can be extended to an odd 4-coloring of $G$, which is a contradiction.
    
\end{proof}

\begin{corollary}\label{lem:2badarray_adjacent}
No face of $G$ has an array representation containing  $a_2^{\Bad}a_3$, $a_3a_2^{\Bad}$, $a_2^{\Bad}a_4$, or $a_4a_2^{\Bad}$ as a subsequence.
\end{corollary}
\begin{proof}
If a face of $G$ were to have any of these as a subsequence, then $G$ would have a path with degree sequence $3-2-4^+-2-2$. But by \cref{Obs:1ThreadNextToOdd}, this cannot occur.
\end{proof}

\begin{lemma}\label{greedylemma}
    Any cycle of $G$ whose array representation  contains $a_4$ also contains $a_1$ or $a_2$. 
\end{lemma}

\begin{proof}
Suppose for contradiction that $G$ contains a cycle $C = v_1v_2\cdots v_{n}$ whose array representation contains $a_4$ but no $a_1$ nor $a_2$. Without loss of generality suppose $v_1v_2v_3v_4$ has degree sequence $a_4$ ($4^+ - 2 - 2 - 2$). Let $G' = G-S$ where $S = \{v_i\in C: \deg(v_i)=2\}$ and let $\vp$ be an odd 4-coloring of $G'$. We will greedily extend $\vp$ to a coloring of $G$. 

By \cref{obs:odd}\ref{obs:odd - no odd vertex adjacent to a 2+ thread}, $\deg_G(v_1)$ is even. Since $a_1$ does not appear in any array representation of $f$, $\deg_G(v_n) = 2$. Hence, $\deg_{G'}(v_1) = \deg_G(v_1) - 2$ so $\deg_{G'}(v_1)$ is also even. Therefore, there exist distinct odd colors $\alpha, \beta$ of $v_1$ under $\vp$.

First, extend $\vp$ by setting $v_2$ to a color in $[4]\setminus\{\vp(v_1),\alpha,\beta\}$. Until $\vp$ is extended to all of $G$, choose the vertex $v_i \in S$ with smallest index $i$ such that $\vp$ is not extended to $v_i$ and color it 
\begin{align*}
    \vp(v_i) \in \begin{cases}
        [4] \setminus \{\vp(v_{i-1}), \oc(v_{i-1})\} & \text{if $v_{i+1}, v_{i+2} \in S$,}\\
        [4] \setminus \{\vp(v_{i-1}), \oc(v_{i-1}), \vp(v_{i+1})\} & \text{if $v_{i+1} \not\in S$,}\\
        [4] \setminus \{\vp(v_{i-1}), \oc(v_{i-1}), \vp(v_{i+2})\} & \text{if $v_{i+2} \not\in S$.}
    \end{cases}
\end{align*}
This is well-defined because no array representation of $C$ contains $a_1$, so at most one of $v_{i+1}$ and $v_{i+2}$ does not lie in $S$. It is clear that this iterative process leads to a proper coloring of $G$. Since no array representation of $C$ has an $a_2$, our process leads to each vertex in $S$ having an odd color.

Suppose $v_i \in C \setminus S$ and $i \neq 1$. Immediately after coloring $v_{i-1}$ in the above process, $v_i$ may not have an odd color. If this is the case, then it will have an odd color after coloring $v_{i+1}$. If $v_i$ does have an odd color after coloring $v_{i-1}$, then the above process avoids $\oc(v_i)$ when coloring $v_{i+1}$ so $v_i$ continues to have an odd color.

Finally, consider $v_1$. The two colors $\alpha$ and $\beta$ are odd colors for $v_1$ prior to coloring any vertex in $S$. Since $\vp(v_2)$ is neither $\alpha$ nor $\beta$, no matter what $v_n$ is colored, at least one of $\alpha$ or $\beta$ remains an odd color of $\vp(v_1)$. Thus, we have extended $\vp$ to an odd 4-coloring of $G$ which is a contradiction.
\end{proof}

\begin{corollary}\label{cor: Greedy faces}
 Let $f$ be a face of $G$. Then any array representation of $f$ which contains $a_4$ also contains $a_1$ or $a_2$. 
\end{corollary}

\begin{proof}
Suppose for contradiction that $f$ is a face whose array representation contains $a_4$ but not $a_1$ nor $a_2$. Then $f$ must contain a cycle whose array representation contains $a_4$ but not $a_2$ nor $a_1$, which cannot occur by \cref{greedylemma}.
\end{proof}

\begin{observation}\label{obs:10face4222}
    $G$ contains no $10$-face with degree sequence $3^+ - 2 - 2^+ - 2 - 3^+ - 2 - 3^+ - 2 - 2 - 2$.
\end{observation}

\begin{proof}
Suppose for contradiction that $G$ contains a $10$-face $v_1a_1v_2a_2v_3a_3v_4x_1x_2x_3v_1$ where $\deg(a_i)=\deg(x_i) = 2$ for all $i\in [3]$.
Since the girth of $G$ is at least $10$, the boundary of this face is a cycle and the incident vertices are distinct.
Let $G'=G - \{x_1,x_2,x_3\}$.
By the minimality of $G$, there is $\vp$ an odd 4-coloring of $G'$. 
Since we cannot extend $\vp$ to $G$, by \cref{Obs:TheOnlyBadColoringOfa3Thread}, we may assume $\vp(v_1) =\oc(v_4)= 1$, $\vp(v_4)=\oc(v_1)=2$, and $2$ and $1$ are unique odd colors of $v_1$ and $v_4$, respectively. 
Our goal is to recolor $a_i$'s under $\vp$ so that we obtain an odd $4$-coloring $\vp$ of $G'$ such that $\vp(v_4)=2$, but $\oc(v_1) \neq 2$.
Then by \cref{Obs:TheOnlyBadColoringOfa3Thread}, we can extend $\vp$ to $G$, which is a contradiction.

Let us recolor $a_1$ under $\vp$ with a color in $[4]\setminus \{1,2,\vp(v_2)\}$.
Since $v_1$ is an odd-vertex, $\oc(v_1)$ is guaranteed, and $\oc(v_1) \neq 2$.
If $\oc(v_2)$ exists, then $\vp$ is an odd $4$-coloring of $G'$, which we desired.
If $\oc(v_2)$ does not exist, then we recolor $a_2$ under $\vp$ with a color in $[4]\setminus \{\vp(v_2), \vp(a_2), \vp(v_3)\}$.
Then by \cref{lemma:three-and-two-choose-odd}, $\oc(v_2)$ exists.
If $\oc(v_3)$ exists, then $\vp$ is an odd $4$-coloring of $G'$, which we desired.
If $\oc(v_3)$ does not exist, then we recolor $a_3$ under $\vp$ with a color in $[4]\setminus \{\vp(v_3), \vp(a_3), \vp(v_4)\}$.
Then by \cref{lemma:three-and-two-choose-odd}, $\oc(v_3)$ exists.
Now, $\oc(v_4)$ is always guaranteed since $v_4$ is an odd-vertex.
Thus, $\vp$ is an odd $4$-coloring of $G'$, which we desired.
This is a contradiction.
\end{proof}

\cref{obs:10face4222} implies the immediate corollary in terms of arrays.

\begin{corollary}
\label{cor: a4a2goodx3 cannot occur}
$G$ contains no face with array representations $a_4a_2^{\Good}a_2^{\Good}a_2^{\Good}$ or $a_4a_4a_2^{\Good}$.
\end{corollary}

\section{The $\forb$-$\flex$ Method}
\label{sec: forb/flex}
In this section we introduce the forb-flex method. We will primarily use it in the context of odd coloring to find many reducible configurations that cannot occur in a minimum counterexample to \cref{thm:odd10}. This tool works by analyzing the structure of the neighborhood of a vertex and assigning to each structure a score based on how many colors it prevents a vertex from receiving. To this end, for a graph $G$ and $u \in V(G)$, we will partition $N(u)$ into 9 types of vertices. Note that we use terminology defined in \S\ref{sec: prelims}.

\begin{definition}
Let $G$ be a graph and $u \in V(G)$. Then $v \in N(u)$ is a:
\begin{enumerate}[label = {}]
    \item \emph{$t_i$-neighbor} if $v$ is the endpoint of an $i$-thread, for $i \in \{1, 2, 3\}$,
    \item \emph{$t_{\Good}$-neighbor} if $v$ is a good $3$-vertex,
    \item \emph{$t_{\Semibad}$-neighbor} if $v$ is a semi-bad $3$-vertex,
    \item \emph{$t_{\Bad}$-neighbor} if $v$ is a bad $3$-vertex,
    \item \emph{$t_{\Worst}$-neighbor} if $v$ is a \worst~$3$-vertex,
    \item \emph{$t_{\Even}$-neighbor} if $v$ is an even $4^+$-vertex,
    \item \emph{$t_{\Odd}$-neighbor} if $v$ is an odd $5^+$-vertex.
\end{enumerate}

\end{definition}

Many of our proofs will require us to consider the resulting graph formed from deleting a vertex $u$ and certain vertices near $u$. Therefore, we define the following:
\begin{definition}
For a graph $G$ and vertex $u \in V(G)$, define:
\begin{align*}
    S[u] := \{u\} &\cup \{v\in V(G): \text{$v$ is on a thread anchored by $u$}\}\\
    &\cup \{w\in N[v] : \text{$v$ is a  $t_{\Worst}$-neighbor of $u$}\}.
\end{align*}
See \cref{figure:vertex-types} for an example.
\end{definition}

\begin{figure}[h!]
    \begin{center}
    \includegraphics[width=.8\textwidth]{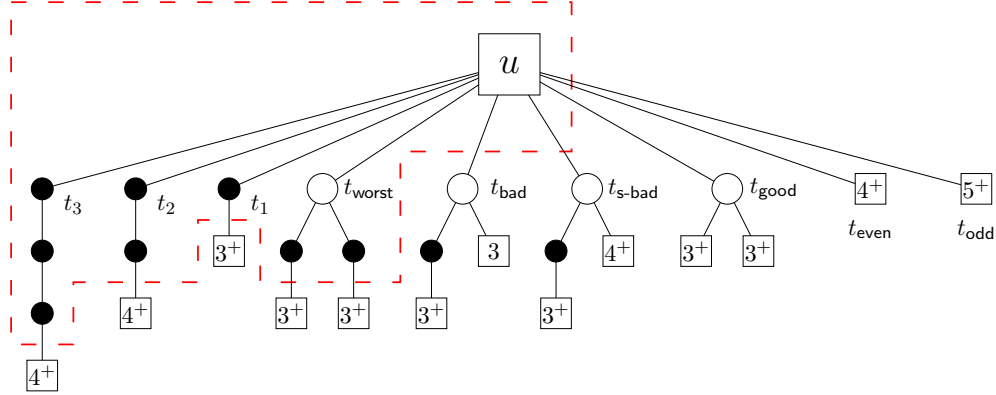}
    \caption{Illustration of the 9 neighbor types of $u$ together with $S[u]$ in red-dashed. Note that by \cref{obs:odd}\ref{obs:odd - no odd vertex adjacent to a 2+ thread}, the anchors of the $2^+$-threads are $4^+$-vertices, and the $t_{\Worst}$-, $t_{\Bad}$-, $t_{\Semibad}$-, $t_{\Good}$-neighbors of $u$ are not anchors of $2^+$-threads.}
    \label{figure:vertex-types}
    \end{center}
\end{figure}

The strategy in our proofs will be to take a minimum counterexample $G$ to \cref{thm:odd10} and consider the graph $G - S[u]$ for a particular choice of $u$. We then fix an odd $4$-coloring $\phi$ on $G - S[u]$ and attempt to extend this coloring to all of $G$. 
The key will be in choosing the right color for $u$. For instance, $u$ cannot receive the color of any of its neighbors, or the potentially only odd color of one of its even-neighbors (note that an odd-vertex always has an odd color in any proper coloring). We will strategically choose a color for $u$ such that, if necessary, one of its neighbors can be safely switched to another color. 
 To this end, we determine the situations under which a neighbor of $u$ can switch colors.
 
For instance, for a $k$-thread $v_1\cdots v_k$ with anchors $u$ and $w$, where $\vp(w)=i$ and $\oc(w)=j$ are given, certain color choices for $u$ give one and only one way to color $v_1\cdots v_k$ while maintaining the odd constraint for all vertices except $u$.
For example, if $uv_1w$ is a $1$-thread with anchors $u$ and $w$, and $\vp(w)=1, \oc(w) = 2$, and $\vp(u) = 3$, then $v_1$ must be colored $\vp(v_1)=4$. Other color choices for $u$ however allow flexibility for the color which $v_1$ receives, guaranteeing us an odd color for $u$ by \cref{lemma:three-and-two-choose-odd}. 
In the example above, if $\vp(w) = 1$, $\oc(w) = 2$, and $\vp(u) = 2$, then one can color $\vp(v_1)\in \{3,4\}$, whichever one will give $u$ an odd color.  
This motivates the following definitions.

\begin{definition}\label{def: flex thread}  Let $G$ be a graph. Let $P$ be a $k$-thread  $v_1\cdots v_k$ where $k \in [3]$ with anchors $u$ and $w$ (where $u \in N(v_1)$, and $u \neq w$). Let $\phi$ be a partial proper coloring of $G$ that includes $u$ and $w$ such that $\phi_{\mathsf{o}}(w)$ exists, and, furthermore, if $\phi(v_k)$ is defined, then $\phi_{\mathsf{o}}(w) \neq \phi(v_k)$. Then $P$ is \textit{$(u,\phi)$-flexible}
if 
\begin{enumerate}
        \item $P$ is a 1-thread and $\phi(u) = \phi_o(w)$,
        \item $P$ is a 2-thread and either $\phi(u) = \phi(w)$ or $\phi(u) = \phi_\mathsf{o}(w)$, or
        \item $P$ is a $3$-thread and $\phi(u) \neq \phi_o(w)$.
    \end{enumerate}
\end{definition}

The next lemma says that when $P$ is $(u,\phi)$-flexible, there exists two ways to extend or recolor $\phi$ on $P$ which differ at $v_1$. This guarantees that no matter how $\phi$ is extended, we can recolor $P$ to guarantee $u$ has an odd color. 

\begin{lemma}\label{lem: two choices}
Let $G$, $P$, and $\phi$ be as described in \cref{def: flex thread}, with $P$ $(u,\phi)$-flexible. Then there exists two proper colorings $\phi'$ and $\phi''$ of $\dom(\phi) \cup P$ such that
\begin{enumerate}[label = {(\arabic*)}]
    \item\label[lemma]{lem: two choices - 1} $\vp(x)= \vp'(x)=\vp''(x)$ for all $x\in \dom(\vp)\setminus P$.
    \item\label[lemma]{lem: two choices - 2} For all $x \in P \cup \{w\}$, $\phi_{\mathsf{o}}'(x)$ and $\phi_{\mathsf{o}}''(x)$ exist. 
    \item\label[lemma]{lem: two choices - 3} $\phi'(v_1)\neq \vp''(v_1)$.
\end{enumerate}
\end{lemma}
\begin{proof}
Start by taking $\phi'$ and $\phi''$ to simply be $\phi$ restricted to $\dom(\phi)\setminus P$.

If $P$ is a $1$-thread, then as $P$ is $(u, \phi)$-flexible, it follows $\phi(u) = \phi_\mathsf{o}(w)$, and thus there exists two colors $\alpha$ and $\beta$ in $[4]\setminus \{\phi(u), \phi(w), \phi_\mathsf{o}(w)\}$. 
Setting $\phi'(v_1) = \alpha$ and $\phi''(v_2) = \beta$ gives the desired colorings. 

If $P$ is a 2-thread, then as $P$ is $(u, \phi)$-flexible, either $\phi(u) = \phi(w)$ or $\phi(u) = \phi_\mathsf{o}(w)$. In both cases, there exists two colors in $[4] \setminus \{\phi(u), \phi(w), \phi_\mathsf{o}(w) \}$, say $\alpha$ and $\beta$. Set $\phi'(v_1) = \alpha$, $\phi'(v_2) = \beta$, and $\phi''(v_1) = \beta$, $\phi''(v_2) = \alpha$. This gives the desired colorings.

If $P$ is a 3-thread, then we define the colorings as follows: set $\phi'(v_2) = \phi''(v_2) = \phi_\mathsf{o}(w)$. Then there exists 2 colors, $\alpha$ and $\beta$, in $[4]\setminus\{\phi(w), \phi_\mathsf{o}(w), \phi(v_2)\}$. Set $\phi'(v_1) = \alpha$, $\phi'(v_3) = \beta$, and set $\phi''(v_1) = \beta$, $\phi''(v_3) = \alpha$. This gives the desired colorings.
\end{proof}

These possibilities are illustrated in \cref{figure:flexible-threads}.

\begin{figure}[tbh!]
\centering

\resizebox{.85\textwidth}{!}{
\input{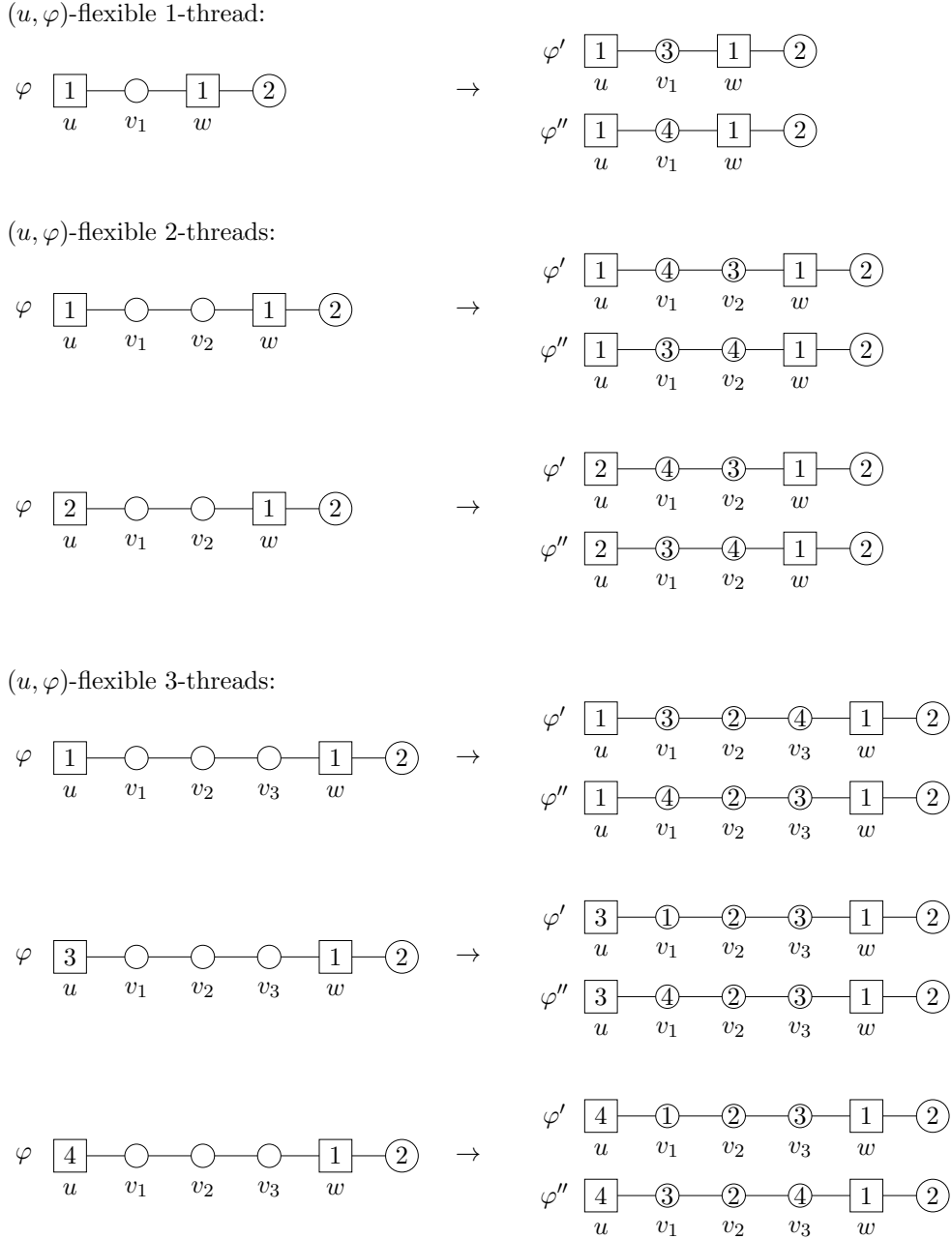}}
    
\caption{Flexible tuples up to relabeling of the colors in $[4]$.}

\label{figure:flexible-threads}
\end{figure}

\begin{definition}
    Let $G$ be a graph, $u \in V(G)$, and $\phi$ be a partial proper $4$-coloring of $G - \{u\}$. Then
    \[\Flex_G(u, \phi) \defeq \{\alpha \in [4] : \text{if $\phi(u)$ is set to  $\alpha$, then $P$ is $(u, \phi)$-flexible for some $P$ anchored by $u$}\}.\]
    When $G$ is clear from context, we will simply write $\Flex(u, \phi)$.
\end{definition}

It is easy to see that if $P$ is $(u,\phi)$-flexible, and $P$ is a $k$-thread, then there are at least $k$ colors in $\Flex(u, \phi)$. This motivates the following definition: 
\begin{definition}\label{def: flex}
 Let $G$ be a graph. The \defn{flexible number} of a vertex $u\in V(G)$ is given by
 \[\flex_G(u) = \max\{k : k\in \{1,2,3\},~u \text{ has a } t_k\text{-neighbor}\}.\]
 When $G$ is clear from context, we will write $\flex(u)$.
\end{definition}

As mentioned above, in our proofs we will be mostly concerned with extending a coloring of $G - S[u]$ to all of $G$, for some vertex $u$. To do so, we must consider the colors which we do not want to color $u$, as these colors would create a conflict. This motivates the following definition:
\begin{definition}
Let $G$ be a graph, $u \in V(G)$, and $\phi$ be a partial proper coloring of $G$. We define the \defn{forbidden color set} of $u$ with respect to $\vp$ as 
 \begin{align*}
  \Forb_G(u, \vp) \defeq &\ \{\vp(x) : x \in N(u) \text{ and $\phi(x)$ is defined}\}\\ \cup &\ \{\oc(x) : \text{$x \in N(u)$, $\deg(x)$  is even, $x$ does not anchor an $(x,\phi)$-flexible thread,}\\
  & \hspace{1cm} \text{and $\phi_{\mathsf{o}}(x)$ is defined.}\} \\
    \cup &\ \{\vp(w) : w \neq u, \text{ $u$ and $w$ anchor a common 1-thread, and $\phi(w)$ is defined} \}.
 \end{align*}
 When $G$ is clear from context we simply write $\Forb(u, \phi)$.
\end{definition}

\begin{lemma}\label{lemma: flex > forb}
Let $G$ be a graph of girth at least 10, $u \in V(G)$, and $\phi$ be an odd 4-coloring of $G -S[u]$. If $|\Flex(u, \phi)| > |\Forb(u, \phi)|$, then there exists an odd $4$-coloring $\phi'$ of $G$ such that every $(u, \phi)$-flexible thread is $(u, \phi')$-flexible.
\end{lemma}
\begin{proof} Let $\vp^0=\vp$.
Since $|\Flex(u, \phi)| > |\Forb(u, \phi)|$, there exists $\alpha \in \Flex(u, \phi)\setminus\Forb(u, \phi)$. First extend $\phi$ by coloring $u$ the color $\alpha$. Let $P$ be the $(u,\vp)$-flexible thread, and let $v$ be the neighbor of $u$ on $P$. 
Next extend $\phi$ by coloring the vertices in $S[u]\setminus V(P)$, one at a time, updating $\phi$ at each step, according to the following rules: for each $x \in \left(N(u) \cap S[u]\right) \setminus \{v\}$,
\begin{itemize}
    \item \textbf{Rule 1:} \textit{$x$ is a $t_1$-neighbor of $u$.} Then $x$ is on a $1$-thread whose anchors are $u$ and $w$ for some $w \in V(G)$. Color $x$ any color in $[4] \setminus \{\alpha, \vp(w), \oc(w)\}$ (or, any color in $[4] \setminus \{\alpha, \vp(w)\}$ if $\oc(w)$ does not exist).
    \item \textbf{Rule 2:} \textit{$x$ is a $t_2$-neighbor of $u$.} Then $x$ is on a $2$-thread, $xy$ with anchors $u \in N(x)$ and $w \in N(y)$ for some $y, w \in V(G)$. Extend $\phi$ by setting $\phi(y)$ to any color in $[4] \setminus \{\alpha, \vp(w), \oc(w)\}$ (or $[4] \setminus \{\alpha, \vp(w)\}$ if $\oc(w)$ does not exist), then color $x$ using any color in $[4] \setminus \{\alpha, \vp(y), \vp(w)\}$.
    \item \textbf{Rule 3:} \textit{$x$ is a $t_3$-neighbor of $u$.} Then $x$ is on a $3$-thread $xyz$ with anchors $u \in N(x)$ and $w \in N(z)$ for some $y, z, w \in V(G)$.
    Extend $\phi$ by setting $\phi(z)$ to any color in $[4] \setminus \{\vp(w), \oc(w)\}$ (or $[4] \setminus \{\vp(w)\}$ if $\oc(w)$ does not exist), and then setting $\phi(y)$ to any color in $[4] \setminus \{\alpha, \vp(z), \vp(w)\}$. Then color $x$ any color in $[4] \setminus \{\alpha, \vp(y), \vp(z)\}$.
    \item \textbf{Rule 4:} \textit{$x$ is a $t_{\Worst}$-neighbor of $u$.} Let $y_1, y_2, w_1, w_2 \in V(G)$ witness $x$ is a $t_{\Worst}$-neighbor where $y_1, y_2 \in N(x)$, $y_1w_1, y_2w_2 \in E(G)$. Color $x$ using any color in $[4] \setminus \{\alpha, \vp(w_1), \vp(w_2)\}$. Then for each $i\in[2]$, color $y_{i}$ using any color in $[4] \setminus \{\vp(x), \vp(w_{i}), \oc(w_{i})\}$ (or $[4] \setminus \{\vp(x), \vp(w_{i})\}$ if $\oc(w_i)$ does not exist).
\end{itemize}
Note that as $P$ is $(u, \vp^0)$-flexible, it is $(u,\vp)$-flexible, and thus there are two extensions of $\vp$ to a proper coloring of $G$ by \cref{lem: two choices}. Note that $v$ has different colors in each coloring by \cref{lem: two choices}\ref{lem: two choices - 3} and so an odd color of $u$ exists under at least one of these extensions by \cref{lemma:three-and-two-choose-odd}. Thus we choose the extension of $\vp$ so that an odd color of $u$ exists. Now all of $G$ is colored, and every vertex has an odd color except possibly for neighbors $w$ of $u$ which anchor $(w, \phi)$-flexible threads. For each such neighbor $w$ of $u$, we choose a $(w, \phi)$-flexible thread $P_w$ and recolor $P_w$ by \cref{lem: two choices} so that an odd color of $w$ exists. Note that recoloring $P_w$ does not change any odd color of vertices other than $w$ since $G$ has girth at least $10$. The result is an odd $4$-coloring $\vp'$ of $G$. Since for each $(u,\vp^0)$-flexible thread $P$, we did not change the color or odd color of the other anchor $w \neq u$ of $P$, it follows any $(u, \vp_0)$-flexible thread of $G$ is also $(u, \phi')$-flexible.
\end{proof}

In order to avoid looking at all possible colorings of $G - S[u]$ and all the possible sets of forbidden colors, we give an upper and lower bound on the size of $\Forb(u,\vp)$ and $\Flex(u,\vp)$.

\begin{definition}\label{def:score}
    Let $G$ be a graph, and $u \in V(G)$ with $v \in N_G(u)$. Then $v$'s \defn{score relative to $u$}, denoted $s_u(v)$, is defined as 
\begin{equation*}
    s_u(v) = \begin{cases}
        0  & \text{if $v$ is a $t_2$-, $t_3$-, or $t_{\Worst}$-neighbor of $u$},\\
        1 & \text{if $v$ is a $t_1$-, $t_{\Good}$-, $t_{\Semibad}$-,  $t_{\Bad}$-, or $t_{\Odd}$-neighbor of $u$},\\
        2 & \text{if $v$ is a $t_{\Even}$-neighbor of $u$}.
    \end{cases}
\end{equation*}

\end{definition}

For simplicity, we write $s(v)$ when $u$ is clear from context.

\begin{definition}\label{def: forb}
Let $G$ be a graph. The \defn{forbidden number} of a vertex $u \in V(G)$ is given by \[\forb_G(u) = \sum_{v\in N_G(u)}s_u(v).\]
When $G$ is clear from context, we write $\forb(G)$.
\end{definition}

It is worth noting that the score of each neighbor of $u$ is precisely the number of colors which possibly contribute to $\Forb(u,\vp)$. Hence \[\forb(u) \geq \max\{|\Forb(u,\vp)|: \vp \text{ is an odd } 4\text{-coloring of } G- S[u]\}.\]

Since each $i$-thread contributes $i$ colors to $\Flex(u,\vp)$, it follows \[\flex(u)\leq \min\{|\Flex(u,\vp)|: \vp \text{ is an odd } 4\text{-coloring of $G-S[u]$}\}.\]

At this point let $G$ be a minimum counterexample (with respect to $|V(G)|$) to \cref{thm:odd10}.
We now prove a number of lemmas about the flexible and forbidden number of vertices of $G$.
\begin{corollary}\label{cor: flex > forb}
    $G$ contains no vertex $u$ with $\flex(u)>\forb(u)$.
\end{corollary}
\begin{proof}
    Suppose for contradiction that there exists $u \in V(G)$ with $\flex(u) > \forb(u)$. Let $\phi$ be an odd 4-coloring of $G - S[u]$. Then
    \[|\Flex(u, \phi)| \geq \flex(u) > \forb(u) \geq |\Forb(u, \phi)|.\] Applying \cref{lemma: flex > forb} shows an odd $4$-coloring of $G$ exists, contradicting that $G$ is a minimum counterexample to Theorem \ref{thm:odd10}.
\end{proof}

\begin{corollary}\label{cor:ti forb}
If $u \in V(G)$ has a $t_i$-neighbor, then $\forb(u) \ge i$ for $i\in \{1,2,3\}$. 
\end{corollary}
\begin{proof}
    Since $u$ has a $t_i$-neighbor, $\flex(u) \ge i$. By \cref{cor: flex > forb}, $\forb(u) \ge \flex(u) \ge i$.
\end{proof}

\begin{lemma}\label{lem:4-4}
Let $u_1$ and $u_2$ be adjacent $4^+$-vertices in $G$. If $u_1$ has a $t_i$-neighbor and $u_2$ has a $t_j$-neighbor where $i,j\in [3]$, then $\forb(u_1) + \forb(u_2) \ge i+j+2$.
\end{lemma}

\begin{proof}
Since $G$ has girth at least 10, it follows $S[u_1]$ and $S[u_2]$ are pairwise disjoint.
Let $P_i$ and $P_j$ be $i$- and $j$-threads anchored by $u_1$ and $u_2$ respectively.
By \cref{cor:ti forb}, $\forb_G(u_1) \ge i$ and $\forb_G(u_2) \ge j$. 
Suppose for contradiction that $\forb_G(u_1)+\forb_G(u_2) \le i+j+1$.
Then, without loss of generality, we may let $\forb_G(u_1) \le i+1$ and $\forb_G(u_2) \le j$.

Let $G_0 \defeq G - (S[u_1] \cup S[u_2])$, and let $G_1 \defeq G - S[u_2]$. Thus, $G_0 = G_1 - S[u_1]$.
Let $\vp^0$ be an odd 4-coloring of $G_0$. 
Note that as $u_2 \notin G_1$, it follows
\[|\Forb_{G_1}(u_1,\vp^0)|\leq \sum_{v\in N_{G_1}(u_1)}s(v) \leq \sum_{v\in N_G(u_1)\setminus\{u_2\}}s(v) = \sum_{v\in N_G(u_1)}s(v)  - 2 = \forb_G(u_1) - 2 \leq i-1.\]

Note that $\flex_{G_1}(u_1) \ge i$, and so  $|\Flex_{G_1}(u_1,\vp^0)|> |\Forb_{G_1}(u_1,\vp^0)|$. By \cref{lemma: flex > forb}, there exists an odd 4-coloring $\vp^1$ of $G_1$ where $P_i$ is $(u,\vp^1)$-flexible (in $G_1$ and thus in $G$).

Now, as $P_i$ is $(u,\vp^1)$-flexible in $G$, it follows $u_1$ contributes at most one color to $\Forb_G(u_2, \phi^1)$, i.e., $\phi^1_\mathsf{o}(u_1)\notin \Forb_G(u_2, \phi^1)$. Thus
% \begin{align*}
% |\Forb_G(u_2,\vp^1)| &\leq \sum_{v\in N_G(u_2)\setminus\{u_1\}}s(v) + 1.\\
%     &\leq \sum_{v\in N_G(u_2)}s(v) - 2 +1\\
%     &= \forb(u_2) - 2 +1 \leq j-1.
% \end{align*}
\[|\Forb_G(u_2,\vp^1)| \leq \sum_{v\in N_G(u_2)\setminus\{u_1\}}s(v) + 1 \leq \sum_{v\in N_G(u_2)}s(v) - 2 +1 = \forb(u_2) - 2 +1 \leq j-1.\]

Since $\flex_G(u_2) \ge j$, $|\Flex_G(u_2,\vp^1)| > |\Forb_G(u_2,\vp^1)|$. Thus by \cref{lemma: flex > forb}, there exists an odd 4-coloring $\phi_2$ of $G$, contradicting that $G$ is a minimum counterexample.
\end{proof}

\begin{lemma}\label{lem:4-4-4}
Let $uvw$ be a $3$-path of $4^+$-vertices in $G$.
If $u$ has a $t_i$-neighbor, $v$ has a $t_j$-neighbor, and $w$ has a $t_k$-neighbor where $i,j,k \in [3]$,
then $\forb(u) + \forb(v) + \forb(w) \ge i+j+k+4$.
\end{lemma}
  
\begin{proof}

Since $G$ has girth at least 10, it follows $S[u], S[v], S[w]$ are pairwise disjoint. 
By \cref{cor:ti forb}, $\forb(u) \ge i$, $\forb(v) \ge j$, and $\forb(w) \ge k$.
By \cref{lem:4-4}, $\forb(u) + \forb(v) \ge i+j+2$ and $\forb(v)+\forb(w) \ge j+k+2$.
Suppose for contradiction that $\forb(u) + \forb(v) + \forb(w) \le i+j+k+3$.
Then $\forb(u) \le i+1$, $\forb(v) \leq j+3$, and $\forb(w) \le k+1$.
Thus $\forb(u) \in \{i, i+1\}, \forb(w) \in \{k, k+1\}$ so that without loss of generality, there are three cases.

Suppose $\forb(u)  = i$, $\forb(v) \leq j + 3$, and $\forb(w) = k$. Let $G_0 \defeq G-(S[u] \cup S[v] \cup S[w])$, $G_1 \defeq G - (S[u] \cup S[w])$, $G_2 \defeq G - S[w]$. Let $\vp^0$ be an odd $4$-coloring of $G_0$. Since $u, w \notin G_1$, we have
% \begin{align*}
%     |\Forb_{G_1}(v, \vp^0)| &\leq \sum_{x \in N_{G_1}} s(x)\\
%     &\leq \sum_{x \in N_G(v) \setminus \{u, w\}} s(x)\\
%     & \leq \forb_G(v) - 4\\
%     & \leq j-1.
% \end{align*}
\[|\Forb_{G_1}(v, \vp^0)| \leq \sum_{x \in N_{G_1}} s(x) \leq \sum_{x \in N_G(v) \setminus \{u, w\}} s(x) \leq \forb_G(v) - 4 \leq j-1.\]

Since $\flex_{G_1}(v) \ge j$, it follows $|\Flex_{G_1}(v, \vp^0)| > |\Forb_{G_1}(v, \vp^0)|$. Thus by \cref{lemma: flex > forb}, there exists an odd 4-coloring $\phi^1$ of $G_1$ such that $v$ anchors a $(v, \phi^1)$-flexible thread, say $P$. Note then that $P$ is $(v, \phi^1)$-flexible in $G_2$ as well. 

Thus as $v$ anchors a flexible thread in $G_2$, it follows that $v$ contributes at most one color to $\Forb_{G_2}(u, \vp^1)$. Thus
% \begin{align*}
%     |\Forb_{G_2}(u, \phi^1)| &\leq |\Forb_{G_2}(u, \phi^0)| + 1 \\
%     &\leq \sum_{x \in N(u) \setminus \{v\}} s(x)\\
%     & \leq (\forb_G(u) -2) + 1 \\
%     & \leq i - 1. 
% \end{align*}
\[|\Forb_{G_2}(u, \phi^1)| \leq |\Forb_{G_2}(u, \phi^0)| + 1 \leq \sum_{x \in N(u) \setminus \{v\}} s(x) \leq (\forb_G(u) -2) + 1  \leq i - 1. \]
Since $\flex_{G_2}(u) \ge i$, it follows $|\Flex_{G_2}(u, \phi^1)| > |\Forb_{G_2}(u, \phi^1)|$. By \cref{lemma: flex > forb}, there exists an odd 4-coloring $\phi^2$ of $G_2$. Note that $P$ is a $(v, \phi^2)$-flexible thread $G$.

Now as $v$ anchors a $(v, \phi^2)$-flexible thread of $G$, it follows 
% \begin{align*}
%     |\Forb_G(w, \phi^2)| &\leq |\Forb_G(w, \vp^0)| + 1\\
%     &\leq \sum_{x \in N(w) \setminus \{v\}} s(x)\\
%     & \leq (\forb_G(w) -2) + 1 \\
%     & \leq k - 1. 
% \end{align*}
\[|\Forb_G(w, \phi^2)| \leq |\Forb_G(w, \vp^0)| + 1\leq \sum_{x \in N(w) \setminus \{v\}} s(x) \leq (\forb_G(w) -2) + 1  \leq k - 1. \]
Since $\flex_G(w) \ge k$, it follows $|\Flex_G(w, \phi^2)| > |\Forb_G(w, \phi^2)|$. Thus by \cref{lemma: flex > forb}, there exists an odd 4-coloring $\phi^3$ of $G$, thereby contradicting that $G$ is a minimum counter example.

The other two cases uses a similar argument, except we change the order in which we color. Namely,
if $\forb_G(u) = i+1, \forb_G(v) \le j+1$, $\forb_G(w) = k+1$, then we take $G_1 \defeq G - (S[w] \cup S[v])$, $G_2 \defeq G - S[v]$.
If $\forb_G(u) = i+1, \forb_G(v) \le j+2$, $\forb_G(w) = k$, we take $G_1 \defeq G - (S[v] \cup S[w]$) and $G_2 \defeq G - S[w]$.
\end{proof}

\begin{lemma}\label{lem:4-bad3-4}
Let $u_1vu_2$ be a $3$-path, where $u_1$ and $u_2$ are $4$-vertices, and $v$ is a semi-bad $3$-vertex in $G$.
If both $u_1$ and $u_2$ have a $t_3$-neighbor, then $\forb(u_1) \ge 3$ and $\forb(u_2) \ge 4$ (or vice-versa). 
\end{lemma}
\begin{proof}
Since $G$ has girth at least $10$, it follows $S[u_1]$, $S[v]$, $S[u_2]$ are pairwise disjoint.
By \cref{cor:ti forb}, $\forb(u_i) \ge 3$ for each $i\in[2]$. 
Suppose for contradiction that $\forb(u_1)=\forb(u_2)=3$.

Let $G_0 \defeq G-(S[{u_1}] \cup S[v] \cup S[{u_2}])$, $G_1 \defeq G - (S[v] \cup S[u_2])$, and $G_2 \defeq G - S[v]$.
Let $\vp^0$ be an odd $4$-coloring of $G_0$.
Since $v \notin G_0$, we have
% \begin{align*}
% |\Forb_{G_1}(u_1, \vp^0)| &\leq \sum_{x \in N_{G_1}(u_1)} s(x)\\
% & \leq \sum_{x \in N_{G}(u_1) \setminus \{v\}} s(x)\\
%     & \leq \forb_G(u_1) - 1\\
%     & = 2.
% \end{align*}
\[|\Forb_{G_1}(u_1, \vp^0)| \leq \sum_{x \in N_{G_1}(u_1)} s(x) \leq \sum_{x \in N_{G}(u_1) \setminus \{v\}} s(x) \leq \forb_G(u_1) - 1 = 2.\]
Thus $\flex_{G_1}(u_1)\ge 3$, and so $|\Flex_{G_1}(u_1, \vp^0)| > |\Forb_{G_1}(u_1, \vp^0)|$. By \cref{lemma: flex > forb}, there exists an odd 4-coloring $\phi^1$ of $G_1$ such that $u_1$ anchors a $(u, \phi^1)$-flexible thread, say $P$. Note that $P$ is $(u, \phi^1)$-flexible in $G$ as well.

Now as $v \notin G_2$, it follows
% \begin{align*}
% |\Forb{G_2}(u_2, \phi^1)| &\leq \sum_{x \in N_{G}(u_2) \setminus \{v\}} s(x)\\
%     & \leq \forb_G(u_2) - 1\\
%     & = 2,
% \end{align*}
\[|\Forb_{G_2}(u_2, \phi^1)| \leq \sum_{x \in N_{G}(u_2) \setminus \{v\}} s(x) \leq \forb_G(u_2) - 1 = 2,\]
and thus since $\flex_{G_2}(u_2)\ge 3$, it follows $|\Flex(u_2, \phi^1)| > |\Forb(u_2, \phi^1)|$. By \cref{lemma: flex > forb}, there exists an odd 4-coloring $\phi^2$ of $G_2$ such that $u_2$ has a $(u_2, \phi^2)$-flexible thread, say $Q$. Note that $Q$ is also $(u_2, \phi^2)$-flexible in $G$.

Let $w$ be the $2$-neighbor of $v$, and $w'$ be the anchor of $w$ not equal to $v$. Extend $\phi^2$ by coloring $v$ an arbitrary color from $[4] \setminus \{\phi^2(u_1), \phi^2(u_2), \phi^2(w')\}.$ Since $w$ isn't colored by $\vp^2$, the resulting coloring is proper. Now color $w$ an arbitrary color from $[4] \setminus \{\phi^2(w'), \phi_\mathsf{o}(w'), \phi^2(v)\}$. The resulting coloring (say $\phi^3$) is proper. At this point, all vertices colored by $\vp^3$ have an odd color except possibly $u_1$ and $u_2$. However, note that $P$ and $Q$ are $(u_1, \phi^3)$ and $(u_2, \phi^3)$-flexible, respectively, in $G$. Thus, if needed, we may recolor $P$ and $Q$ so that $u_1$ and $u_2$ have an odd color. The result is an odd 4-coloring of $G$, a contradiction.
Therefore, it cannot be the case that both $\forb_G(u_1) = 3$ and $\forb_G(u_2) = 3$. It therefore follows either $\forb_G(u_1) \geq 3$ and $\forb_G(u_2) \geq 4$, or vice-versa.
\end{proof}

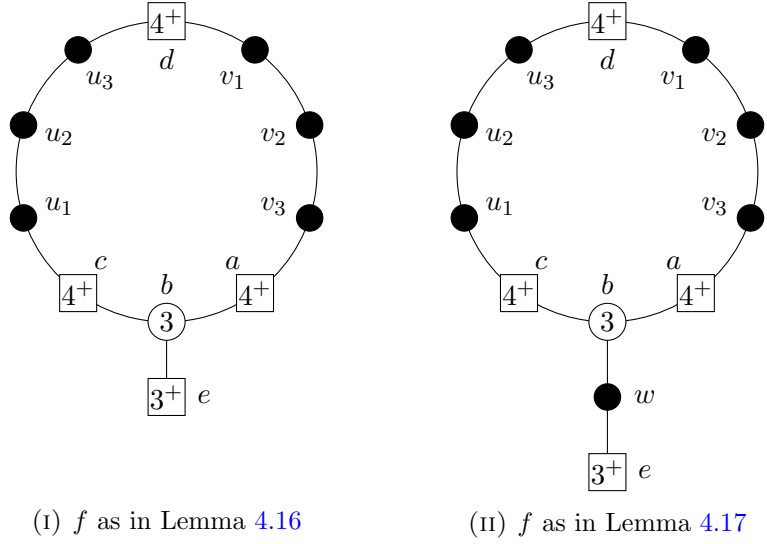
\begin{figure}[!h]
    \centering
    \begin{subfigure}[t]{0.4\textwidth}
        \centering
        \begin{adjustbox}{valign=t}
            \begin{tikzpicture}
\draw (0,0) circle (2);
\draw (0, -2) -- (0, -3);
\draw[fill=white] (270:2) circle (7pt);
\node at (270:1.5) {$b$};
\draw[fill=white] (270+36:2) +(-7pt,-7pt) rectangle +(7pt,7pt);
\node at (270+1*36:1.5) {$a$};
\draw[fill=black] (270+2*36:2) circle (5pt);
\node at (270+2*36:1.5) {$v_3$};
\draw[fill=black] (270+3*36:2) circle (5pt);
\node at (270+3*36:1.5) {$v_2$};
\draw[fill=black] (270+4*36:2) circle (5pt);
\node at (270+4*36:1.5) {$v_1$};
\draw[fill=white] (270+5*36:2) +(-7pt,-7pt) rectangle +(7pt,7pt);
\node at (270+5*36:1.5) {$d$};
\draw[fill=black] (270+6*36:2) circle (5pt);
\node at (270+6*36:1.5) {$u_3$};
\draw[fill=black] (270+7*36:2) circle (5pt);
\node at (270+7*36:1.5) {$u_2$};
\draw[fill=black] (270+8*36:2) circle (5pt);
\node at (270+8*36:1.5) {$u_1$};
\draw[fill=white] (270+9*36:2) +(-7pt,-7pt) rectangle +(7pt,7pt);
\node at (270+9*36:1.5) {$c$};
\draw[fill=white] (270:3) +(-7pt,-7pt) rectangle +(7pt,7pt);
\node at (0.5,-3) {$e$};
\node at (270+5*36:2) {$4^+$};
\node at (270+9*36:2) {$4^+$};
\node at (270+11*36:2) {$4^+$};
\node at (270+10*36:2) {$3$};
\node at (0,-3) {$3^+$};
\node at (0,-4.125) {};

\end{tikzpicture}
        \end{adjustbox}
        \caption{$f$ as in Lemma \ref{obs:a_4a_4a_1a_1good3}}
        \label{fig: a_4 a_4 a_1 a_1 - good3}
    \end{subfigure}
    \hspace{-1cm}
    \begin{subfigure}[t]{0.4\textwidth}
        \centering
        \begin{adjustbox}{valign=t}
            \begin{tikzpicture}
\draw (0,0) circle (2);
\draw (0, -2) -- (0, -4);
\draw[fill=white] (270:2) circle (7pt);
\node at (270:1.5) {$b$};
\draw[fill=white] (270+36:2) +(-7pt,-7pt) rectangle +(7pt,7pt);
\node at (270+1*36:1.5) {$a$};
\draw[fill=black] (270+2*36:2) circle (5pt);
\node at (270+2*36:1.5) {$v_3$};
\draw[fill=black] (270+3*36:2) circle (5pt);
\node at (270+3*36:1.5) {$v_2$};
\draw[fill=black] (270+4*36:2) circle (5pt);
\node at (270+4*36:1.5) {$v_1$};
\draw[fill=white] (270+5*36:2) +(-7pt,-7pt) rectangle +(7pt,7pt);
\node at (270+5*36:1.5) {$d$};
\draw[fill=black] (270+6*36:2) circle (5pt);
\node at (270+6*36:1.5) {$u_3$};
\draw[fill=black] (270+7*36:2) circle (5pt);
\node at (270+7*36:1.5) {$u_2$};
\draw[fill=black] (270+8*36:2) circle (5pt);
\node at (270+8*36:1.5) {$u_1$};
\draw[fill=white] (270+9*36:2) +(-7pt,-7pt) rectangle +(7pt,7pt);
\node at (270+9*36:1.5) {$c$};
\draw[fill=black] (270:3) circle (5pt);
\node at (0.5,-3) {$w$};
\draw[fill=white] (270:4) +(-7pt,-7pt) rectangle +(7pt,7pt);
\node at (0.5,-4) {$e$};
\node at (270+5*36:2) {$4^+$};
\node at (270+9*36:2) {$4^+$};
\node at (270+11*36:2) {$4^+$};
\node at (270+10*36:2) {$3$};
\node at (0,-4) {$3^+$};
\end{tikzpicture}
        \end{adjustbox}
        \caption{$f$ as in Lemma \ref{obs:a_4a_4a_1a_1semibad3}}
        \label{fig: a_4 a_4 a_1 a_1 - sbad3}
    \end{subfigure}

    \caption{Different cases of the array representation $a_4 a_4 a_1 a_1$ as considered in Lemmas \ref{obs:a_4a_4a_1a_1good3} and \ref{obs:a_4a_4a_1a_1semibad3}.
     Note that in (II), $e$ is a $3^+$-vertex by  \cref{obs:odd}\ref{obs:odd - no odd vertex adjacent to a 2+ thread}.}
     \label{fig: all-perm a_4 a_4 a_1 a_1}
\end{figure}

\begin{lemma}\label{obs:a_4a_4a_1a_1good3}
Let $f$ be a face of $G$ with array representation $a_4a_4a_1a_1$. Suppose further the boundary of $f$ contains a $3$-vertex, say $b$, with neighbors $a, c$ also on the boundary of $f$. If $\forb(a) = \forb(c) = 3$, then both $a$ and $c$ each have at most one $t_3$-neighbor.
\end{lemma}

\begin{proof}
Let $f$ be a face with array representation $a_4a_4a_1a_1$. Let the vertices on $f$ be $c$, $u_1$, $u_2$, $u_3$, $d$, $v_1$, $v_2$, $v_3$, $a$, $b$ with degree sequence $4^+ - 2 - 2 - 2 -4^+ - 2 - 2 -2 -4^+ - 3$. Let $e$ be the neighbor of $b$ not on the boundary of $f$. See \cref{fig: all-perm a_4 a_4 a_1 a_1}\hyperref[fig: a_4 a_4 a_1 a_1 - good3]{(I)} for an illustration. Now suppose for contradiction that $\forb(a) = \forb(c) = 3$ and both $a$ and $c$ are adjacent to at least two $t_3$-neighbors. %(the argument for $c$ is symmetric).

As $G$ has girth at least $10$, it follows $S[a]$, $S[c]$ are disjoint. Let $G_0 \defeq G - (S[a]\cup S[c]\cup \{b\})$. Let $G_1 \defeq G - (S[c] \cup \{b\})$.
By the minimality of $G$, $G_0$ has an odd $4$-coloring $\vp^0$.

Since $b \notin G_0$, it follows
% \begin{align*}
% |\Forb_{G_1}(a, \phi^0)| &\leq \sum_{x \in N_G(a) \setminus \{b\}} s(x)\\
%     & \leq \forb_G(a) - 1\\
%     & = 2.
% \end{align*}
\[|\Forb_{G_1}(a, \phi^0)| \leq \sum_{x \in N_G(a) \setminus \{b\}} s(x) \leq \forb_G(a) - 1 = 2.\]
Thus since $\flex_{G_1}(a)\ge 3$, it follows $|\Flex_{G_1}(a, \phi^0)| > |\Forb_{G_1}(a,\phi^0)|$, therefore, there exists an odd 4-coloring $\phi^1$ of $G_1$ such that $a$ has a $(a, \phi^1)$-flexible thread by \cref{lemma: flex > forb}; furthermore, as $a$ contains two $t_3$-neighbors in $G$ (and thus in $G_1$), we may assume $P$ is the thread anchored by $a$ that is \textit{not} on the boundary of the face $f$ (we may assume this as this thread contributes 3 colors to $\Flex_{G_1}(a, \phi^0)$). Now let $\alpha \defeq \phi^1(a)$. We extend $\phi^1$ as follows: set $\phi^1(b)$ to be any color $\beta \in [4] \setminus \{\alpha, \vp(e), \oc(e)\}$. 
Now color $c$ any color in $[4]\setminus \Forb_G(c,\vp^2)$. Since $\forb_G(c) = 3$, it follows $|\Forb_G(c,\vp^2)| \leq 3$ and thus the coloring remains proper. Next, by following Rule 1 - Rule 4 in \cref{lemma: flex > forb}, we extend the coloring to include $S[c]\setminus \{u_1, u_2, u_3\}$. At this point, $\dom(\vp^1)=V(G)\setminus\{u_1, u_2, u_3, v_1, v_2, v_3\}$, and every vertex except possibly $c, u_2, v_2, a$ have an odd color. 

If $\phi^2_\mathsf{o}(c)$ exists, then color $u_1$ a color in $[4]\setminus\{\phi^2(c), \phi^2_{\mathsf{o}}(c)\}$; otherwise, color $u_1$ a color not equal to $\phi^2(c)$ and note this is an odd color for $c$. Now color $u_2$ a color in $[4]\setminus\{\phi^2(u_1), \phi^2(c)\}$, and then color $u_3$ a color in $[4]\setminus\{\phi^2(u_2), \phi^2(u_1), \phi^2(d)\}$. Now color $v_1$ a color in $[4]\setminus\{\phi^2(d), \phi^2_\mathsf{o}(d)\}$, and color $v_2$ in $[4]\setminus\{\phi^2(v_1), \phi^2(d)\}$, then color $v_3$ a color in $[4]\setminus\phi^2(v_1), \phi^2(v_2), \phi^2(a)$. At this point, all of $G$ is colored and every vertex except possibly $a$ has an odd color. However, $P$ remains $(a, \phi^2)$-flexible. Thus we may recolor $P$ if needed so that $a$ has an odd-color\footnote{Note that the anchor of $P$ not equal to $a$ is not an isolated vertex in $G_0$. Thus it has an odd color that is not coming from a vertex on $P$ or any thread of $c$. Therefore, even if this anchor was an anchor of a thread of $c$, it remains $(a, \phi^2)$-flexible.}. This shows $G$ is odd 4-colorable, a contradiction. Thus $a$ has at most one $t_3$-neighbor.
\end{proof}

\begin{lemma}\label{obs:a_4a_4a_1a_1semibad3}
Let $f$ be a face of $G$ with array representation $a_4a_4a_1a_1$. Further suppose $f$ contains a semi-bad $3$-vertex, say $b$, with neighbors $a$ and $c$ also on the boundary of $f$. If $\forb(a)+\forb(c)\leq 7$ and $\forb(a)\leq \forb(c)$, then $a$ is adjacent to at most one $t_3$-neighbor.
\end{lemma}

\begin{proof}
Let $f$ be a face with array representation $a_4a_4a_1a_1$. Let the vertices on $f$ be $c$, $u_1$, $u_2$, $u_3$, $d$, $v_1$, $v_2$, $v_3$, $a$, $b$ with degree sequence $4^+ - 2 - 2 - 2 -4^+ - 2 - 2 -2 -4^+ - 3$. Let $w$ be the $2$-neighbor of $b$ not on $f$, and let $e$ be the neighbor of $w$ that is not $b$. See \cref{fig: all-perm a_4 a_4 a_1 a_1}\hyperref[fig: a_4 a_4 a_1 a_1 - sbad3]{(II)} for an illustration. Now suppose for contradiction, $\forb(a) + \forb(c) \leq 7$ and $a$ is adjacent to at least two $t_3$-neighbors. By \cref{lem:4-bad3-4} and the assumptions that $\forb(a) +\forb(c) \le 7$ and $\forb(a) \le \forb(c)$, we obtain $\forb(a)=3$ and $\forb(c)=4$. 

Since $G$ has girth 10, $S[a]$ and $S[c]$ are disjoint.
Let $G_0 \defeq G - (S[a]\cup S[c]\cup \{b, w\})$.
By the minimality of $G$, $G_0$ has an odd $4$-coloring $\vp^0$. Let $G_1 \defeq G_0 \cup S[a]$.

Since $b \notin G_1$ it follows
% \begin{align*}
% |\Forb_{G_1}(a, \phi^0)| &\leq \sum_{x \in N_{G}(a) \setminus \{b\}} s(x)\\
%     & \leq \forb_G(a) - 1\\
%     & = 2.
% \end{align*}
\[|\Forb_{G_1}(a, \phi^0)| \leq \sum_{x \in N_{G}(a) \setminus \{b\}} s(x) \leq \forb_G(a) - 1 = 2.\]
Thus since $\flex_{G_1}(a)\ge 3$, it follows $|\Flex_{G_1}(a, \phi^0)|> |\Forb_{G_1}(a,\phi^0)|$. Thus by \cref{lemma: flex > forb}, there exists an odd 4-coloring $\phi^1$ of $G_1$ such that $a$ anchors a $(a, \phi^1)$-flexible thread, say $P$; furthermore, as $a$ contains two $t_3$ neighbor in $G$ (and thus in $G_1$), we may assume $P$ is the thread anchored by $a$ that is not on the boundary of the face $f$ (we may assume this as this thread contributes 3 colors to $\Flex_{G_1}(a, \phi^0)$.) We then extend $\vp^1$ as follows: since $b$ is uncolored by $\vp^1$, it follows $b$ contributes at most 1 to $\Forb_{G}(c, \phi^1)$. Thus 
% \begin{align*}
%     |\Forb_{G}(c, \phi^2)| &\leq \sum_{x \in N_{G}(c) \setminus \{b\}} s(x) + 1\\
%     &=\sum_{x \in N_{G}(c)} s(x) - 2 + 1\\
%     &= \forb_G(c) - 1\\
%     & = 3.
% \end{align*} 
\[|\Forb_{G}(c, \phi^1)| \leq \sum_{x \in N_{G}(c) \setminus \{b\}} s(x) + 1=\sum_{x \in N_{G}(c)} s(x) - 2 + 1= \forb_G(c) - 1 = 3.\]
Thus we may properly color $c$ a color in $ [4]\setminus \Forb_G(c,\vp^1)$. Now by following the coloring rules in \cref{lemma: flex > forb}, we extend $\vp^1$ to $\vp^2$ to include all of $S[c]\setminus \{u_1, u_2, u_3\}$.

Next, set $\vp^2(b)\in [4]\setminus \{\vp^2(c), \vp^2(a), \vp^2(e)\}$, then set $\vp^2(w)\in [4]\setminus \{\vp^1(b), \vp^1(e), \oc^1(e)\}$. At this point the coloring includes $V(G) \setminus \{u_1, u_2, u_3, v_1, v_2, v_3\}$ and every vertex except possibly $\{c, u_2, v_2, a\}$ have an odd color.

The proof then follows exactly as in the last paragraph of \cref{obs:a_4a_4a_1a_1good3}.
\end{proof}

\section{Proof of Theorem~\ref{thm:odd10}}\label{sec:odd10}

We note that we use terminology defined in \S\ref{sec: prelims} and \S\ref{sec: forb/flex}. In this section, we prove the following:
\begin{theorem*}[\ref{thm:odd10}]
    If $G$ is a planar graph with girth at least $10$, then $\odd(G) \le 4$.
\end{theorem*}

\begin{proof}
Suppose for contradiction, the theorem is false. 
Among all counterexamples, choose a plane graph $G$ such that $|V(G)|$ is minimized. Then $G$ is connected and, as mentioned in the introduction, the results of ~\cite{cho2023odd} imply the girth of $G$ is exactly 10.

We let $\mu : V(G) \cup F(G) \to \mathbb{R}$ defined by 
\[
    \mu(x) = \begin{cases}
        2\deg(x) - 6, & x \in V(G),\\
        \ell(x) - 6, & x \in F(G)
    \end{cases}
\]
be the initial assignment of charge to the vertices and faces of $G$. By Euler's formula, \[\sum_{x \in V(G)\cup F(G)}\mu(x) = -12.\] 

We redistribute the charges as follows:

\vspace{1em}
\begin{mdframed}

\emph{Discharging Rules}
\vspace{8pt}
\begin{enumerate}[label = {(V\arabic*)}, itemsep = 5pt]
    \item\label{rule V 3thread}\label{discharging rule start} Each $4^+$-vertex sends charge $\frac{5}{6}$ to its $t_3$-neighbors.
    \item\label{rule V 2thread} Each $4^+$-vertex sends charge $\frac{2}{3}$ to its $t_2$-neighbors.
    \item\label{rule V 1thread} Each $4^+$-vertex sends charge $\frac{1}{3}$ to its $t_1$-neighbors.
    \item\label{rule V path 4-3-2} Each $4^+$-vertex sends charge $\frac{1}{5}$ each to the $2$-neighbors of its $t_{\Worst}$- and $t_{\Bad}$-neighbors.
    \item\label{rule V semibad} Each $4^+$-vertex sends charge $\frac{1}{10}$ to the $2$-neighbors of its $t_{\Semibad}$-neighbors.
\end{enumerate}
\vspace{8pt}
\begin{enumerate}[label = {(F\arabic*)}, itemsep = 5pt]
    \item\label{rule Fgood} Each face sends charge $\frac{2}{3}$ to incident good 2-vertices (recall  Definition \ref{def: good, bad, worst}).
    \item\label{Rule Fbad} Each face sends charge $\frac{11}{15}$ to incident bad 2-vertices.
    \item\label{Rule Fworst} Each face sends charge $\frac{4}{5}$ to incident \worst~ 2-vertices.
    \item\label{Rule F 2thread} Each face sends charge $\frac{2}{3}$ to incident vertices in a $2$-thread.
    \item\label{Rule F 3thread} Each face sends charge $\frac{7}{12}$ to the endpoints of a $3$-thread and charge $1$ to the middle vertex.
    \label{discharging rule finish}
\end{enumerate}
\vspace{8pt}
\begin{enumerate}[label = {(R\arabic*)}]
\item\label{Rule extra} 
After all previous rules are applied, a $4^+$-vertex $v$ divides its remaining charge evenly among its incident poor faces. 
\end{enumerate}
\end{mdframed}

\vspace{1em}

More precisely, let $\mu': V(G) \cup F(G) \to \mathbb{R}$ be the charge obtained by applying \ref{rule V 3thread} through \ref{rule V semibad} and \ref{rule Fgood} through \ref{Rule F 3thread}. We let $p(v)$ be the number of poor faces that $v$ is incident to (we count a face only once, even if $v$ appears multiple times on a boundary walk of $f$.)
Then, if $f$ is poor and incident to $v$, by \ref{Rule extra}, we send charge $\frac{\mu'(v)}{p(v)}$ from $v$ to $f$. Let $\mu'': V(G) \cup F(G) \to \mathbb{R}
$ be the final charge after applying \ref{Rule extra}.

Figure \ref{fig:vertex-discharging} and \ref{fig:odd-face-discharging} below illustrate the vertex and face discharging rules.

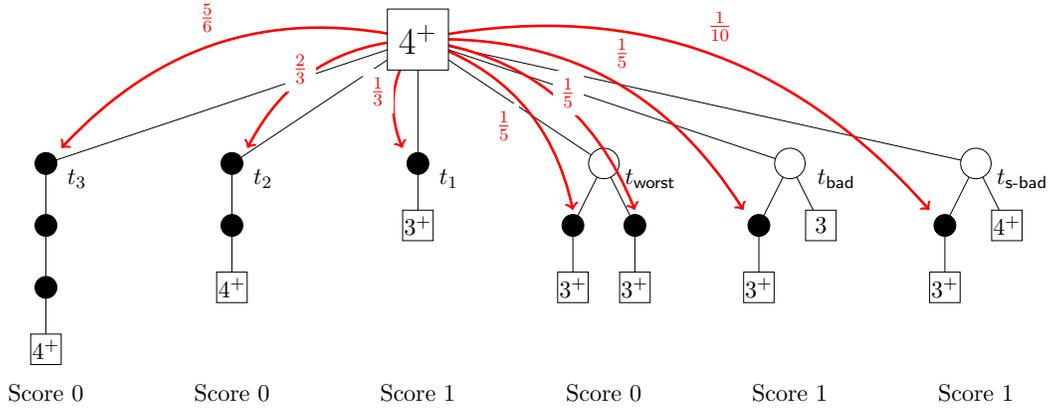
\begin{figure}[tbh!]
    \centering
    \resizebox{.85\textwidth}{!}{\begin{tikzpicture}

\draw (0,0) -- (-6,-2) -- (-6,-5);
\draw (0,0) -- (-3,-2) -- (-3,-4);
\draw (0,0) -- (0,-2) -- (0,-3);
\draw (0,0) -- (6,-2);
\draw (6,-2) -- (5.5,-3) -- (5.5,-4);
\draw (6,-2) -- (6.5,-3);
\draw (0,0) -- (3,-2);
\draw (3,-2) -- (2.5,-3) -- (2.5,-4);
\draw (3,-2) -- (3.5,-3) -- (3.5,-4);
\draw (0,0) -- (9,-2);
\draw (9,-2) -- (8.5,-3) -- (8.5,-4);
\draw (9,-2) -- (9.5,-3);

% \draw[red, dashed] (-9.5, 1) -- (1, 1) -- (1, -3.5) -- (-2.5,-3.5) -- (-2.5, -2.5) -- (-5.5, -2.5) -- (-5.5, -3.5) -- (-7.5, -3.5) -- (-7.5, -4.5) -- (-9.5, -4.5) -- cycle;

\draw[fill=black] (-6,-2) circle (5pt);
\draw[fill=black] (-6,-3) circle (5pt);
\draw[fill=black] (-6,-4) circle (5pt);
\draw[fill=white] (-6,-5) +(-7pt, -7pt) rectangle +(7pt, 7pt);
\node at (-6,-5) {$4^+$};
\node at (-5.5,-2.25) {$t_3$};

\draw[fill=black] (-3,-2) circle (5pt);
\draw[fill=black] (-3,-3) circle (5pt);
\draw[fill=white] (-3,-4) +(-7pt, -7pt) rectangle +(7pt, 7pt);
\node at (-3,-4) {$4^+$};
\node at (-2.5,-2.25) {$t_2$};

\draw[fill=black] (0,-2) circle (5pt);
\draw[fill=white] (0,-3) +(-7pt, -7pt) rectangle +(7pt, 7pt);
\node at (0,-3) {$3^+$};
\node at (.5,-2.25) {$t_1$};

\draw[fill=white] (3,-2) circle (7pt);
\draw[fill=black] (2.5,-3) circle (5pt);
\draw[fill=white] (2.5,-4) +(-7pt, -7pt) rectangle +(7pt, 7pt);
\draw[fill=black] (3.5,-3) circle (5pt);
\draw[fill=white] (3.5,-4) +(-7pt, -7pt) rectangle +(7pt, 7pt);
\node at (2.5,-4) {$3^+$};
\node at (3.5,-4) {$3^+$};
\node at (3.75, -2.25) {$t_{\Worst}$};

\draw[fill=white] (6,-2) circle (7pt);
\draw[fill=black] (5.5,-3) circle (5pt);
\draw[fill=white] (5.5,-4) +(-7pt, -7pt) rectangle +(7pt, 7pt);
\draw[fill=white] (6.5,-3) +(-7pt, -7pt) rectangle +(7pt, 7pt);
\node at (5.5,-4) {$3^+$};
\node at (6.5,-3) {$3$};
\node at (6.75, -2.25) {$t_{\Bad}$};

\draw[fill=white] (9,-2) circle (7pt);
\draw[fill=black] (8.5,-3) circle (5pt);
\draw[fill=white] (8.5,-4) +(-7pt, -7pt) rectangle +(7pt, 7pt);
\draw[fill=white] (9.5,-3) +(-7pt, -7pt) rectangle +(7pt, 7pt);
\node at (8.5,-4) {$3^+$};
\node at (9.5,-3) {$4^+$};
\node at (9.75, -2.25) {$t_{\Semibad}$};

\draw[->, red, very thick] (0,0) to[bend right] node [midway, above left, fill=white] {$\frac{5}{6}$} (-6+0.25, -2+0.25);
\draw[->, red, very thick] (0,0) to[bend right] node [midway, left, fill=white] {$\frac{2}{3}$} (-3+0.25, -2+0.25);
\draw[->, red, very thick] (0,0) to[bend right] node [midway, left, fill=white] {$\frac{1}{3}$} (0-0.25, -2+0.25);
\draw[->, red, very thick] (0,0) to[bend left] node [midway, right, fill=white] {$\frac{1}{5}$} (3.5, -3+0.25);
\draw[->, red, very thick] (0,0) to[bend left] node [midway, below left, fill=white] {$\frac{1}{5}$} (2.5, -3+0.25);
\draw[->, red, very thick] (0,0) to[bend left] node [midway, above right, fill=white] {$\frac{1}{5}$} (5.5-0.25, -3+0.25);
\draw[->, red, very thick] (0,0) to[bend left] node [midway, above right, fill=white] {$\frac{1}{10}$} (8.5-0.25, -3+0.25);

\draw[fill=white] (0,0) +(-14pt, -14pt) rectangle +(14pt, 14pt);
\node at (0,0) {\LARGE $4^+$};

%%%%%%%% If we hate this we can delete it %%%%%%%%%%%%%%
\node at (-6,-5.7){Score $0$};
\node at (-3,-5.7){Score $0$};
\node at (0,-5.7){Score $1$};
\node at (3,-5.7){Score $0$};
\node at (6,-5.7){Score $1$};
\node at (9,-5.7){Score $1$};
\end{tikzpicture}}
    \caption{Vertex discharging rules \ref{rule V 3thread} - \ref{rule V semibad} and the associated scores of a neighbor.}
    \label{fig:vertex-discharging}
\end{figure}

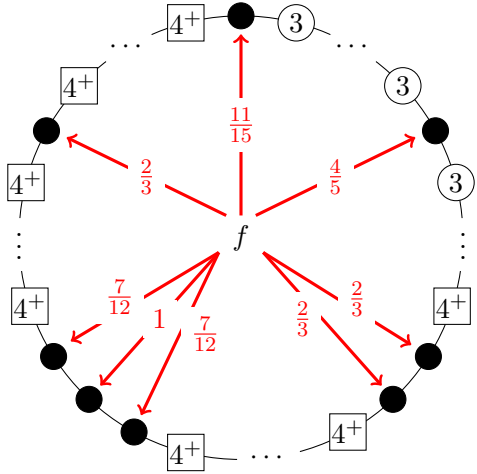
\begin{figure}[tbh!]
    \centering
    \resizebox{.4\textwidth}{!}{\begin{tikzpicture}
\draw (0,0) circle (3);
\draw[fill=white, draw=white] (270+0.2*36:3) circle (12pt);
\draw[fill=white, draw=white] (270+2.45*36:3) circle (10pt);
\draw[fill=white, draw=white] (270+4.15*36:3) circle (10pt);
\draw[fill=white, draw=white] (270+5.85*36:3) circle (10pt);
\draw[fill=white, draw=white] (270+7.55*36:3) circle (10pt);
%%%%%%%%%%%%%%%%%%%%%%%%%%%%%%%%%
%Vertex Placement%
%%%%%%%%%%%%%%%%%%%%%%%%%%%%%%%%%
%Good 1-thread%
\draw[fill=white] (270+6.3*36:3) +(-7pt,-7pt) rectangle +(7pt,7pt);
\draw[fill=white] (270+7.1*36:3) +(-7pt,-7pt) rectangle +(7pt,7pt);
\draw[fill=black] (270+6.7*36:3) circle (5pt);

%Bad 1-thread%
\draw[fill=white] (270+5.4*36:3) +(-7pt,-7pt) rectangle +(7pt,7pt);
\draw[fill=white] (270+4.6*36:3) circle (7pt);
\draw[fill=black] (270+5*36:3) circle (5pt);

%Worst 1-thread%
\draw[fill=white] (270+3.7*36:3) circle (7pt);
\draw[fill=white] (270+2.9*36:3) circle (7pt);
\draw[fill=black] (270+3.3*36:3) circle (5pt);

%2-thread$
\draw[fill=white] (270+2*36:3) +(-7pt,-7pt) rectangle +(7pt,7pt);
\draw[fill=white] (270+0.8*36:3) +(-7pt,-7pt) rectangle +(7pt,7pt);
\draw[fill=black] (270+1.2*36:3) circle (5pt);
\draw[fill=black] (270+1.6*36:3) circle (5pt);

%3-thread%
\draw[fill=white] (270+8*36:3) +(-7pt,-7pt) rectangle +(7pt,7pt);
\draw[fill=white] (270+-.4*36:3) +(-7pt,-7pt) rectangle +(7pt,7pt);
\draw[fill=black] (270+8.4*36:3) circle (5pt);
\draw[fill=black] (270+8.8*36:3) circle (5pt);
\draw[fill=black] (270+9.2*36:3) circle (5pt);
%%%%%%%%%%%%%%%%%%%%%%%%%%%%%%%%%%
%Node-Labels$
%%%%%%%%%%%%%%%%%%%%%%%%%%%%%%%%%%

%2-thread%
\node at (270 + .8*36:3) {$4^+$};
\node at (270 + 2*36:3) {$4^+$};

%3-thread%
\node at (270 + 8*36:3) {$4^+$};
\node at (270 + 9.6*36:3) {$4^+$};

%Good 1-thread%
\node at (270 + 7.1*36:3) {$4^+$};
\node at (270 + 6.3*36:3) {$4^+$};

%Bad 1-thread%
\node at (270 + 5.4*36:3) {$4^+$};
\node at (270 + 4.6*36:3) {$3$};

%Worst 1-thread%
\node at (270 + 2.9*36:3) {$3$};
\node at (270 + 3.7*36:3) {$3$};

%SPACES%
\node at (270 + .2*36:3) {$\cdots$};
\node at (270 + 5.85*36:3) {$\cdots$};
\node at (270 + 7.5*36:3) {$\vdots$};
\node at (270 + 2.5*36:3) {$\vdots$};
\node at (270 + 4.15*36:3) {$\cdots$};
% \node at (1.525,2.65) {$\ddots$};

%Discharging%
\node at (270 + 0*36:0) {$f$};
\draw[->, red, very thick] (0,.3) to node [midway, fill=white] {$\frac{11}{15}$} (0,2.75);
\draw[->, red, very thick] (.2,.3) to node [midway, fill=white] {$\frac{4}{5}$} (2.355,1.355);
\draw[->, red, very thick] (-.2,.3) to node [midway, fill=white] {$\frac{2}{3}$} (-2.355,1.355);
\draw[->, red, very thick] (-.3,-.2) to node [midway, left, fill=white] {$\frac{7}{12}$} (-2.3,-1.4);
\draw[->, red, very thick] (-.3,-.2) to node [midway, fill=white] {$1$} (-1.9,-2);
\draw[->, red, very thick] (-.3,-.2) to node [midway, right, fill=white] {$\frac{7}{12}$} (-1.35,-2.4);

%2-thread%
\draw[->, red, very thick] (.3,-.2) to node [midway, left, fill=white] {$\frac{2}{3}$} (1.9,-2);
\draw[->, red, very thick] (.3,-.2) to node [midway, right, fill=white] {$\frac{2}{3}$} (2.3,-1.45);
\end{tikzpicture}}
    \caption{Face discharging rules \ref{rule Fgood} - \ref{Rule F 3thread}.}
    \label{fig:odd-face-discharging}
\end{figure}

In \S\ref{Sec:vertex_final_charge}, we show $\mu''(v) \geq 0$ for each $v \in V(G)$. In \S\ref{Sec:non-poor-face_final_charge}, we show $\mu''(f) \geq 0$ for each rich face $f$ of $G$. Lastly, in \S\ref{Sec:poor-face_final_charge}, we show $\mu''(f) \geq 0$ for each poor face $f$ in $G$. Then as 
\[0 \le \sum_{x \in V(G)\cup F(G)}\mu''(x)=\sum_{x \in V(G)\cup F(G)}\mu(x) = -12,\] 
we obtain a contradiction, completing the proof of \cref{thm:odd10}.
\end{proof}
\subsection{Every vertex has nonnegative final charge}\label{Sec:vertex_final_charge}

\begin{lemma}\label{lem: All vertices have nonnegative final charge}
Each vertex has nonnegative final charge. That is $\mu''(v)\geq 0$ for all $v\in V(G)$.
\end{lemma}
\begin{proof}
We argue each vertex has nonnegative final charge via a series of claims. Let $v \in V(G)$.

\begin{claim}
    If $\deg(v) \geq 6$, then $\mu''(v) \geq 0$.
\end{claim}

\begin{claimproof}
Since $v$ sends at most $\frac{5}{6}$ charge to any neighbor by rule \ref{rule V 3thread}, it follows the remaining charge on $v$ before \ref{Rule extra} is applied is given by $\mu'(v)=(2\deg(v)-6)-(\frac{5}{6}\deg(v)) = \frac{7}{6} \deg(v) - 6 \ge 1 \geq 0$. Thus $\mu''(v) \ge 0$.
\end{claimproof}

\begin{claim}
    If $\deg(v) = 5$, then $\mu''(v) \geq 0$.
\end{claim}
\begin{claimproof}
    Let $N(v) = \{x_1, x_2, x_3, x_4, x_5\}$. By relabeling, we may assume $s(x_i) \leq s(x_j)$ for $i < j$. 
If $v$ contains a $t_3$-neighbor, then without loss of generality let it be $x_1$. By \cref{cor:ti forb}, $\forb(v) \ge 3$. Thus $s(x_5) \geq 1$. Therefore $v$ sends $\frac{5}{6}$ charge to $x_1$, at most $\frac{5}{6}$ charge to $x_2, x_3, x_4$ each by rule \ref{rule V 3thread}, and at most $\frac{1}{3}$ charge to $x_5$ by rule \ref{rule V 1thread}. Thus $v$ sends at most $\frac{22}{6}$ charge to its neighbors, and thus $\mu'(v) \geq (2\deg(v) - 6) -(\frac{22}{6}) \geq 0$. Thus $\mu''(v) \geq 0$.

Now assume $v$ has no $t_3$-neighbors. Then $v$ sends at most $\frac{2}{3}$ charge to each of its neighbors by rule \ref{rule V 2thread}, and thus $\mu'(v) \geq (2\deg(v) - 6) - (\frac{2}{3}\deg(v)) \geq 0$. Therefore $\mu''(v) \geq 0$.
\end{claimproof}

\begin{claim}
    If $\deg(v) = 4$, then $\mu''(v) \geq 0$.
\end{claim}

\begin{claimproof}
Let $N(v) = \{x_1, x_2, x_3, x_4\}$. By relabeling, we may assume $s(x_i) \leq s(x_j)$ for $i < j$. Suppose $v$ has a $t_3$-neighbor; without loss of generality, assume it is $x_1$. Then $v$ sends $\frac{5}{6}$ charge to $x_1$ by rule \ref{rule V 3thread}. 
By \cref{cor:ti forb}, $\forb(v) \ge 3$. It follows that if $s(x_2) = 0$, then $s(x_3) \geq 1$ and $s(x_4) \geq 2$, thus $v$ sends at most $\frac{5}{6}$ charge to $x_2$ by rule \ref{rule V 3thread}, at most $\frac{1}{3}$ charge to $x_3$ by rule \ref{rule V 1thread}, and at most $0$ charge to $x_4$, and thus sends at most $2$ total charge to its neighbors. Thus $\mu'(v)\ge (2(4)-6)-(2)\ge 0$.
On the other hand, if $s(x_2) \geq 1$, then $v$ sends at most $\frac{1}{3}$ charge to each of $x_2, x_3, x_4$ by rule \ref{rule V 1thread}, and thus sends at most $\frac{11}{6}$ charge total. Thus $\mu'(v)\ge (2(4)-6)-(\frac{11}{6})\ge 0$. Thus, in both cases, $\mu'(v) \geq 0$ and so $\mu''(v) \geq 0$.

Now suppose $v$ has no $t_3$-neighbor but has a $t_2$-neighbor; without loss of generality, assume it is $x_1$. Then $v$ sends $\frac{2}{3}$ charge to $x_1$ by rule \ref{rule V 2thread}.
By \cref{cor:ti forb}, $\forb(v) \geq 2$. Thus at most two of $s(x_2), s(x_3), s(x_4)$ can be 0. If $s(x_2) = 0$ and $s(x_3) = 0$, then $s(x_4) = 2$. Thus $v$ sends at most $\frac{2}{3}$ charge to $x_2$ and $x_3$ by rule \ref{rule V 2thread}, and at most 0 charge to $x_4$. Thus $v$ sends at most 2 charge and so, $\mu'(v)\ge (2(4)-6)-(2)\ge 0$. If $s(x_2) = 0$, and $s(x_3) \geq 1$, then $v$ sends at most $\frac{2}{3}$ charge to $x_2$ by rule \ref{rule V 2thread}, and at most $\frac{1}{3}$ charge to $x_3$ and $x_4$ each by rule \ref{rule V 1thread}. Thus $v$ sends at most 2 charge and so, $\mu'(v)\ge (2(4)-6)-(2)\ge 0$. If $s(x_2)\geq 1$, then $v$ sends at most $\frac{1}{3}$ charge to each of $x_2, x_3, x_4$ by rule \ref{rule V 1thread}, and thus sends at most $\frac{5}{3}$ charge. Thus $\mu'(v)\ge (2(4)-6)-(\frac{5}{3})\ge 0$. Therefore, $\mu'(v) \geq 0$ and so $\mu''(v) \geq 0$. 

Now suppose $v$ has no $t_3$- nor $t_2$-neighbor. Then $v$ sends charge at most $\frac{2}{5}$ to each of its neighbors by rule \ref{rule V path 4-3-2}, and thus sends at most $\frac{8}{5}$ charge. Therefore, $\mu'(v)\ge (2(4)-6)-(\frac{8}{5}) \geq 0$ and so $\mu''(v) \geq 0$. 
Thus in all cases, $\mu''(v) \geq 0$.
\end{claimproof}

\begin{claim}
    If $\deg(v) \leq 3$, then $\mu''(v) \geq 0$.
\end{claim}

\begin{claimproof}
    By \cref{obs:odd}\ref{obs:odd - no degree 1 vertex}, there are no vertices of degree 1, and if $\deg(v) = 3$, then $v$ does not participate in any discharging rule. Therefore, $\mu'(v) \geq 0$ and so $\mu''(v) \geq 0$. 
    Thus, we may assume $\deg(v) = 2$. By \cref{obs:odd}\ref{obs:odd - no 4+ thread}, $v$ is on a $k$-thread for $k = 1, 2, 3$. 

First suppose $v$ is on a 1-thread. If $v$ is a good vertex, then by rule \ref{rule V 1thread}, both of $v$'s neighbors send $\frac{1}{3}$ charge to $v$. By \cref{cutvertex}, $v$ is incident to at least 2 faces, and by rule \ref{rule Fgood} each face incident to $v$ sends $\frac{2}{3}$ charge to $v$. Thus $\mu'(v)= (2(2)-6)+(\frac{1}{3}(2)+\frac{2}{3}(2)) = 0$. 
If $v$ is a bad vertex, let $y$ be its $3$-neighbor and $w$ be its $4$-neighbor. By \cref{obs:3vx4nb}, $y$ is incident to a $4^+$-vertex $x$ (note $x \neq w$ as $G$ has girth $10$). If $y$ is a semi-bad $3$-vertex, then $y$ have two $4^+$-neighbors $x$ and $z$, which each of them sends $\frac{1}{10}$ charge to $v$ by rule \ref{rule V semibad}. If $y$ is a bad $3$-vertex, then by rule \ref{rule V path 4-3-2}, $x$ sends $\frac{1}{5}$ charge to $v$. By rule \ref{rule V 1thread}, $w$ sends $\frac{1}{3}$ charge to $v$. By rule \ref{Rule Fbad}, the two faces incident to $v$ each send $\frac{11}{15}$ charge to $v$. Therefore, $\mu'(v)= (2(2)-6)+(\frac{1}{5}+\frac{1}{3}+\frac{11}{15}(2)) = 0$ and so $\mu''(v) \geq 0$. 

If $v$ is a \worst~vertex, let $y_1$ and $y_2$ be its neighbors. By \cref{obs:3vx4nb}, $y_1$ is incident to a $4^+$-vertex $x_1$ and $y_2$ is incident to a $4^+$-vertex $x_2$. Since $g(G) = 10$, all the vertices $y_1, y_2, x_1, x_2$ are distinct. By rule \ref{rule V path 4-3-2}, $x_1$ and $x_2$ each send $\frac{1}{5}$ charge to $v$. By rule \ref{Rule Fworst}, the two faces incident to $v$ each send $\frac{4}{5}$ charge to $v$. Therefore, $\mu'(v)=(2(2)-6)+(\frac{1}{5}(2)+\frac{4}{5}(2)) = 0$ and so $\mu''(v) \geq 0$.

Next suppose $v$ is on a 2-thread. By \cref{obs:odd}\ref{obs:odd - no odd vertex adjacent to a 2+ thread}, $v$ has a $4^+$-neighbor, $x$. By rule \ref{rule V 2thread}, $x$ sends $\frac{2}{3}$ charge to $v$. By rule \ref{Rule F 2thread}, the two faces incident to $v$ each send $\frac{2}{3}$ charge to $v$. Therefore, $\mu'(v)= (2(2)-6)+(\frac{2}{3}+\frac{2}{3}(2)) = 0$ and so $\mu''(v) \geq 0$.

Finally suppose $v$ is on a 3-thread. Suppose $v$ is in a 3-thread, say $v_1v_2v_3$ with anchors $u$ and $w$. By \cref{obs:odd}\ref{obs:odd - no odd vertex adjacent to a 2+ thread}, each of $u$ and $w$ are $4^+$-vertices. If $v = v_2$, then by rule \ref{Rule F 3thread}, the two faces incident to $v$ each send $1$ charge to $v$. Therefore, $\mu'(v)=(2(2)-6)+(1(2)) =0$ and so $\mu''(v) \geq 0$.  If $v = v_1$ or $v = v_3$, then by rule \ref{rule V 3thread}, the anchor adjacent to $v$ sends $\frac{5}{6}$ charge to $v$. By \ref{Rule F 3thread}, the two faces incident to $v$ each send $\frac{7}{12}$ charge to $v$. Therefore, $\mu'(v)=(2(2)-6)+(\frac{5}{6}+\frac{7}{12}(2)) = 0$ and so $\mu''(v) \geq 0$. 
\end{claimproof}

Thus, in all cases, $\mu''(v) \geq 0$. This completes the proof of \cref{lem: All vertices have nonnegative final charge}.
\end{proof}

\subsection{Every rich face has nonnegative final charge}\label{Sec:non-poor-face_final_charge}

Recall by \cref{cor: all faces have array representations}, every boundary of a face of $G$ can be represented by arrays.
For each array in \cref{def:array}, we compute the average charge a face discharges to the vertices represented by the array based on rules \ref{rule Fgood} - \ref{Rule F 3thread}. For example, the array $a_4$ represents the subsequence $4^+ - 2- 2- 2$ of the sequence $4^+ - 2 - 2 - 2 - 4^+$. By discharging rules \ref{discharging rule start} - \ref{discharging rule finish}, a face whose boundary degree sequence contains $a_4$ will distribute charge $0$ to the $4^+$-vertex, $1$ to the middle $2$-vertex, and $\frac{7}{12}$ to the other two $2$-vertices. Thus the average charge received by the vertices represented by $a_4$ from a face is $\frac{13}{24}$. Figure \ref{fig:arrays} summarizes the average charge received per vertex for each array.

\begin{figure}[!ht]
    \centering
    \bgroup
    \def\arraystretch{1.5}
    \begin{tabular}{l|l|c}
       Array  & Degree Sequence & $\substack{\text{Average Charge Received}\\ \text{per Vertex}}$\\
       \hline
        $a_4$ & $\boxed{4^+ - 2 - 2 - 2} - 4^+$ & $\frac{13}{24}$ \\
        $a_3$ & $\boxed{4^+ - 2 - 2} - 4^+$ & $\frac{4}{9}$ \\
        $a_2^{\Worst}$ & $\boxed{3 - 2} - 3$ &$\frac{2}{5}$\\
        $a_2^{\Bad}$ & $\boxed{3 - 2} - 4^+$ or $\boxed{4^+ -2} - 3$ & $\frac{11}{30}$\\
        $a_2^{\Good}$ & $\boxed{4^+ - 2} - 4^+$ & $\frac{1}{3}$\\
        $a_1$ & $\boxed{3^+} - 3^+$ & $0$
    \end{tabular}
    \egroup
    \caption{Arrays in decreasing order of average charge received to its represented vertices by a face.}
    \label{fig:arrays}
\end{figure}

Recall that for each $f \in F(G)$, $\mu'(f)$ is the charge of $f$ after applying \ref{rule Fgood} through \ref{Rule F 3thread}, and $\mu''(f)$ is a charge of $f$ after applying \ref{Rule extra}.
Note that $\mu''(f) \ge \mu'(f)$.

\begin{lemma}\label{lem: all-non-poor-faces nonnegative final charge}
Each rich face has nonnegative final charge. That is $\mu''(f)\geq 0$ for all rich faces $f\in F(G)$.
\end{lemma}
\begin{proof}
Since $\mu''(f) \geq \mu'(f)$, it suffices to show $\mu'(f) \geq 0$.
We argue each rich face has nonnegative final charge via a series of claims. Since $G$ has girth 10, $\ell(f)\geq 10$. We note that the assumption $f$ is rich is only used for faces of length $10$. Recall that the initial charge of a face is $\mu(f) = \ell(f) - 6$. Also recall by \cref{cor: all faces have array representations}, every vertex of $f$ is contained in some array.

\begin{claim}
    If $\ell(f) \geq 14$, $\mu'(f)\geq 0$.
\end{claim}

\begin{claimproof}
Since the average charge received by a vertex represented by an array is at most $\frac{13}{24}$, we have $\mu'(f) \ge (\ell(f) - 6) - (\frac{13}{24}\ell(f)) = \frac{11}{24}\ell(f) - 6 \geq \frac{11}{24}(14) - 6 \geq 0.$
\end{claimproof}

\begin{claim}
    If $\ell(f) = 13$, $\mu'(f)\geq 0$.
\end{claim}

\begin{claimproof}
Suppose $f$ is a $13$-face. Fix an array representation of $f$, and let $x$ denote the number of $a_4$ arrays in this representation. Clearly $0\leq x\leq 3$ and hence at most $12$ of the $13$ vertices incident to $f$ receive an average charge of $\frac{13}{24}$. The other vertices receive an average charge of at most $\frac{4}{9}$.  Thus, $\mu'(f) \geq \mu'(f)\geq (\ell(f) - 6) - (\frac{13}{24}(12) + \frac{4}{9}(1))\geq 0$.
\end{claimproof}

\begin{claim}
    If $\ell(f) = 12$, $\mu'(f)\geq 0$.
\end{claim}

\begin{claimproof}
Suppose $f$ is a $12$-face. Fix an array representation of $f$ and let $x$ and $y$ denote the number of $a_4$ and $a_3$ arrays in the array representation, respectively. Clearly $0\leq x\leq 3$ and $x \neq 3$ by \cref{cor: Greedy faces}. If $x\leq 1$ then $ \mu'(f)\geq (12 - 6) - (\frac{13}{24}(4) + \frac{4}{9}(8))\geq 0$, as desired. Thus, assume $x = 2$. Then $0\leq y\leq 1$. If $y = 0$ then $\mu'(f)\geq (12 - 6) - (\frac{13}{24}(8)+\frac{2}{5}(4))\geq 0$ and if $y = 1$ then $f$ has an array representation which is some permutation of $a_4, a_4, a_3, a_1$, and thus $\mu'(f) = (12 - 6) - (\frac{13}{24}(8)+\frac{4}{9}(3)) \geq 0$. 
\end{claimproof}

\begin{claim}\label{claim: face-11 final charge}
    If $\ell(f)=11, \mu'(f)\geq 0$.
\end{claim}

\begin{claimproof}
 Suppose $f$ is an $11$-face. Fix an array representation of $f$, and let $x$ and $y$ denote the number of $a_4$ and $a_3$ arrays in this representation, respectively. Clearly $0\leq x\leq 2$ and $0\leq y\leq 3$. 

 First suppose $x=0$. Then the average charge received per vertex by a face is at most $\frac{4}{9}$. Thus $\mu'(f) \geq (11 - 6) - (\frac{4}{9}(11)) \geq 0$.

 Next suppose $x = 1$. Then it follows $0\leq y \leq 2$. If $y = 0$ then  $\mu'(f) \geq (11-6) -(\frac{13}{24}(4) + \frac{2}{5}(7)) \geq 0$. If $y = 2$, then the array representation of $f$ is some permutation of $a_4, a_3, a_3, a_1$ and thus, $\mu'(f) = (11- 6)- (\frac{13}{24}(4)+ \frac{4}{9}(6)) \geq 0$. Thus we may assume $y = 1$.
 If the array representation contains $a_2^{\Worst}$, then this is the only $a_2$ it contains, and thus 
 $\mu'(f) \geq (11-6) - (\frac{13}{24}(4) + \frac{4}{9}(3) + \frac{2}{5}(2)) \geq 0$.
 If the array representation does not contain $a_2^{\Worst}$, then it contains at most two copies of the other $a_2$ arrays, and thus:
 $\mu'(f) \geq (11-6) - (\frac{13}{24}(4) + \frac{4}{9}(3) + \frac{11}{30}(4)) \geq 0$.

Finally, suppose $x = 2$. Clearly $0\leq y \leq 1$. If $y = 1$, then the array representation of $f$ is some permutation of $a_4,a_4,a_3$, which cannot occur by \cref{cor: Greedy faces}. Thus, we may assume $y = 0$. 
Then clearly the array representation contains at most one $a_2$.
As an $a_2^{\Worst}$ cannot be adjacent to an $a_4$, it follows the array representation does not contain $a_2^{\Worst}$. If the array representation does not contain $a_2^{\Bad}$, 
then:
$\mu'(f) \geq (11-6) - (\frac{13}{24}(8) + \frac{1}{3}(2)) = 0$. So assume the array representation contains $a_2^{\Bad}$. 
Then the boundary of $f$ can be represented by 
$a_4a_1a_2^{\Bad}a_4$ (recall we may start at an arbitrary vertex and follow an arbitrary orientation). However, having $a_2^{\Bad}a_4$ as a subsequence violates \cref{lem:2badarray_adjacent}, so this cannot occur. This completes the claim. 
\end{claimproof}

\begin{claim}
If $\ell(f) = 10$, and $f$ is rich, then $\mu'(f) \ge 0$.
\end{claim}
\begin{claimproof}
Suppose $f$ is a $10$-face which is rich (recall \cref{def:poor}). Fix an array representation of $f$. Let $x$ and $y$ denote the number of $a_4$ and $a_3$ arrays in this representation, respectively. Clearly $0\leq x\leq 2$ and $0\leq y\leq 3$. 

First suppose $x = 0$. 
If $y = 0$, then $\mu'(f) \geq (10-6) - (\frac{2}{5}(10)) = 0$. 
If $y = 1$, then the array representation of $f$ contains a subsequence of the form $abc$, where $b = a_3$. Note that by degree restrictions, $a, c \neq a_2^{\Worst}$, and by \cref{lem:2badarray_adjacent}, $a,c \neq a_2^{\Bad}$. 
If $a = c = a_2^{\Good}$, then the representation can contain at most one more $a_2$. It follows $\mu'(f) \geq (10-6) -(\frac{4}{9}(3) + \frac{1}{3}(4) + \frac{2}{5}(2)) \geq 0.$ Otherwise, at most one of $a$ or $c$ is equal to $a_2^{\Good}$. Thus at most two more $a_2$ can occur, and it follows $\mu'(f) \geq (10-6) - (\frac{4}{9}(3) + \frac{1}{3}(2) + \frac{2}{5}(4)) \geq 0$.
Therefore we may assume $y \neq 1$ when $x = 0$.
If $y = 2$, then let $a$ and $b$ be two copies of $a_3$ in the array representation of $f$.
Note that by degree restrictions, the array representation does not contain $a_2^{\Worst}$ adjacent to $a$ or $b$, and by \cref{lem:2badarray_adjacent}, the array representation does not contain $a_2^{\Bad}$ adjacent to $a$ or $b$. 
If $a$ and $b$ are not adjacent, then the array representation does not contain any copies of $a_2^{\Worst}$ or $a_2^{\Bad}$; furthermore, it can contain at most two copies of $a_2^{\Good}$, and thus $\mu'(f) \geq (10-6)-(\frac{4}{9}(6) + \frac{1}{3}(4)) \geq 0$. 
If $a$ and $b$ are adjacent, then the representation contains at most one copy of $a_2^{\Worst}$, and by \cref{lem:2badarray_adjacent}, at most one copy of $a_2^{\Bad}$. In either case, it contains at most one $a_2$, and thus $\mu'(f) \geq (10-6)-(\frac{4}{9}(6) + \frac{2}{5}) \geq 0$. 
If the array representation contains no copies of $a_2^{\Worst}$ or $a_2^{\Bad}$ then it can contain at most two copies of $a_2^{\Good}$, and thus $\mu'(f) \geq 0$, as argued a few sentences previously. 
If $y=3$, then the array representation of $f$ contains three $a_3$ and one $a_1$ by degree restrictions, and thus $\mu'(f) \geq (10-6)-(\frac{4}{9}(9)) = 0$.
Thus we may assume it is not the case that $x = 0$.

Therefore we may assume $x=1$. Clearly $0\leq y \leq 2$. If $y = 2$, then the array representation of $f$ is some permutation of $a_4,a_3,a_3$, which cannot occur by \cref{greedylemma}. 
If $y = 1$, then it follows the representation containing at most one $a_2$. By degree restrictions, this $a_2$ is not $a_2^{\Worst}$, and it is not $a_2^{\Bad}$ by \cref{lem:2badarray_adjacent}. Thus, the representation of $f$ is either some permutation of $a_4, a_3, a_2^{\Good}, a_1$, contradicting that $f$ is rich; or $f$ is some permutation of $a_4, a_3, a_1, a_1, a_1$, in which case $\mu'(f)\geq (10-6) - (\frac{13}{24}(4)+\frac{4}{9}(3))\geq 0$. 
Thus we may assume $y \neq 1$. 
So assume $y = 0$.
Since $x = 1$, the array representation contains a subsequence of the form $abc$, where $b = a_4$. Note that by degree requirements, $a, c \neq a_2^{\Worst}$, and by Corollary \ref{lem:2badarray_adjacent}, $a,c \neq a_2^{\Bad}$. 
If $a = c = a_2^{\Good}$, then the representation can contain at most one more $a_2$. However, by degree requirements this $a_2$ cannot be $a_2^{\Worst}$ or $a_2^{\Bad}$. If this $a_2$ were $a_2^{\Good}$, then the array representation would be $a_4a_2^{\Good}a_2^{\Good}a_2^{\Good}$, contradicting \cref{cor: a4a2goodx3 cannot occur}. Thus the representation contains no additional $a_2$, and so $\mu'(f) \geq (10-6) - (\frac{13}{24}(4) + \frac{1}{3}(4)) \geq 0$.
If at most one of $a$ or $c$ is equal to $a_2^{\Good}$, then the other is equal to $a_1$. In addition, at most one more $a_2$ can occur, and so it follows
$\mu'(f) \geq (10-6) - (\frac{13}{24}(4) + \frac{1}{3}(2) + \frac{2}{5}(2)) \geq 0$. Thus we may assume we do not have $x = 1$ and $y = 0$.

So suppose $x = 2$. Then $y = 0$. Let $a$ and $b$ be the two copies of $a_4$. If $a$ and $b$ are not adjacent in the representation, then it follows $f$ can be represented by $a_4a_1a_4a_1$. However, this is a permutation of $a_4, a_4, a_1, a_1$, and thus this contradicts that $f$ is a rich face. Thus we may assume $a$ and $b$ are adjacent in the representation. Note that the representation is not equal to  $a_4a_4a_1a_1$, as this is a permutation of $a_4, a_4, a_1, a_1$, yet $f$ is a rich face. So we may assume $f$ is represented by $a_4a_4a_2$. It follows from degree requirements that $a_2$ is $a_2^{\Good}$.
 However, this cannot occur by \cref{cor: a4a2goodx3 cannot occur}. Thus $x \neq 2$.
This concludes the claim.
\end{claimproof}

Therefore, every rich face $f$ has nonnegative final charge. 
This completes the proof of \cref{lem: all-non-poor-faces nonnegative final charge}. 
\end{proof}

\subsection{Every poor face has nonnegative final charge}\label{Sec:poor-face_final_charge}

Let $v$ be incident to at least one poor face. Recall by \ref{Rule extra}, $v$ sends charge $\frac{\mu'(v)}{p(v)}$ to each poor face $f$ incident to $v$.
For convenience, let $\rr(v) = \frac{\mu'(v)}{p(v)}$.  

\begin{lemma}\label{lem:4vx_forb}
Let $v$ be a $4$-vertex incident to at least one poor face.
\begin{enumerate}[label=(\arabic*), itemsep = 4pt]
    \item\label{4vx_forb4} If $\forb(v) \ge 4$, then $\rr(v) \ge \frac{1}{12}$.
    \item\label{4vx_forb5} If $\forb(v) \ge 5$, then $\rr(v) \ge \frac{1}{6}$.
    \item\label{4vx_forb6} If $\forb(v) \ge 6$, then $\rr(v) \ge \frac{1}{4}$.
\end{enumerate} 
\end{lemma}
\begin{proof}Suppose $v$ is a $4$-vertex. 
Then $p(v) \leq 4$. Let $x_i$ denote the number of neighbors of $v$ with score $i$, where $i \in \{0, 1, 2\}$. Note that $\forb(v) = x_1 + 2x_2$. Note that by the discharging rules \ref{rule V 3thread} and \ref{rule V 1thread}, $\mu'(v) \geq (2(4) - 6) - (\frac{5}{6}(x_0) + \frac{1}{3}(x_1) + 0(x_2))$.

If $\forb(v) \ge 4$, then as $\deg(v) = 4$, it follows $x_0 \leq 2$. If $x_0 = 2$, then $(x_0, x_1, x_2) = (2, 0, 2)$ and thus $\mu'(v) \geq (2(4) - 6) - (\frac{5}{6}(2)) \geq \frac{1}{3}$. Thus $\rr(v) \ge \frac{1}{12}$. 
If $x_0 = 1$, then $(x_0, x_1, x_2) = (1, 2, 1)$ or $(x_0, x_1, x_2) = (1, 0, 3)$. In the first case, $\mu'(v) \geq (2(4) - 6) - (\frac{5}{6}(1) + \frac{1}{3}(2)) \geq \frac{1}{2}$, and in the latter case, $\mu'(v) \geq (2(4) - 6) - (\frac{5}{6}(1)) \geq \frac{1}{2}$ as well. Thus $\rr(v) \geq \frac{1}{12}$.
If $x_0 = 0$, then as $x_1 \leq 4$, it follows: $\mu'(v) \geq (2(4) - 6) - (\frac{1}{3}(4)) \geq \frac{2}{3}$. Thus $\rr(v) \geq \frac{1}{12}$ in this case as well.

If $\forb(v) \ge 5$, then $x_0 \leq 1$. If $x_0 = 1$, then $x_2 \geq 2$, and thus $x_1 \leq 1$. So $\mu'(v) \geq (2(4) - 6) - (\frac{5}{6}(1) + \frac{1}{3}(1))  = \frac{5}{6}$. Thus $\rr(v) \geq \frac{1}{6}$. If $x_0 = 0$, then as $x_1 \leq 3$, it follows  $\mu'(v) \geq (2(4) - 6) - (\frac{1}{3}(3)) = 1$, and $\rr(v) \geq \frac{1}{6}$.

If $\forb(v) \geq 6$, then $x_0 \leq 1$. If $x_0 = 1$, then $x_1 = 0$, and so $\mu'(v) \geq (2(4) - 6) - (\frac{5}{6}(1)) \geq 1$. So $\rr(v) \geq \frac{1}{4}$. If $x_0 = 0$, then $x_1 \leq 2$, and thus $\mu'(v) \geq (2(4) - 6) - (\frac{1}{3}(2)) \geq 1$, so $\rr(v) \geq \frac{1}{4}$.
\end{proof}

\begin{lemma}\label{lem:4vx_forb_no_t3} 
Let $v$ be a $4$-vertex such that $v$ is incident to at least one poor face and $v$ has no $t_3$-neighbors.
\begin{enumerate}[label=(\arabic*), itemsep = 4pt]
    \item\label{4vx_forb3_no_t3} If $\forb(v) \ge 3$, then $\rr(v) \ge \frac{1}{12}$.
    \item\label{4vx_forb4_no_t3} If $\forb(v) \ge 4$, then $\rr(v) \ge \frac{1}{6}$.
    \item\label{4vx_forb5_no_t3} If $\forb(v) \ge 5$, then $\rr(v) \ge \frac{1}{4}$.
    \item\label{4vx_forb6_no_t3} If $\forb(v) \ge 6$, then $\rr(v) \ge \frac{1}{3}$.
\end{enumerate} 
\end{lemma}
\begin{proof}
The proofs follow similarly to the calculations in \cref{lem:4vx_forb}, except now we use the bound by the discharging rules \ref{rule V 2thread} and \ref{rule V 1thread}, $\mu'(v) \geq (2(4) - 6) - (\frac{2}{3}(x_0) + \frac{1}{3}(x_1))$ (i.e., $\frac{2}{3}$ has replaced $\frac{5}{6}$).
\end{proof}

\begin{lemma}\label{lem: all-poor-faces nonnegative final charge}
Each poor face has nonnegative final charge. That is $\mu''(f)\geq 0$ for all poor faces $f\in F(G)$.
\end{lemma}

\begin{proof}
We argue each poor face has nonnegative final charge via a series of claims, namely, looking at the different possibilities of poor faces. By \cref{def:poor}, the array representation of a poor face is a permutation of $a_4,a_4,a_1,a_1$ or $a_4,a_3,a_2^{\Good},a_1$. Thus the boundary of a poor face can be represented by one of the following: \begin{itemize}
    \item $a_4a_4a_1a_1$
    \item $a_4a_1a_4a_1$
    \item $a_4a_3a_2^{\Good}a_1$
    \item $a_4a_3a_1a_2^{\Good}$
    \item $a_4a_2^{\Good}a_3a_1$.
    \end{itemize}
Note then that the boundary of a poor face contains no $3$-vertex with a $2$-neighbor.

 \begin{claim}\label{clm:4141}
If $f$ has array representation $a_4 a_1 a_4 a_1$, then $\mu''(f) \ge 0$.
\end{claim}

\begin{claimproof}
Let $f$ be a face with array representation $a_4a_1a_4a_1$.
By rule \ref{Rule F 3thread}, $\mu'(f) = (10-6) - (\frac{7}{12}(4) + 1(2)) = -\frac{1}{3}$. Let $a,b$ be two adjacent $4^+$-vertices in the boundary of $f$. Note that by \cref{obs:odd}\ref{obs:odd - no odd vertex adjacent to a 2+ thread}, both $a$ and $b$ are even degree vertices.

Suppose at least one of $a$ or $b$ is a $6^+$-vertex; without loss of generality, assume it is $a$.
Since $a$ has at least one $4^+$-neighbor (namely $b$), and since by rule \ref{rule V 3thread} a vertex sends at most $\frac{5}{6}$ charge to one of its neighbors, it follows $\mu'(a) \ge (2k-6) - (\frac{5}{6}(k-1)) = \frac{7}{6}k - \frac{31}{3}$, and thus $\rr(a) \ge (\frac{7}{6}k - \frac{31}{3})\frac{1}{k} > \frac{1}{4}$. Therefore, the pair $a,b$ together sends at least $\frac{1}{4}$ charge to $f$ by rule \ref{Rule extra}, it follows $\mu''(f) \ge \mu'(f) + \frac{1}{4} + \frac{1}{4} \ge (-\frac{1}{3}) + \frac{1}{4} + \frac{1}{4} \ge 0$.

Now suppose both $a$ and $b$ are $4$-vertices.
Since each of $a$ and $b$ has a $t_3$-neighbor, by \cref{cor:ti forb} it follows $\forb(a), \forb(b) \ge 3$.
Suppose $\forb(a) \ge 5$. Then by \cref{lem:4vx_forb}\ref{4vx_forb5}, $\rr(a) \ge \frac{1}{6}$. Similarly, if $\forb(b) \geq 5$ then $\rr(a) \geq \frac{1}{6}$.
If $\forb(a), \forb(b) \le 4$, then by \cref{lem:4-4}, $\forb(a) = \forb(b) = 4$.
Thus, by \cref{lem:4vx_forb}\ref{4vx_forb4}, $\rr(a) \ge \frac{1}{12}$ and $\rr(b) \ge \frac{1}{12}$.
In both cases, $\rr(a) + \rr(b) \ge \frac{1}{6}$.
Since the boundary of $f$ contains two pairs of adjacent even $4^+$-vertices, and each of the pairs sends at least $\frac{1}{6}$ charge to $f$ by rule \ref{Rule extra}, it follows $\mu''(f) \ge \mu'(f) + \frac{1}{6} + \frac{1}{6} \ge (-\frac{1}{3}) + \frac{1}{6} + \frac{1}{6} =0$.
This completes the proof of \cref{clm:4141}.
\end{claimproof}

\begin{claim}\label{clm:4411}
If $f$ has array representation $a_4 a_4 a_1 a_1$, then $\mu''(f) \ge 0$.
\end{claim}

\begin{claimproof}
Let $f$ have array representation $a_4a_4a_1a_1$. Let $abc$ be the path on the boundary of $f$ where $a$ and $c$ are $4^+$-vertices, and $b$ is a $3^+$-vertex (see \cref{fig:a_4a_4a_a_1a_1}).

\begin{figure}[h]
    \centering
            \begin{tikzpicture}
\draw (0,0) circle (2);
\draw[fill=white] (270:2) + (-7pt, -7pt) rectangle +(7pt,7pt);
\node at (270:1.5) {$b$};
\draw[fill=white] (270+36:2) +(-7pt,-7pt) rectangle +(7pt,7pt);
\node at (270+1*36:1.5) {$c$};
\draw[fill=black] (270+2*36:2) circle (5pt);
\node at (270+2*36:1.5) {};
\draw[fill=black] (270+3*36:2) circle (5pt);
\node at (270+3*36:1.5) {};
\draw[fill=black] (270+4*36:2) circle (5pt);
\node at (270+4*36:1.5) {};
\draw[fill=white] (270+5*36:2) +(-7pt,-7pt) rectangle +(7pt,7pt);
\node at (270+5*36:1.5) {};
\draw[fill=black] (270+6*36:2) circle (5pt);
\node at (270+6*36:1.5) {};
\draw[fill=black] (270+7*36:2) circle (5pt);
\node at (270+7*36:1.5) {};
\draw[fill=black] (270+8*36:2) circle (5pt);
\node at (270+8*36:1.5) {};
\draw[fill=white] (270+9*36:2) +(-7pt,-7pt) rectangle +(7pt,7pt);
\node at (270+9*36:1.5) {$a$};
\node at (270+5*36:2) {$4^+$};
\node at (270+9*36:2) {$4^+$};
\node at (270+10*36:2) {$3^+$};
\node at (270+11*36:2) {$4^+$};
\end{tikzpicture}
    \caption{The face $f$ considered in \cref{clm:4411}.} 
    \label{fig:a_4a_4a_a_1a_1}
\end{figure}
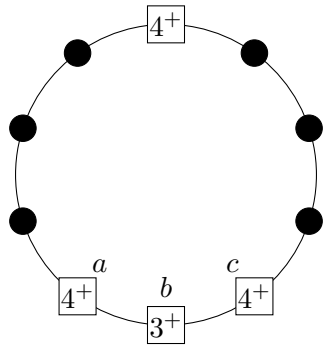

By \ref{Rule F 3thread}, $\mu'(f) = (10-6) - (\frac{7}{12}(4)+1(2)) = -\frac{1}{3}$.
Note that by rule \ref{rule V 3thread} a vertex sends at most $\frac{5}{6}$ charge to one of its neighbors.
Since $b$ has two $4^+$-neighbors, it follows $\mu'(b) \geq (2\deg(b)-6) - (\frac{5}{6}(\deg(b)-2))$. We have $\eta(b) \geq \mu'(b)/\deg(b) = \frac{1}{6}(7\deg(b) - 26)/\deg(b) \geq \frac{1}{3}$ when $\deg(b) \geq 6$ and thus $\mu''(f) \geq \mu'(v) + \rr(b) \geq 0$. Therefore, we may assume $\deg(b) \leq 5$. Also, since $\deg(b) \geq 4$, it follows  $\rr(b) \geq\frac{1}{3}\cdot\frac{1}{4}= \frac{1}{12}$.
Now suppose $x \in \{a,c\}$ is a $k$-vertex, where $k \ge 6$. Since $x$ is adjacent to at least one $4^+$-neighbor, namely $b$, it follows $\mu'(x) \ge (2k-6) - (\frac{5}{6}(k-1)) = \frac{7}{6}k - \frac{31}{6}$, and thus $\rr(x) \ge (\frac{7}{6}k - \frac{31}{6})\frac{1}{k} \geq \frac{7}{6} - \frac{31}{6k} \geq \frac{11}{36}$. Since $\rr(b) \geq \frac{1}{12}$, it follows that $\mu''(f) \geq \mu'(f)+\rr(x)+\rr(b)\ge (-\frac{1}{3})+\frac{11}{36}+\frac{1}{12}\ge 0$. Therefore, we may assume $\deg(a), \deg(b), \deg(c) \leq 5$. By \cref{obs:odd}\ref{obs:odd - no odd vertex adjacent to a 2+ thread}, $a,c$ are $4$-vertices. We now split into 3 cases based on the value of $\deg(b)$.

\subsubsection{Case 1: $b$ is a $5$-vertex.} 
~Then as $b$ has at least two $4^+$-neighbors (namely $a$ and $c$), it sends charges to at most 3 other vertices. By \cref{obs:odd}\ref{obs:odd - }, there are at most five $2$-vertices contained in threads anchored by $b$. Thus it is easy to see $b$ can send at most $\frac{5}{6} + \frac{2}{3} + \frac{2}{5} = \frac{19}{10}$ charge by rules \ref{rule V 3thread}, \ref{rule V 2thread}, and \ref{rule V path 4-3-2}. Thus $\mu'(b) \ge (2(5)-6) - (\frac{19}{10}) = \frac{21}{10}$ and so $\rr(b) \ge \frac{21}{10}\cdot\frac{1}{5}>\frac{1}{3}$.
Thus, $\mu''(f) \ge \mu'(f) + \rr(b) \ge (-\frac{1}{3})+ \frac{1}{3} =0$ by rule \ref{Rule extra}.

\subsubsection{Case 2: $b$ is a $4$-vertex.}~
Since $b$ has at least two even $4^+$-neighbors $a$ and $c$, $\forb(b) \ge 4$.

Suppose $b$ has a $t_3$-neighbor. Then, as $\deg(b) = 4$, it follows $\forb(b) \le 6$.
Suppose $\forb(b) = 6$.
Then by \cref{lem:4vx_forb}\ref{4vx_forb6}, $\rr(b) \ge \frac{1}{4}$. Since $a$ and $c$ have $t_3$-neighbors, then by \cref{cor:ti forb} it follows $\forb(a), \forb(c) \geq 3$.
It then follows by \cref{lem:4-4-4} that either $\forb(a) \geq 4$ or $\forb(c) \geq 4$; without loss of generality, we assume $\forb(a) \ge 4$. By \cref{lem:4vx_forb}\ref{4vx_forb4}, $\rr(a) \ge \frac{1}{12}$.
Thus, $\mu''(f) \ge \mu'(f) + \rr(a) + \rr(b) \ge (-\frac{1}{3}) + \frac{1}{12} + \frac{1}{4} = 0$ by rule \ref{Rule extra}. Suppose $\forb(b) = 5$.
Then by \cref{lem:4vx_forb}\ref{4vx_forb5}, $\rr(b) \ge \frac{1}{6}$.
If $\forb(a) \ge 5$, then by \cref{lem:4vx_forb}\ref{4vx_forb5}, $\rr(a) \ge \frac{1}{6}$.
Thus, $\mu''(f) \ge \mu'(f) + \rr(a) + \rr(b) \ge (-\frac{1}{3}) + \frac{1}{6} + \frac{1}{6} = 0$ by rule \ref{Rule extra}.
By symmetry, we may assume $\forb(a), \forb(c) \le 4$.
By \cref{lem:4-4-4}, $\forb(a) = \forb(c)=4$.
By \cref{lem:4vx_forb}\ref{4vx_forb4}, $\rr(a) \ge \frac{1}{12}$ and $\rr(c) \ge \frac{1}{12}$.
Thus, $\mu''(f) \ge \mu'(f) + \rr(a) + \rr(b) + \rr(c) \ge (-\frac{1}{3}) + \frac{1}{12} + \frac{1}{6} + \frac{1}{12}= 0$ by rule \ref{Rule extra}.
Suppose $\forb(b) = 4$.
Then by \cref{lem:4vx_forb}\ref{4vx_forb4}, $\rr(b) \ge \frac{1}{12}$.
If $\forb(a) \ge 6$, then by \cref{lem:4vx_forb}\ref{4vx_forb6}, $\rr(a) \ge \frac{1}{4}$.
Thus, $\mu''(f) \ge \mu'(f) + \rr(a) + \rr(b) \ge (-\frac{1}{3}) + \frac{1}{4} + \frac{1}{12} = 0$ by rule \ref{Rule extra}.
By symmetry, we may assume $\forb(a), \forb(c) \le 5$.
By \cref{lem:4-4-4}, without loss of generality, we may assume $\forb(a) = 5$ and $\forb(c) \ge 4$.
By \cref{lem:4vx_forb}\ref{4vx_forb4} and \ref{4vx_forb5}, $\rr(a) \ge \frac{1}{6}$ and $\rr(c) \ge \frac{1}{12}$.
Thus, $\mu''(f) \ge \mu'(f) + \rr(a) + \rr(b) + \rr(c) \ge (-\frac{1}{3}) + \frac{1}{6} + \frac{1}{12} + \frac{1}{12}= 0$ by rule \ref{Rule extra}. By the above arguments, we may assume that $b$ has no $t_3$-neighbor.

Suppose $b$ has a $t_2$-neighbor.
Then, as $\deg(b) = 4$, it follows $\forb(b) \le 6$. Suppose $\forb(b) = 6$. Then by \cref{lem:4vx_forb_no_t3}\ref{4vx_forb6_no_t3}, $\rr(b) \ge \frac{1}{3}$.
Thus, $\mu''(f) \ge \mu'(f) + \rr(b) \ge (-\frac{1}{3}) + \frac{1}{3}= 0$ by rule \ref{Rule extra}.
Suppose $\forb(b) =5$. 
Then by \cref{lem:4vx_forb_no_t3}\ref{4vx_forb5_no_t3}, $\rr(b) \ge \frac{1}{4}$.  
By \cref{lem:4-4-4}, without loss of generality, we may assume $\forb(a) \ge 4$.
By \cref{lem:4vx_forb}\ref{4vx_forb4}, $\rr(a) \ge \frac{1}{12}$.
Thus, $\mu''(f) \ge \mu'(f) + \rr(a) + \rr(b) \ge (-\frac{1}{3}) + \frac{1}{12} + \frac{1}{4}= 0$ by rule \ref{Rule extra}.
Suppose $\forb(b) = 4$.
Then by \cref{lem:4vx_forb_no_t3}\ref{4vx_forb4_no_t3}, $\rr(b) \ge \frac{1}{6}$.
If $\forb(a) \ge 5$, then by \cref{lem:4vx_forb}\ref{4vx_forb5}, $\rr(a) \ge \frac{1}{6}$.
Thus, $\mu''(f) \ge \mu'(f) + \rr(a) + \rr(b) \ge (-\frac{1}{3}) + \frac{1}{6} + \frac{1}{6}= 0$ by rule \ref{Rule extra}.
By symmetry, we may assume $\forb(a), \forb(c) \le 4$.
By \cref{lem:4-4-4}, $\forb(a) = \forb(c) = 4$.
By \cref{lem:4vx_forb}\ref{4vx_forb4}, $\rr(a) \ge \frac{1}{12}$ and $\rr(c) \ge \frac{1}{12}$.
Thus, $\mu''(f) \ge \mu'(f) + \rr(a) + \rr(b) + \rr(c) \ge (-\frac{1}{3}) + \frac{1}{6} + \frac{1}{12} + \frac{1}{12}= 0$ by rule \ref{Rule extra}.
By the above arguments, we may assume that $b$ has no $t_2$- or $t_3$-neighbors.

Suppose $b$ has a $t_{\Worst}$-neighbor. Then as the boundary of a poor face contains no $3$-vertex with a $2$-neighbor, it follows $b$ is incident with at most two poor faces.
Since $b$ has at least two even $4^+$-neighbors $a$ and $c$, and $b$ has no $t_2$- or $t_3$-neighbors, $\mu'(b) \ge (2 (4)-6) - (\frac{2}{5}(2)) = \frac{6}{5}$ by rule \ref{rule V path 4-3-2} and thus $\rr(b) \geq \frac{6}{5}\cdot\frac{1}{2} = \frac{3}{5} > \frac{1}{3}$.
Thus, $\mu''(f) \ge \mu'(f) + \rr(b) \ge (-\frac{1}{3}) + \frac{1}{3}= 0$ by rule \ref{Rule extra}.

Thus we may assume that $b$ has no $t_3$-, $t_2$-, or $t_{\Worst}$-neighbors.
Then, $\mu'(b) \ge (2(4)-6) - (\frac{1}{3}(2)) = \frac{4}{3}$ and $\rr(b) \ge \frac{4}{3}\cdot \frac{1}{4} \geq \frac{1}{3}$.
Thus, $\mu''(f) \ge \mu'(f) + \rr(b) \ge (-\frac{1}{3}) + \frac{1}{3}= 0$ by rule \ref{Rule extra}.

\subsubsection{Case 3: $b$ is a $3$-vertex.} ~Since $b$ has two $4^+$-neighbors, it follows $b$ is either a good $3$-vertex or a semi-bad $3$-vertex.

Subcase 3.1: Suppose $b$ is a good $3$-vertex. Since $a$ and $c$ have $t_3$-neighbors, then by \cref{cor:ti forb} it follows $\forb(a), \forb(c) \geq 3$.
Recall before we split into cases we assumed $a$ and $c$ were $4$-vertices.
Since $\forb(a) \ge 3$, and it has a $t_{\Good}$-neighbor (namely $b$) and a $t_3$-neighbor, it follows the most charge $a$ gives is $\mu'(a) \ge (2(4)-6) - (\frac{5}{6}(2)) = \frac{1}{3}$ by rule \ref{rule V 3thread} and thus $\rr(a) \ge\frac{1}{3}\cdot\frac{1}{4}=\frac{1}{12}$.
By symmetry, $\rr(c) \ge \frac{1}{12}$.

We now show that if $a$ is adjacent to at most one $t_3$-neighbor, then $a$ is incident to at most 3 poor faces. Indeed, as $f$ is a 10-face and $g(G) \geq 10$, we may assume $f$ contains no cycles in its interior. Since \cref{obs:odd} implies $G$ has no $1$-vertices, it follows the interior of $f$ contains no vertices or edges of $G$. Thus the other two neighbors of $a$ are not in the interior of $f$. Let $e$ be the neighbor of $b$ that's not $a$ or $c$. Let $f'$ be the face containing $e, b, a$; we show $f'$ is not poor. 

Suppose it were. Note that $\deg(e) \geq 3$ as we are in the case when $b$ is good. Since the neighbors of $a$ are not in the interior of $f$, there exists a neighbor of $a$, say $x$, such that $x$ is incident to $f'$; furthermore, $ebax$ is a path on the boundary of $f'$ and has degree sequence $3^+ - 3 - 4 - \deg(x)$. But as $f'$ is poor, this degree sequence is only possible if $x$ is the endpoint of a $3$-thread, contradicting that $a$ has at most one $3$-thread. Thus $f'$ is not poor, and so $a$ is incident to at most $3$ poor faces. By symmetry, the same is true of $c$. 

If both $\forb(a) = \forb(c) = 3$, then by \cref{obs:a_4a_4a_1a_1good3}, both $a$ and $c$ are adjacent to at most one $t_3$-neighbor, and as discussed above both are incident to at most 3 poor faces. Since both $a$ and $c$ are adjacent to at most one $t_3$-neighbor and at least one $t_{\Good}$-neighbor (namely $b$), it follows for $x \in \{a, c\}$ that $x$ has two other neighbors $y$ and $z$ with $s_x(y)=s_x(z)=1$ or $s_x(y)=2$ and $s_x(z)=0$. Thus $\mu'(x)\ge (2(4)-6)-(\frac{5}{6}+\frac{2}{3})=\frac{1}{2}$ by the discharging rules.

As $a$ and $c$ are incident with at most 3 poor faces, it follows $\rr(a), \rr(c) \ge \frac{1}{2}\cdot \frac{1}{3}=\frac{1}{6}$.
Thus, $\mu''(f) \ge \mu'(f) + \rr(a) +\rr(c) \ge (-\frac{1}{3}) + \frac{1}{6} + \frac{1}{6}= 0$ by rule \ref{Rule extra}. 

So assume $\forb(a) \geq 4$ (the case of $\forb(c) \geq 4$ is symmetric). Then as $a$ has a $t_3$- and $t_{\Good}$-neighbor, which have scores $0$ and $1$ respectively, it follows $a$ has a neighbor $x$ with $s_a(x) = 2$ and a neighbor $y$ with $s_a(y) \geq 1$, and thus contains at most one $t_3$-neighbor. Therefore, $a$ is incident to at most $3$ poor faces. 
Since $s_a(x) = 2$ and $s_a(y) \geq 1$, it also follows $\mu'(a) \ge (2 (4)-6) - (\frac{5}{6}+\frac{1}{3}) = \frac{5}{6}$. Thus $\rr(a) \ge\frac{5}{6}\cdot \frac{1}{3}=\frac{5}{18} > \frac{1}{4}$. Recall at the start of Subcase 3.1 we showed $\eta(c) \geq \frac{1}{12}$.
Thus, $\mu''(f) \ge \mu'(f) + \rr(a) +\rr(c) \ge (-\frac{1}{3}) + \frac{1}{4} + \frac{1}{12}= 0$ by rule \ref{Rule extra}. This completes subcase 3.1.

Subcase 3.2:
Suppose $b$ is a semi-bad $3$-vertex.
By \cref{lem:4-bad3-4}, without loss of generality, we may assume $\forb(a) \ge 3$ and $\forb(c) \ge 4$.
Note that no poor face has a degree sequence containing $4^+ - 3 - 2$. 
Thus, as $b$ is a semi-bad $3$-vertex, this implies that $f$ is the only poor face that is incident with $b$. It follows $a$ and $c$ are incident with at most $3$ poor faces.

Suppose  $\forb(x) \geq 4$ for each $x \in \{a, c\}$. Then as $x$ is incident to at least one $t_3$-neighbor (namely the 3-thread along the boundary of $f$), and a $t_{\Semibad}$-neighbor (namely $b$), it follows $x$ has two neighbors $w$ and $y$ with $s(w)  = 2$ and $s(y) \geq 1$. Therefore $\mu'(x) \geq (2(4) - 6) - (\frac{5}{6} + \frac{1}{3} + \frac{1}{10}) = \frac{11}{15}$. Since $x$ is incident to at most $3$ poor faces, it follows $\eta(x) \geq \frac{11}{45}$. Thus if $\forb(x) \geq 4$ for both $x \in \{a, c\}$, then it follows
$\mu''(f) \geq \mu'(f) + \eta(a) + \eta(c) \geq (-\frac{1}{3}) + \frac{11}{45} + \frac{11}{45} \geq 0$. Therefore, as $\forb(c) \geq 4$, we may assume $\forb(a) = 3$.

If $\forb(c) = 4$, then as $\forb(a) =3$, it follows by \cref{obs:a_4a_4a_1a_1semibad3} that $a$ is adjacent to at most one $t_3$-neighbor. Also since $\forb(a) = 3$, $a$ has two other neighbors $x$ and $y$ with $s_a(x)=s_a(y)=1$ or $s_a(x)=2$, $s_a(y)=0$. Thus, in both cases, $\mu'(a)\ge (2(4)-6)-(\frac{5}{6}+\frac{2}{3}+\frac{1}{10})=\frac{2}{5}$. Thus as $a$ is incident to at most $3$ poor faces, it follows $\rr(a) \ge \frac{2}{5}\cdot\frac{1}{3}=\frac{2}{15}$.
As $\forb(c) = 4$ and $c$ has a $t_3$- and $t_{\Semibad}$-neighbor, it follows $c$ has two other neighbors $x$ and $y$ with $s_c(x) = 2$ and $s_c(y) \geq 1$, and thus $\mu'(c) \ge (2 (4) -6) - (\frac{5}{6} + \frac{1}{3} + \frac{1}{10}) = \frac{11}{15}$ and so $\rr(c) \ge \frac{11}{15}\cdot\frac{1}{3}=\frac{11}{45}$.
Thus, $\mu''(f) \ge \mu'(f) + \rr(a) + \rr(c)\ge (-\frac{1}{3}) + \frac{2}{15} + \frac{11}{45} \geq 0$ by rule \ref{Rule extra}.

If $\forb(c) \ge 5$, then it has two neighbors $x$ and $y$ with $s_c(x) = 2$ and $s_c(y) = 2$. Thus $\mu'(c) \ge (2 (4) - 6) - (\frac{5}{6} + \frac{1}{10}) = \frac{16}{15} > 1$ and $\rr(c) \ge \frac{1}{3}$.
Thus, $\mu''(f) \ge \mu'(f) + \rr(c) \ge (-\frac{1}{3}) + \frac{1}{3}= 0$ by rule \ref{Rule extra}.
This completes subcase 3.2.

Therefore, we see in all cases, $\mu''(f) \geq 0$. This completes the proof of \cref{clm:4411}.
\end{claimproof}

\begin{claim}\label{clm:4321}
If $f$ has an array representation that is a permutation of $a_4a_3a_2^{\Good}a_1$, then $\mu''(f) \ge 0$.
\end{claim}
\begin{claimproof}
Let $f$ be a poor face with an array representation that is a permutation of $a_4a_3a_2^{\Good}a_1$.
Then by symmetry, $f$ can be represented by one of the following:

\begin{itemize}
    \item $a_4a_3a_2^{\Good}a_1$
    \item $a_4a_3a_1a_2^{\Good}$
    \item $a_4a_2^{\Good}a_3a_1$.
\end{itemize}
See \cref{fig: a_4 a_3 a_2^good a_1}.

\begin{figure}[!htb]
\centering
\begin{subfigure}[t]{0.3\textwidth}
\begin{tikzpicture}
\draw (0,0) circle (2);
\draw[fill=black] (270:2) circle (5pt);
\draw[fill=white] (270+36:2) +(-7pt,-7pt) rectangle +(7pt,7pt);
\draw[fill=black] (270+2*36:2) circle (5pt);
\draw[fill=black] (270+3*36:2) circle (5pt);
\draw[fill=black] (270+4*36:2) circle (5pt);
\draw[fill=white] (270+5*36:2) +(-7pt,-7pt) rectangle +(7pt,7pt);
\draw[fill=white] (270+6*36:2) +(-7pt,-7pt) rectangle +(7pt,7pt);
\draw[fill=black] (270+7*36:2) circle (5pt);
\draw[fill=white] (270+8*36:2) + (-7pt, -7pt) rectangle + (7pt, 7pt);
\draw[fill=black] (270+9*36:2) circle (5pt);
\node at (270 + 5*36:2) {$4^+$};
\node at (270 + 8*36:2) {$4^+$};
\node at (270 + 11*36:2) {$4^+$};
\node at (270 + 6*36:2) {$4^+$};
\end{tikzpicture}
\caption*{(1) $a_4a_3a_2^{\Good}a_1$}
\end{subfigure}
\hfill
\begin{subfigure}[t]{0.3\textwidth}
\begin{tikzpicture}
\draw (0,0) circle (2);
\draw[fill=black] (270:2) circle (5pt);
\draw[fill=white] (270+36:2) +(-7pt,-7pt) rectangle +(7pt,7pt);
\draw[fill=black] (270+2*36:2) circle (5pt);
\draw[fill=black] (270+3*36:2) circle (5pt);
\draw[fill=black] (270+4*36:2) circle (5pt);
\draw[fill=white] (270+5*36:2) +(-7pt,-7pt) rectangle +(7pt,7pt);
\draw[fill=black] (270+6*36:2) circle (5pt);
\draw[fill=white] (270+7*36:2) +(-7pt,-7pt) rectangle +(7pt,7pt);
\draw[fill=white] (270+8*36:2) + (-7pt, -7pt) rectangle + (7pt, 7pt);
\draw[fill=black] (270+9*36:2) circle (5pt);
\node at (270 + 5*36:2) {$4^+$};
\node at (270 + 8*36:2) {$4^+$};
\node at (270 + 11*36:2) {$4^+$};
\node at (270 + 7*36:2) {$4^+$};
\end{tikzpicture}
\caption*{(2) $a_4a_3a_1a_2^{\Good}$}
\end{subfigure}
\hfill
\begin{subfigure}[t]{0.3\textwidth}
\begin{tikzpicture}
\draw (0,0) circle (2);
\draw[fill=black] (270:2) circle (5pt);
\draw[fill=white] (270+36:2) +(-7pt,-7pt) rectangle +(7pt,7pt);
\draw[fill=black] (270+2*36:2) circle (5pt);
\draw[fill=black] (270+3*36:2) circle (5pt);
\draw[fill=black] (270+4*36:2) circle (5pt);
\draw[fill=white] (270+5*36:2) +(-7pt,-7pt) rectangle +(7pt,7pt);
\draw[fill=white] (270+6*36:2) +(-7pt,-7pt) rectangle +(7pt,7pt);
\draw[fill=black] (270+7*36:2) circle (5pt);
\draw[fill=black] (270+8*36:2) circle (5pt);
\draw[fill=white] (270+9*36:2) + (-7pt, -7pt) rectangle + (7pt, 7pt);
\node at (270 + 5*36:2) {$4^+$};
\node at (270 + 6*36:2) {$4^+$};
\node at (270 + 9*36:2) {$4^+$};
\node at (270 + 11*36:2) {$4^+$};
\end{tikzpicture}
\caption*{(3) $a_4a_2^{\Good}a_3a_1$}
\end{subfigure}
\caption{Faces considered in \cref{clm:4321} with corresponding array representations starting from the top vertex and proceeding clockwise.}
\label{fig: a_4 a_3 a_2^good a_1}
\end{figure}
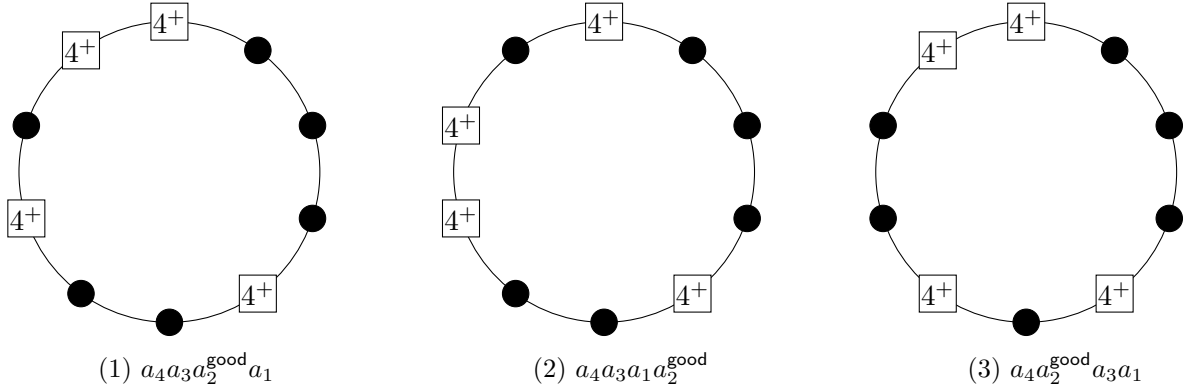

By rules \ref{rule Fgood}, \ref{Rule F 2thread}, and \ref{Rule F 3thread}, $\mu'(f) = (10-6) - (\frac{7}{12}(2) + 1 + \frac{2}{3}(3)) = - \frac{1}{6}$.

By \cref{obs:odd}\ref{obs:odd - no odd vertex adjacent to a 2+ thread} and \cref{Obs:1ThreadNextToOdd}, there exists two adjacent even $4^+$-vertices, say $a$ and $b$, incident with $f$.
Suppose $a$ is a $k$-vertex, where $k \ge 6$. Since $b$ receives no charge from $a$, it follows $\mu'(a) \ge (2k-6) - (\frac{5}{6}(k-1)) = \frac{7}{6}k-\frac{31}{6}$ and thus $\rr(a) \ge (\frac{7}{6}k-\frac{31}{6})\frac{1}{k} = \frac{7}{6}-\frac{31}{6k} > \frac{1}{6}$.
Thus, $\mu''(f) \ge \mu'(f) + \rr(a) \ge (-\frac{1}{6}) + \frac{1}{6}= 0$ by rule \ref{Rule extra}.
By symmetry, we may assume that both $a$ and $b$ are $4$-vertices.

Suppose $a$ has no $t_2$- or $t_3$-neighbors.
Then $\mu'(a) \ge (2 (4) - 6) - (\frac{2}{5}(3)) = \frac{4}{5}$ and $\rr(a) \ge\frac{4}{5}\cdot \frac{1}{4}= \frac{1}{5} > \frac{1}{6}$.
Thus, $\mu''(f) \ge \mu'(f) + \rr(a) \ge (-\frac{1}{6}) + \frac{1}{6}= 0$ by rule \ref{Rule extra}.
By symmetry, we may assume that both $a$ and $b$ have at least one $t_2$- or $t_3$-neighbor.

Suppose $\forb(a) \ge 5$.
Then by \cref{lem:4vx_forb}\ref{4vx_forb5}, $\rr(a) \ge \frac{1}{6}$.
Thus, $\mu''(f) \ge \mu'(f) + \rr(a) \ge (-\frac{1}{6}) + \frac{1}{6}= 0$ by rule \ref{Rule extra}.
By symmetry, we may assume that $\forb(a), \forb(b) \le 4$.

Suppose both $a$ and $b$ have a $t_3$-neighbor.
By \cref{lem:4-4}, $\forb(a) + \forb(b) \ge 8$, and thus $\forb(a) = \forb(b) = 4$.
By \cref{lem:4vx_forb}\ref{4vx_forb4}, $\rr(a) \ge \frac{1}{12}$ and $\rr(b) \ge \frac{1}{12}$.
Thus, $\mu''(f) \ge \mu'(f) + \rr(a) + \rr(b) \ge (-\frac{1}{6}) + \frac{1}{12} + \frac{1}{12}= 0$ by rule \ref{Rule extra}.

So suppose exactly one in $\{a,b\}$ has a $t_3$-neighbor.
Without loss of generality, assume it is $a$. 
Note that $b$ has a $t_2$-neighbor by our previous assumption.
If $\forb(b) = 4$, then by \cref{lem:4vx_forb_no_t3}\ref{4vx_forb4_no_t3}, $\rr(b) \ge \frac{1}{6}$ and $\mu''(f) \ge \mu'(f) + \rr(b) \ge (-\frac{1}{6}) + \frac{1}{6}= 0$.
Thus, we may assume that $\forb(b) \le 3$ by rule \ref{Rule extra}.
Since $a$ has a $t_3$-neighbor and $b$ has a $t_2$-neighbor, then by \cref{lem:4-4}, $\forb(a) + \forb(b) \ge 7$, and thus  $\forb(a) =4$ and $\forb(b)=3$.
By \cref{lem:4vx_forb}\ref{4vx_forb4} and \cref{lem:4vx_forb_no_t3}\ref{4vx_forb3_no_t3}, $\rr(a) \ge \frac{1}{12}$ and $\rr(b) \ge \frac{1}{12}$.
Thus, $\mu''(f) \ge \mu'(f) + \rr(a) + \rr(b) \ge (-\frac{1}{6}) + \frac{1}{12} + \frac{1}{12}= 0$ by rule \ref{Rule extra}.

For the last case, assume both $a$ and $b$ have no $t_3$-neighbor.
Note that both $a$ and $b$ have a $t_2$-neighbor by our previous assumption.
If $\forb(a) = 4$, then by \cref{lem:4vx_forb_no_t3}\ref{4vx_forb4_no_t3}, $\rr(a) \ge \frac{1}{6}$ and $\mu''(f) \ge \mu'(f) + \rr(a) \ge (-\frac{1}{6}) + \frac{1}{6}= 0$.
By symmetry, $\forb(a), \forb(b) \le 3$ by rule \ref{Rule extra}.
By \cref{lem:4-4}, $\forb(a) + \forb(b) \ge 6$, and we obtain $\forb(a) =\forb(b)=3$.
By \cref{lem:4vx_forb_no_t3}\ref{4vx_forb3_no_t3}, $\rr(a) \ge \frac{1}{12}$ and $\rr(b) \ge \frac{1}{12}$.
Thus, $\mu''(f) \ge \mu'(f) + \rr(a) + \rr(b) \ge (-\frac{1}{6}) + \frac{1}{12} + \frac{1}{12}= 0$ by rule \ref{Rule extra}.

This completes the proof of all the cases in \cref{clm:4321}. 
\end{claimproof}

Therefore, every poor face $f \in F(G)$ has nonnegative final charge, that is $\mu''(f)\geq 0$ for all poor faces $f\in F(G)$.
This completes the proof of \cref{lem: all-poor-faces nonnegative final charge}.
\end{proof}

\section{Proof of Theorem~\ref{thm:pcf11}}\label{sec:pcf11}

We note that we use terminology from \S\ref{sec: prelims} and
\S\ref{sec: forb/flex} in this section. 
Before proving the theorem, we need a few observations. Suppose \cref{thm:pcf11} is false; among all counterexamples to \cref{thm:pcf11}, let $G$ be one with $|V(G)|$ minimized.  

The observation below follows from ~\cite[Lemma 2.1, Lemma 2.3]{cho2022proper}. While these lemmas are for a graph $G$ that is a minimal counterexample to the statement that all planar graphs with girth at least 12 have $\chi_{\PCF}(G) \leq 4$, the proof of these lemmas do not use the girth condition. Thus the statements are true for a minimal counterexample $G$ to \cref{thm:pcf11} as well.

\begin{observation}\label{obs:lem}
The following hold for $G$, where $n_{k^+}(v)$ denotes the number of $k^+$-neighbors of $v$:
\begin{enumerate}[label = {(\arabic*)}]
    \item\label[observation]{obs:lem1} $G$ has no $1$-vertex.
    \item\label[observation]{obs:lem2} $G$ has no $3$-vertex anchoring a $2^+$-thread.
    \item \label[observation]{obs:lem3} $G$ has no $4^+$-thread.
    \item\label[observation]{obs:lem4} If a $4$-vertex $v$ has a $t_3$-neighbor, then $v$ anchors at most $1+n_{3^+}(v)$ many $2^+$-threads. Thus if $v$ is a $4$-vertex with a $t_3$-neighbor, then the total number of $t_3$ and $t_2$-neighbors of $v$ is at most $2$.
    \item\label[observation]{obs:lem5} If a $5$-vertex $v$ anchors a $3$-thread, then $v$ anchors at most $4+n_{3^+}(v)$ many $2^+$-threads. 
\end{enumerate}
\end{observation}

As $G$ is connected, every face has a boundary. The following corollary shows the boundary can always be written in terms of arrays (see \cref{def:array}).
\begin{corollary}\label{lem: all faces have array representations}
The boundary of every face of $G$ has an array representation.
\end{corollary}
\begin{proof}
By \cref{obs:lem}\ref{obs:lem1}, the degree sequence of any face boundary never contains $1$. 
By \ref{obs:lem}\ref{obs:lem3}, the degree sequence of any face boundary does not contain the subsequence $2 -2 -2 -2$. 
By \cref{obs:lem}\ref{obs:lem2}, the degree sequence of any face boundary does not contain the subsequence $3 - 2 -2$. 
Therefore, the only subsequences the degree sequence of a face boundary can contain are exactly the arrays as defined in \cref{def:array}. Thus the degree sequence of every face boundary can be written as a concatenation of arrays. 
\end{proof}

We also show a certain structure does not appear in $G$ by using the forb-flex method of \S \ref{sec: forb/flex}, namely by slightly modifying the proof of \cref{lemma: flex > forb}.

\begin{observation}\label{obs: pcf-reduce}
$G$ contains no $4$-vertex with three $t_2$-neighbors and one $t_1$-neighbor.
\end{observation}
\begin{proof}
    Suppose for contradiction $u\in V(G)$ has three $t_2$-neighbors and one $t_1$-neighbor. Let $\vp$ be a PCF $4$-coloring of $G-S[u]$. Then $\vp$ is also an odd $4$-coloring of $G - S[u]$. Recall by \cref{def: forb} and \cref{def: flex}, $\forb(u) = 1$ and $\flex(u) = 2$. Since $\flex(u)>\forb(u)$, it follows by \cref{cor: flex > forb} that $\phi$ can be extended to be an odd 4-coloring of all of $G$. However, by slightly changing the coloring rules used in \cref{lemma: flex > forb}, we claim that $\phi$ can actually be extended to a PCF coloring of all of $G$. Indeed, let $\phi_{*}(v)$ be the unique color in the neighborhood of $v$, should it exist. Then by replacing $\phi_{\mathsf{o}}(w)$ with $\phi_{*}(w)$ in the rules outlined \cref{lemma: flex > forb}, it follows $\phi$ is extended to an odd coloring of $G$ such that every vertex in $G - S[u]$ contains a unique color in its neighborhood. Furthermore, by our assumption, $S[u]$ contains only $2$- or $4$-vertices. As $\phi$ is an odd coloring for $G$, it follows by \cref{lem:odd-and-pcf} that each vertex in $S[u]$ contains a unique color in its neighborhood. Thus $\phi$ is indeed a PCF coloring of $G$, completing the proof.
\end{proof}

Finally, we require a definition:

\begin{definition}
Let $v\in V(G)$ be on a $2$-thread. We say $v$ is \defn{supported} if its adjacent anchor has a $t_1$- or $t_3$-neighbor. Otherwise, we say $v$ is \defn{unsupported}.
\end{definition}

See \cref{fig:supported}.

\begin{figure}[tbh!]
    \centering
    \resizebox{.3\textwidth}{!}{\begin{tikzpicture}

%%%%% Supported %%%%%%%

%%Edges%%
\draw (1.5,0) -- (-2.25,0);
\draw (3,-1) -- (-2.25,-1);

%% Vertices %%
%% 1-thread %%
\draw[fill = white] (0,0) + (-7pt, -7pt) rectangle + (7pt, 7pt);
\draw[fill = black] (-.75, 0) circle (5pt);
\draw[fill = black] (-1.5, 0) circle (5pt);
\draw[fill=white] (-2.25,0) + (-7pt, -7pt) rectangle + (7pt, 7pt);
\draw[fill=black] (.75, 0) circle (5pt);
\draw[fill=white] (1.5,0) + (-7pt, -7pt) rectangle + (7pt, 7pt);

%% 3-thread %%
\draw[fill = white] (0,-1) + (-7pt, -7pt) rectangle + (7pt, 7pt);
\draw[fill = black] (-.75, -1) circle (5pt);
\draw[fill = black] (-1.5, -1) circle (5pt);
\draw[fill=white] (-2.25,-1) + (-7pt, -7pt) rectangle + (7pt, 7pt);
\draw[fill=black] (.75, -1) circle (5pt);
\draw[fill=black] (1.5, -1) circle (5pt);
\draw[fill=black] (2.25, -1) circle (5pt);
\draw[fill=white] (3,-1) + (-7pt, -7pt) rectangle + (7pt, 7pt);

%% Labels %%
\node at (-.75, .35) {$v$};
\node at (-.75, -.65) {$v$};
\node at (-2.25, 0) {$4^+$};
\node at (0, 0) {$4^+$};
\node at (1.5, 0) {$3^+$};
\node at (3, -1) {$4^+$};
\node at (-2.25, -1) {$4^+$};
\node at (0, -1) {$4^+$};
\node at (0, -1.25) {};

\end{tikzpicture}}
    \caption{$v$ supported.}
    \label{fig:supported}
\end{figure}
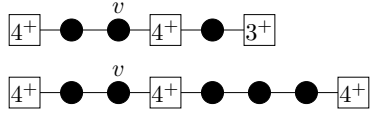

We are now ready to prove \cref{thm:pcf11}, which we restate below for convenience.

\begin{theorem*}[\ref{thm:pcf11}]
    If $G$ is a planar graph with girth at least $11$, then $\pcf(G) \le 4$.
\end{theorem*}

\begin{proof}
Suppose for contradiction, the theorem is false. 
Among all counterexamples, choose a plane graph $G$ such that $|V(G)|$ is minimized. Then $G$ is connected and, as mentioned in the introduction, the results of ~\cite{cho2022proper} imply the girth of $G$ is exactly 11.

Let $G$ be drawn in the plane without edge-crossings. Let $F(G)$ be the set of faces of $G$. Define $\mu : V(G) \cup F(G) \to \mathbb{R}$ by 
\[
    \mu(x) = \begin{cases}
        2\deg(x) - 6, & x \in V(G),\\
        \ell(x) - 6, & x \in F(G).
    \end{cases}
\]
By Euler's formula, \[\sum_{x \in V(G)\cup F(G)}\mu(x) = -12.\] 

We redistribute the charges as follows:

\vspace{1em}
\begin{mdframed}

\emph{Discharging Rules}
\vspace{8pt}
\begin{enumerate}[label = {(V\arabic*)}, itemsep = 5pt]
    \item\label{pcf:rule:vertex1} Each $4^+$-vertex sends charge $1$ to its $t_3$-neighbors.
    \item\label{pcf:rule:vertex2}
    Each $4^+$-vertex sends charge $1$ to its supported $t_2$-neighbors.
    \item\label{pcf:rule:vertex3} Each $4^+$-vertex sends charge $\frac{1}{2}$ to its unsupported $t_2$-neighbors.
\end{enumerate}
\vspace{8pt}
\begin{enumerate}[label = {(F\arabic*)}, itemsep = 5pt]
   \item\label{pcf:rule:face1} Each face sends charge $1$ to its incident $2$-vertices that are on a $1$-thread.
\item\label{pcf:rule:face2} Each face sends charge $\frac{1}{2}$ to its incident supported $2$-vertices, and $\frac{3}{4}$ to its incident unsupported $2$-vertices. 
\item\label{pcf:rule:face3} Each face sends charge $\frac{1}{2}$ to incident endpoints of a $3$-thread and charge 1 to the middle point.
\end{enumerate}

\end{mdframed}

\vspace{1em}

\begin{figure}[!h]
\centering
\begin{adjustbox}{valign=t}
\begin{tikzpicture}
\draw (0,0) -- (4,0);
\draw[->,red,very thick] (0,0.3) .. controls (0.25, 0.6) and (0.75, 0.6) .. (1,0.3);
\draw[fill=white] (0,0) +(-7pt,-7pt) rectangle +(7pt,7pt);
\draw[fill=black] (1,0) circle (5pt);
\draw[fill=black] (2,0) circle (5pt);
\draw[fill=black] (3,0) circle (5pt);
\draw[fill=white] (4,0) +(-7pt,-7pt) rectangle +(7pt,7pt);
\node at (0,0) {$4^+$};
\node at (4,0) {$4^+$};
\node[red] at (0.5, 0.8) {$1$};
\end{tikzpicture}
\end{adjustbox}
\qquad
\begin{adjustbox}{valign=t}
\begin{tikzpicture}
\draw (0,0) -- (3,0);
\draw[->,red,very thick] (0,0.3) .. controls (0.25, 0.6) and (0.75, 0.6) .. (1,0.3);
\draw[fill=white] (0,0) +(-7pt,-7pt) rectangle +(7pt,7pt);
\draw[fill=black] (1,0) circle (5pt);
\draw[fill=black] (2,0) circle (5pt);
\draw[fill=white] (3,0) +(-7pt,-7pt) rectangle +(7pt,7pt);
\node at (0,0) {$4^+$};
\node at (3,0) {$4^+$};
\node[red] at (0.5, 0.8) {$1$};
\node at (1,-0.5) {supported};
\end{tikzpicture}
\end{adjustbox}
\qquad
\begin{adjustbox}{valign=t}
\begin{tikzpicture}
\draw (0,0) -- (3,0);
\draw[->,red,very thick] (0,0.3) .. controls (0.25, 0.6) and (0.75, 0.6) .. (1,0.3);
\draw[fill=white] (0,0) +(-7pt,-7pt) rectangle +(7pt,7pt);
\draw[fill=black] (1,0) circle (5pt);
\draw[fill=black] (2,0) circle (5pt);
\draw[fill=white] (3,0) +(-7pt,-7pt) rectangle +(7pt,7pt);
\node at (0,0) {$4^+$};
\node at (3,0) {$4^+$};
\node[red] at (0.5, 0.85) {$\frac{1}{2}$};
\node at (1,-0.5) {unsupported};
\end{tikzpicture}
\end{adjustbox}
\caption{Vertex discharging rules \ref{pcf:rule:vertex1}-\ref{pcf:rule:vertex3}.}\label{figure:pcf-vertex-discharge-rules}
\end{figure}
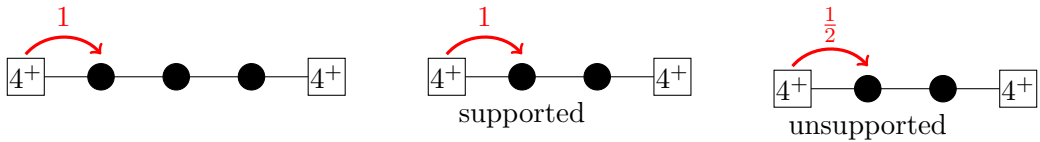

\begin{figure}[tbh!]
    \centering
    \resizebox{.4\textwidth}{!}{\begin{tikzpicture}
\draw (0,0) circle (3);
\draw[fill=white, draw=white] (270+0*36:3) circle (12pt);
\draw[fill=white, draw=white] (270+6.5*36:3) circle (10pt);
\draw[fill=white, draw=white] (270+3.5*36:3) circle (10pt);
% \draw[fill=white, draw=white] (270+5.85*36:3) circle (10pt);
% \draw[fill=white, draw=white] (270+7.55*36:3) circle (10pt);
%%%%%%%%%%%%%%%%%%%%%%%%%%%%%%%%%
%Vertex Placement%
%%%%%%%%%%%%%%%%%%%%%%%%%%%%%%%%%
%Good 1-thread%
% \draw[fill=white] (270+6.3*36:3) +(-7pt,-7pt) rectangle +(7pt,7pt);
% \draw[fill=white] (270+7.1*36:3) +(-7pt,-7pt) rectangle +(7pt,7pt);
% \draw[fill=black] (270+6.7*36:3) circle (5pt);

%Bad 1-thread%
\draw[fill=white] (270+5.4*36:3) +(-7pt,-7pt) rectangle +(7pt,7pt);
\draw[fill=white] (270+4.6*36:3) +(-7pt,-7pt) rectangle +(7pt,7pt);
% \draw[fill=white] (270+4.6*36:3) circle (7pt);
\draw[fill=black] (270+5*36:3) circle (5pt);

%Worst 1-thread%
% \draw[fill=white] (270+3.7*36:3) circle (7pt);
% \draw[fill=white] (270+2.9*36:3) circle (7pt);
% \draw[fill=black] (270+3.3*36:3) circle (5pt);

%2-thread$
\draw[fill=white] (270+2.4*36:3) +(-7pt,-7pt) rectangle +(7pt,7pt);
\draw[fill=white] (270+1.2*36:3) +(-7pt,-7pt) rectangle +(7pt,7pt);
\draw[fill=black] (270+1.6*36:3) circle (5pt);
\draw[fill=black] (270+2*36:3) circle (5pt);

%3-thread%
\draw[fill=white] (270+7.5*36:3) +(-7pt,-7pt) rectangle +(7pt,7pt);
\draw[fill=white] (270+-.9*36:3) +(-7pt,-7pt) rectangle +(7pt,7pt);
\draw[fill=black] (270+7.9*36:3) circle (5pt);
\draw[fill=black] (270+8.3*36:3) circle (5pt);
\draw[fill=black] (270+8.7*36:3) circle (5pt);
%%%%%%%%%%%%%%%%%%%%%%%%%%%%%%%%%%
%Node-Labels$
%%%%%%%%%%%%%%%%%%%%%%%%%%%%%%%%%%

%2-thread%
\node at (270 + 1.2*36:3) {$4^+$};
\node at (270 + 2.4*36:3) {$4^+$};

%3-thread%
\node at (270 + 7.5*36:3) {$4^+$};
\node at (270 + 9.1*36:3) {$4^+$};

%Good 1-thread%
% \node at (270 + 7.1*36:3) {$4^+$};
% \node at (270 + 6.3*36:3) {$4^+$};

%Bad 1-thread%
\node at (270 + 5.4*36:3) {$3^+$};
\node at (270 + 4.6*36:3) {$3^+$};

%Worst 1-thread%
% \node at (270 + 2.9*36:3) {$3$};
% \node at (270 + 3.7*36:3) {$3$};

%SPACES%
\node at (270 + 0*36:3) {$\cdots$};
\node at (270 + 6.5*36:3) {$\cdots$};
\node at (270 + 3.5*36:3) {$\cdots$};
% \node at (270 + 2.5*36:3) {$\vdots$};
% \node at (270 + 4.15*36:3) {$\cdots$};
% \node at (1.525,2.65) {$\ddots$};

%Discharging%
\node at (270 + 0*36:0) {$f$};
\draw[->, red, very thick] (0,.3) to node [midway, fill=white] {$1$} (0,2.75);
% \draw[->, red, very thick] (.2,.3) to node [midway, fill=white] {$\frac{4}{5}$} (2.355,1.355);
% \draw[->, red, very thick] (-.2,.3) to node [midway, fill=white] {$\frac{2}{3}$} (-2.355,1.355);
\draw[->, red, very thick] (-.3,-.2) to node [midway, left, fill=white] {$\frac{1}{2}$} (-2.65,-.7);
\draw[->, red, very thick] (-.3,-.2) to node [midway, fill=white] {$1$} (-2.4,-1.35);
\draw[->, red, very thick] (-.3,-.2) to node [midway, right, fill=white] {$\frac{1}{2}$} (-2,-1.9);

%2-thread%
\draw[->, red, very thick] (.3,-.2) to node [midway, left, fill=white] {$\frac{1}{2}$} (2.3,-1.55);
\draw[->, red, very thick] (.3,-.2) to node [midway, right, fill=white] {$\frac{3}{4}$} (2.6,-.9);

%%Supported%%
\node[font=\scriptsize] at (1,-1.45) {supported};
\node[font=\scriptsize] at (1.65,0) {unsupported};
\end{tikzpicture}}
    \caption{Face discharging rules \ref{pcf:rule:face1} - \ref{pcf:rule:face3}.}
    \label{fig:pcf-face-discharging}
\end{figure}

Let $\mu'$ be the final charge after the discharging rules are applied. In \S \ref{subsection: vertex charge}, we show $\mu'(v) \geq 0$ for all $v \in V(G)$. In \S \ref{subsection: face charge}, we show $\mu'(f) \geq 0$ for all $f \in F(G)$. Then as
\[0 \le \sum_{x \in V(G)\cup F(G)}\mu'(x)=\sum_{x \in V(G)\cup F(G)}\mu(x) = -12,\] 
we obtain a contradiction, completing the proof of \cref{thm:pcf11}.
\end{proof}

\subsection{Every vertex has nonnegative final charge}\label{subsection: vertex charge}

\begin{lemma}\label{lem: pcf-vertices-final-charge}
Each vertex has nonnegative final charge. That is, $\mu'(v) \ge 0$ for all $v \in V(G)$. 
\end{lemma}

\begin{proof}
We argue each vertex has nonnegative final charge via a series of claims. Let $v \in V(G)$. Let $x_2$ and $x_3$ be the number of $t_2$- and $t_3$-neighbors of $v$, respectively. Note these are the only neighbors that receive a charge from $v$.
\begin{claim}
    If $\deg(v) \geq 6$, then $\mu'(v) \geq 0$.
\end{claim}
\begin{claimproof}
    As $v$ sends charge at most 1 to any of its neighbors, it follows $\mu'(v) = (2\deg(v) - 6) - (\deg(v)) = \deg(v) - 6 \geq 0$.
\end{claimproof}

\begin{claim}
    If $\deg(v) = 5$, then $\mu'(v) \geq 0$.
\end{claim}
\begin{claimproof}
    First Suppose $x_3 \geq 1$. Then by \cref{obs:lem}\ref{obs:lem5}, $x_2 + x_3 \leq 4 + n_{3^+}(v)$, and it follows by degree restrictions that $x_2 + x_3 \leq 4$. Since $v$ at most charge 1 to any neighbor, it follows $\mu'(v) \geq (2(5) - 6) - (4(1)) = 0$ by rules \ref{pcf:rule:vertex1} and \ref{pcf:rule:vertex2}.

    Now suppose $x_3 = 0$. If $x_2 \leq 4$, then as $v$ sends at most 1 to each by rule \ref{pcf:rule:vertex2}, it follows $\mu'(v) \geq (2(5) - 6) - (4(1)) \geq 0$. So we may assume $x_2 = 5$. Then $v$ has no $t_1$- or $t_3$-neighbor, and thus none of $v$'s $t_2$-neighbors are supported. Therefore, $v$ sends at most $\frac{1}{2}$ to each of its neighbors by rule \ref{pcf:rule:vertex3}, and thus $\mu'(v) \geq (2(5) - 6) - (\frac{1}{2}(5)) \geq 0$.
\end{claimproof}

\begin{claim}
    If $\deg(v) = 4$, then $\mu'(v) \geq 0$.
\end{claim}
\begin{claimproof}
    First suppose $x_3 \geq 1$. By \cref{obs:lem}\ref{obs:lem4}, $x_2 + x_3 \leq 2$. As $v$ sends at most 1 charge to each $t_2$- or $t_3$-neighbor by rules \ref{pcf:rule:vertex1} and \ref{pcf:rule:vertex2}, it follows $\mu'(v) \leq (2(4) - 6) - (1(2)) = 0$.

    Now suppose $x_3 = 0$. If $x_2 = 4$, then all of $v$'s neighbors are unsupported, and thus $v$ gives charge at most $\frac{1}{2}$ to each by rule \ref{pcf:rule:vertex3}. Therefore, $\mu'(v) \geq (2(4) - 6) - (\frac{1}{2}(4)) = 0$. If $x_2 \leq 2$, then as $v$ sends charge at most 1 to each of these neighbors by rule \ref{pcf:rule:vertex2}, it follows $\mu'(v) \geq (2(4) - 6) - (1(2)) = 0$. Thus assume $x_2 = 3$. If none of these three $2$-neighbors are supported, then $v$ sends at most $\frac{1}{2}$ to each by rule \ref{pcf:rule:vertex3}, and thus $\mu'(v) \geq (2(4) - 6) - (\frac{1}{2}(3)) \geq 0$. Therefore, we may assume at least one of $v$'s $t_2$-neighbors is supported which means $v$ must also have a $t_1$-neighbor. However, this is not possible by \cref{obs: pcf-reduce}.
    \end{claimproof}

\begin{claim}
    If $\deg(v) \leq 3$, then $\mu'(v) \geq 0$.
\end{claim}
\begin{claimproof}
    By \cref{obs:lem}\ref{obs:lem1}, $\deg(v) \geq 2$. If $\deg(v) = 3$, then $v$ doesn't participate in any discharging rules, and thus its charge is $\mu'(v) = 2(3) - 6 = 0$. So we assume $\deg(v) = 2$. Then we argue that $v$ cannot be a cut-vertex. To see this, consider the proof of \cref{cutvertex} also works in the setting of PCF coloring. Therefore, $v$ is incident to two different faces.

    Suppose $v$ is on a $1$-thread. As it receives charge $\frac{1}{2}$ from each of its incident faces by rule \ref{pcf:rule:face1}, $\mu'(v) \geq (2(2) - 6) + (\frac{1}{2}(2)) \geq 0$.

    Suppose $v$ is on a $2$-thread. If $v$ is an unsupported vertex, it receives $\frac{3}{4}$ from each of its incident faces by rule \ref{pcf:rule:face2}, and $\frac{1}{2}$ from its neighboring anchor by rule \ref{pcf:rule:vertex3}. Thus $\mu'(v) \geq (2(2) -6) + (\frac{3}{4}(2) + \frac{1}{2}(1)) = 0$. If $v$ is supported, then it receives $\frac{1}{2}$ from each face by rule \ref{pcf:rule:face2}, and $1$ from its neighboring anchor by rule \ref{pcf:rule:vertex2}, and thus $\mu'(v) \geq (2(2)-6) + (\frac{1}{2}(2) + 1(1)) = 0$.

    Suppose $v$ is on a $3$-thread. If $v$ is the middle vertex, then it receives 1 from each face by rule \ref{pcf:rule:face3}, and so $\mu'(v) \geq (2(2) - 6) + (1(2)) = 0$. If $v$ is an endpoint, it receives $\frac{1}{2}$ from each face by rule \ref{pcf:rule:face3} and $1$ from its neighboring anchor by rule \ref{pcf:rule:vertex1}, and thus $\mu'(v) \geq (2(2) - 6) + (\frac{1}{2}(2) + 1(1)) = 0$.

Since there are no $4^+$-threads (\cref{obs:lem}\ref{obs:lem3}), this covers all cases and completes the claim.
\end{claimproof}
Thus, in all cases, $\mu'(v)\geq 0$. This completes the proof of \cref{lem: pcf-vertices-final-charge}.
\end{proof}

\subsection{Every face has nonnegative final charge}\label{subsection: face charge}

In the following proof, we consider the average charge a face gives to each vertex in an array based on rules \ref{pcf:rule:face1} - \ref{pcf:rule:face3}. Note $a_3$ receives a different average charge depending on the number of supported 2-vertices ($\frac{1}{2}$ if both are unsupported, $\frac{5}{12}$ if exactly one is supported, and $\frac{1}{3}$ if both are supported). See \cref{fig:arrays-pcf}.

\begin{figure}[!ht]
    \centering
    \bgroup
    \def\arraystretch{1.5}
    \begin{tabular}{l|l|c}
       Array  & Degree Sequence & $\substack{\text{Average Charge Received}\\ \text{per Vertex}}$\\
       \hline
        $a_4$ & $\boxed{4^+ - 2 - 2 - 2} - 4^+$ & $\frac{1}{2}$ \\
        $a_3$ & $\boxed{4^+ - 2 - 2} - 4^+$ & $\frac{1}{2}$, $\frac{5}{12}$, or $\frac{1}{3}$ \\
        $a_2$ & $\boxed{3^+ - 2} - 3^+$ & $\frac{1}{2}$\\
        $a_1$ & $\boxed{3^+} - 3^+$ & $0$
    \end{tabular}
    \egroup
    \caption{Average charge received per vertex from a face.}
    \label{fig:arrays-pcf}
\end{figure}

\begin{lemma}\label{lem: pcf-faces-final-charge}
Each face has nonnegative final charge. That is $\mu'(f) \ge 0$ for all $f \in F(G)$.    
\end{lemma}
\begin{proof}

We argue each face has nonnegative final charge via a series of claims. Since $G$ has girth 11, $\ell(f)\geq 11$. Recall that the initial charge of a face is $\mu(f) = \ell(f) - 6$. Also recall by \cref{lem: all faces have array representations}, every vertex of $f$ is contained in some array.

\begin{claim}
    If $\ell(f)\geq 12, \mu'(f)\geq 0$.
\end{claim}

\begin{claimproof}
    Suppose $f$ is a $12^+$-face. Since the maximum average charge given to a vertex is $\frac{1}{2}$, we have $\mu'(f) \geq (\ell(f) - 6) - (\frac{1}{2}\ell(f))\geq 0$.
\end{claimproof}

\begin{claim}
    If $\ell(f) = 11, \mu'(f)\geq 0$.
\end{claim}

\begin{claimproof}
    Suppose $f$ is an $11$-face. Fix an array representation of $f$.
    Let $x_i$ denote the number of $a_i$ arrays in an array representation of $f$ for $i\in[4]$, and let $s$ denote the number of supported vertices incident to $f$.

    Suppose $s\geq 2$. If two supported vertices are adjacent, then $f$ contains an $a_3$, which receives an average of $\frac{1}{3}$ per vertex. Since all the other arrays of $f$'s boundary receive at most $\frac{1}{2}$ average charge, it follows $\mu'(f) \geq (11 - 6) - (\frac{1}{2}(8) + \frac{1}{3}(3)) = 0$. If two supported vertices are not adjacent, then $f$ contains at least two $a_3$ arrays which receive an average charge of at most $\frac{5}{12}$. Thus $\mu'(f) \geq (11 - 6) - (\frac{1}{2}(5) + \frac{5}{12}(6)) = 0$. Therefore we may assume $f$ contains at most one supported vertex, i.e., $s\leq 1$.

    Suppose $x_1 \geq 1$. Since $a_1$ receives an average charge of $0$, and the other arrays on $f$'s boundary receive an average charge of at most $\frac{1}{2}$, it follows $\mu'(f)\geq (11-6) - (\frac{1}{2}(7)+\frac{5}{12}(3))\geq 0$. So we assume $x_1 = 0$. We now split into cases based on the value of $x_4$. Note $x_4\leq 2$ since $\ell(f)=11$.

    \emph{Case 1: $x_4 = 2$.} If the array representation of $f$ contains $a_4$ twice, but no $a_1$, then the array representation of $f$ must be a permutation of $a_4,a_4,a_3$. However, in this case, both $2$-vertices represented by the $a_3$ are supported since both are neighbors to anchors of $3$-threads. This contradicts the assumption that $s\leq 1$. Thus $x_4\neq 2$.

    \emph{Case 2: $x_4 = 1$.} Since $\ell(f) = 11$ and the array representation of $f$ contains one $a_4$, but no $a_1$, by parity arguments, $x_3 = 1$. However, if $x_3 = 1$, then the array representation of $f$ is some permutation of $a_4,a_3,a_2,a_2$, and regardless of the permutation, both $2$-vertices of the $a_3$ are supported since both are neighbors to anchors of $1$- or $3$-thread. This contradicts the assumption that $s\leq 1$. Thus $x_4\neq 1$.
    
    \emph{Case 3: $x_4 = 0$.} Since $\ell(f) = 11$ and the array representation of $f$ contains neither $a_4$ nor $a_1$, by parity arguments, $x_3\in \{1,3\}$. If $x_3 = 1$ then the array representation of $f$ is some permutation of $a_3,a_2,a_2,a_2,a_2$, and regardless of the permutation, both $2$-vertices of the $a_3$ are supported since both are neighbors to anchors of $1$-threads. This contradicts $s\leq 1$ and hence we may assume that $x_3 = 3$. In this case, the array representation of $f$ is some permutation of $a_3,a_3,a_3,a_2$. Since $f$ contains exactly one $1$-thread, both of its anchors are also anchors of distinct $2$-threads. Thus, both of the $2$-threads have at least one supported vertex each, contradicting the assumption that $s\leq 1$. Thus, $x_4\neq 0$. This covers all cases and completes the claim.
\end{claimproof}
By the two claims above, it follows $\mu'(f) \geq 0$ for all $f \in F(G)$.
\end{proof}

\section{Acknowledgements}
This work was initiated at the 2023 Graduate Research Workshop in Combinatorics, which was supported in
part by NSF grant \#1953985 and a generous award from the Combinatorics Foundation.
We are also grateful to the University of Wyoming for their hospitality during this workshop. 
The work of Eun-Kyung Cho was supported by Basic Science Research Program through the National Research Foundation of Korea (NRF) funded by the Ministry of Education (No. RS-2023-00244543).
The work of Emily Heath was partially supported by NSF RTG Grant DMS-1839918.
The work of Hyemin Kwon was supported by the framework of international cooperation program managed by the National Research Foundation of Korea (NRF-2023K2A9A2A06059347).
The work of Zhiyuan Zhang was supported by the Ontario Graduate Scholarship (OGS) Program and partially supported by the NSERC under Discovery Grant (No. 2019-04269). 
We also thank the developers of Sagemath~\cite{sagemath}, which we used for preliminary explorations.
\appendix

% \bibliographystyle{abbrv}
% \bibliography{references}
\printbibliography

@misc{cho2023brooks,
      title={Brooks-type theorems for relaxations of square colorings}, 
      author={Eun-Kyung Cho and Ilkyoo Choi and Hyemin Kwon and Boram Park},
      year={2023},
      eprint={2302.06125},
      archivePrefix={arXiv},
      primaryClass={math.CO}
}

@book{west2001introduction,
    AUTHOR = {West, Douglas B.},
     TITLE = {Introduction to graph theory},
 PUBLISHER = {Prentice Hall, Inc., Upper Saddle River, NJ},
      YEAR = {1996},
     PAGES = {xvi+512},
      ISBN = {0-13-227828-6},
   MRCLASS = {05-01},
  MRNUMBER = {1367739},
}

@misc{cho2022relaxation, 
      title={Relaxation of Wegner's Planar Graph Conjecture for maximum degree 4}, 
      author={Eun-Kyung Cho and Ilkyoo Choi and Bernard Lidický},
      year={2022},
      eprint={2212.10643},
      archivePrefix={arXiv},
      primaryClass={math.CO}
}

@misc{liu2024asymptotically,
      title={Asymptotically Optimal Proper Conflict-Free Colouring}, 
      author={Chun-Hung Liu and Bruce Reed},
      year={2024},
      eprint={2401.02155},
      archivePrefix={arXiv},
      primaryClass={math.CO}
}

@article{kamyczura2024conflict,
    AUTHOR = {Kamyczura, Mateusz and Przyby\l o, Jakub},
     TITLE = {On conflict-free proper colourings of graphs without small
              degree vertices},
   JOURNAL = {Discrete Math.},
  FJOURNAL = {Discrete Mathematics},
    VOLUME = {347},
      YEAR = {2024},
    NUMBER = {1},
     PAGES = {Paper No. 113712, 6},
      ISSN = {0012-365X,1872-681X},
   MRCLASS = {05C15},
  MRNUMBER = {4649931},
       DOI = {\url{10.1016/j.disc.2023.113712}},
       URL = {https://doi.org/10.1016/j.disc.2023.113712},
}

@misc{dai2023new,
      title={New bounds for odd colourings of graphs}, 
      author={Tianjiao Dai and Qiancheng Ouyang and François Pirot},
      year={2023},
      eprint={2306.01341},
      archivePrefix={arXiv},
      primaryClass={math.CO}
}

@article{liu2024proper,
    AUTHOR = {Liu, Chun-Hung},
     TITLE = {Proper conflict-free list-coloring, odd minors, subdivisions,
              and layered treewidth},
   JOURNAL = {Discrete Math.},
  FJOURNAL = {Discrete Mathematics},
    VOLUME = {347},
      YEAR = {2024},
    NUMBER = {1},
     PAGES = {Paper No. 113668, 16},
      ISSN = {0012-365X,1872-681X},
   MRCLASS = {05C15 (05C83 05C85)},
  MRNUMBER = {4634919},
       DOI = {10.1016/j.disc.2023.113668},
       URL = {https://doi.org/10.1016/j.disc.2023.113668},
}

@misc{cho2022proper,
      title={Proper conflict-free coloring of sparse graphs}, 
      author={Eun-Kyung Cho and Ilkyoo Choi and Hyemin Kwon and Boram Park},
      year={2022},
      eprint={2203.16390},
      archivePrefix={arXiv},
      primaryClass={math.CO}
}

@article{cho2023odd,
    AUTHOR = {Cho, Eun-Kyung and Choi, Ilkyoo and Kwon, Hyemin and Park, Boram},
     TITLE = {Odd coloring of sparse graphs and planar graphs},
   JOURNAL = {Discrete Math.},
  FJOURNAL = {Discrete Mathematics},
    VOLUME = {346},
      YEAR = {2023},
    NUMBER = {5},
     PAGES = {Paper No. 113305, 7},
      ISSN = {0012-365X,1872-681X},
   MRCLASS = {05C15 (05C10)},
  MRNUMBER = {4533825},
MRREVIEWER = {Anna\ O.\ Ivanova},
       DOI = {10.1016/j.disc.2022.113305},
       URL = {https://doi.org/10.1016/j.disc.2022.113305},
}

@article{petrusevski2021colorings,
    AUTHOR = {Petru\v{s}evski, Mirko and \v{S}krekovski, Riste},
     TITLE = {Colorings with neighborhood parity condition},
   JOURNAL = {Discrete Appl. Math.},
  FJOURNAL = {Discrete Applied Mathematics. The Journal of Combinatorial
              Algorithms, Informatics and Computational Sciences},
    VOLUME = {321},
      YEAR = {2022},
     PAGES = {385--391},
      ISSN = {0166-218X,1872-6771},
   MRCLASS = {05C15 (05C10)},
  MRNUMBER = {4467654},
MRREVIEWER = {John\ A.\ Engbers},
       DOI = {10.1016/j.dam.2022.07.018},
       URL = {https://doi.org/10.1016/j.dam.2022.07.018},
}

@article{caro2022remarks,
    AUTHOR = {Caro, Yair and Petru\v{s}evski, Mirko and \v{S}krekovski,
              Riste},
     TITLE = {Remarks on odd colorings of graphs},
   JOURNAL = {Discrete Appl. Math.},
  FJOURNAL = {Discrete Applied Mathematics. The Journal of Combinatorial
              Algorithms, Informatics and Computational Sciences},
    VOLUME = {321},
      YEAR = {2022},
     PAGES = {392--401},
      ISSN = {0166-218X,1872-6771},
   MRCLASS = {05C15},
  MRNUMBER = {4467655},
MRREVIEWER = {Hung-Lin\ Fu},
       DOI = {10.1016/j.dam.2022.07.024},
       URL = {https://doi.org/10.1016/j.dam.2022.07.024},
}

@article{caro2023remarks,
    AUTHOR = {Caro, Yair and Petru\v{s}evski, Mirko and \v{S}krekovski,
              Riste},
     TITLE = {Remarks on proper conflict-free colorings of graphs},
   JOURNAL = {Discrete Math.},
  FJOURNAL = {Discrete Mathematics},
    VOLUME = {346},
      YEAR = {2023},
    NUMBER = {2},
     PAGES = {Paper No. 113221, 14},
      ISSN = {0012-365X,1872-681X},
   MRCLASS = {05C15},
  MRNUMBER = {4499342},
MRREVIEWER = {Anna\ O.\ Ivanova},
       DOI = {10.1016/j.disc.2022.113221},
       URL = {https://doi.org/10.1016/j.disc.2022.113221},
}

@article{petr2022odd,
   title={The Odd Chromatic Number of a Planar Graph is at Most 8},
   volume={39},
   ISSN={1435-5914},
   url={http://dx.doi.org/10.1007/s00373-023-02617-z},
   DOI={10.1007/s00373-023-02617-z},
   number={2},
   journal={Graphs and Combinatorics},
   publisher={Springer Science and Business Media LLC},
   author={Petr, Jan and Portier, Julien},
   year={2023},
   month=mar }

@article{fabrici2022proper,
title = {Proper conflict-free and unique-maximum colorings of planar graphs with respect to neighborhoods},
journal = {Discrete Applied Mathematics},
volume = {324},
pages = {80-92},
year = {2023},
issn = {0166-218X},
doi = {https://doi.org/10.1016/j.dam.2022.09.011},
url = {https://www.sciencedirect.com/science/article/pii/S0166218X22003584},
author = {Igor Fabrici and Borut Lužar and Simona Rindošová and Roman Soták}
}

@misc{wang2022odd,
      title={On odd colorings of sparse graphs}, 
      author={Tao Wang and Xiaojing Yang},
      year={2022},
      eprint={2212.06563},
      archivePrefix={arXiv},
      primaryClass={math.CO}
}

@misc{ahn2022proper,
      title={The proper conflict-free $k$-coloring problem and the odd $k$-coloring problem are NP-complete on bipartite graphs}, 
      author={Jungho Ahn and Seonghyuk Im and Sang-il Oum},
      year={2022},
      eprint={2208.08330},
      archivePrefix={arXiv},
      primaryClass={cs.CC}
}

@article{liu20231,
    AUTHOR = {Liu, Runrun and Wang, Weifan and Yu, Gexin},
     TITLE = {1-planar graphs are odd 13-colorable},
   JOURNAL = {Discrete Math.},
  FJOURNAL = {Discrete Mathematics},
    VOLUME = {346},
      YEAR = {2023},
    NUMBER = {8},
     PAGES = {Paper No. 113423, 7},
      ISSN = {0012-365X,1872-681X},
   MRCLASS = {05C15 (05C10)},
  MRNUMBER = {4564916},
       DOI = {10.1016/j.disc.2023.113423},
       URL = {https://doi.org/10.1016/j.disc.2023.113423},
}

@misc{cranston2022odd,
      title={Odd Colorings of Sparse Graphs}, 
      author={Daniel W. Cranston},
      year={2022},
      eprint={2201.01455},
      archivePrefix={arXiv},
      primaryClass={math.CO}
}

@article{cranston2023note,
    AUTHOR = {Cranston, Daniel W. and Lafferty, Michael and Song, Zi-Xia},
     TITLE = {A note on odd colorings of 1-planar graphs},
   JOURNAL = {Discrete Appl. Math.},
  FJOURNAL = {Discrete Applied Mathematics. The Journal of Combinatorial
              Algorithms, Informatics and Computational Sciences},
    VOLUME = {330},
      YEAR = {2023},
     PAGES = {112--117},
      ISSN = {0166-218X,1872-6771},
   MRCLASS = {05C15 (05C10)},
  MRNUMBER = {4537616},
       DOI = {10.1016/j.dam.2023.01.011},
       URL = {https://doi.org/10.1016/j.dam.2023.01.011},
}

@misc{qi2022odd,
      title={Odd coloring of two subclasses of planar graphs}, 
      author={Mengke Qi and Xin Zhang},
      year={2022},
      eprint={2205.09317},
      archivePrefix={arXiv},
      primaryClass={math.CO}
}

@manual{sagemath,
  Key          = {SageMath},
  Author       = {The {Sage Developers}},
  Title        = {{S}ageMath, the {S}age {M}athematics {S}oftware {S}ystem ({V}ersion 10.0)},
  note         = {{\tt https://www.sagemath.org}},
  Year         = {2023},
}
\end{document}